\pdfoutput=1

\documentclass[11pt]{article}
\usepackage[margin=1in]{geometry}
\usepackage{amsmath}
\usepackage{enumitem}
\usepackage{amssymb}
\usepackage{amsthm}
\usepackage{multirow}
\usepackage{float}
\usepackage{setspace}
\usepackage{microtype}
\usepackage{subcaption}
\usepackage{titlesec}
\usepackage{pgfplots}
\usepackage[longnamesfirst]{natbib}
\usepackage[nottoc]{tocbibind}
\renewcommand\bibname{References}
\usepackage[hypertexnames=false]{hyperref}
\usepackage[hang,flushmargin]{footmisc}
\usepackage{stackengine,scalerel}
\usepackage[norefs,nocites]{refcheck}
\usepackage{etoolbox}

\usepackage[ruled,
  linesnumbered,
  commentsnumbered,
  procnumbered,
]{algorithm2e}

\hypersetup{pdfborder = {0 0 0},colorlinks=true,
linkcolor=blue,urlcolor=blue,citecolor=blue}
\pgfplotsset{compat=1.9}
\def\arraystretch{1.3}
\newtoggle{journal}
\togglefalse{journal}

\renewcommand{\P}{\ensuremath{\mathbb{P}}}
\newcommand{\E}{\ensuremath{\mathbb{E}}}
\newcommand{\N}{\ensuremath{\mathbb{N}}}
\newcommand{\bX}{\ensuremath{\mathbf{X}}}
\newcommand{\bY}{\ensuremath{\mathbf{Y}}}
\newcommand{\bT}{\ensuremath{\mathbf{T}}}
\newcommand{\bP}{\ensuremath{\mathbf{P}}}
\newcommand{\cN}{\ensuremath{\mathcal{N}}}
\newcommand{\cH}{\ensuremath{\mathcal{H}}}
\newcommand{\cF}{\ensuremath{\mathcal{F}}}
\newcommand{\cD}{\ensuremath{\mathcal{D}}}
\newcommand{\cA}{\ensuremath{\mathcal{A}}}
\newcommand{\cM}{\ensuremath{{\mathcal{M}}}}
\newcommand{\cX}{\ensuremath{\mathcal{X}}}
\newcommand{\cR}{\ensuremath{\mathcal{R}}}
\newcommand{\rd}{\ensuremath{\mathrm{d}}}
\renewcommand{\c}{\ensuremath{\mathrm{c}}}
\newcommand{\I}{\ensuremath{\mathbb{I}}}
\newcommand{\R}{\ensuremath{\mathbb{R}}}
\newcommand{\T}{\ensuremath{\mathsf{T}}}
\newcommand{\GCV}{\ensuremath{\mathrm{GCV}}}
\newcommand{\LOOCV}{\ensuremath{\mathrm{LOOCV}}}
\newcommand{\CI}{\ensuremath{\mathrm{CI}}}
\newcommand{\bigsetminus}{\mathbin{\big\backslash}}
\newcommand{\Bigsetminus}{\mathbin{\Big\backslash}}
\newcommand{\LM}{{\mathrm{LM}}}
\newcommand{\SD}{{\mathrm{SD}}}
\newcommand{\ABias}{{\mathrm{ABias}}}
\newcommand{\ASD}{{\mathrm{ASD}}}
\newcommand{\AVar}{{\mathrm{AVar}}}
\newcommand{\betaf}{{\beta_{f}}}
\newcommand{\betamu}{{\beta_{\mu}}}
\newcommand{\betasigma}{{\beta_{\sigma}}}

\DeclareMathOperator{\AMSE}{AMSE}
\DeclareMathOperator{\argmin}{argmin}
\DeclareMathOperator{\Bias}{Bias}
\DeclareMathOperator{\Bin}{Bin}
\DeclareMathOperator{\Cov}{Cov}
\DeclareMathOperator{\Gam}{Gamma}
\DeclareMathOperator{\Exp}{Exp}
\DeclareMathOperator{\Leb}{Leb}
\DeclareMathOperator{\MSE}{MSE}
\DeclareMathOperator{\NW}{NW}
\DeclareMathOperator{\Unif}{Unif}
\DeclareMathOperator{\Var}{Var}

\newcommand{\flbeta}{{\ThisStyle{%
      \ensurestackMath{\stackengine{-0.5\LMpt}{\SavedStyle \beta}%
        {\SavedStyle {\rule{3.7\LMpt}{0.3\LMpt}}}
{U}{c}{F}{F}{S}}\vphantom{\beta}}}}

\newcommand{\diffi}{\,\mathrm{d}}
\newcommand{\diffd}{\mathrm{d}}
\newcommand{\TODO}[1]{\textcolor{red}{\textup{\textsc{TODO}: #1}}}
 \newcommand{\SDhattab}{{\widehat{\textrm{SD}}}}
\renewcommand{\betaf}{{\beta_{\!\!\:f}}}

\titleformat{\paragraph}[hang]{\bfseries\upshape}{}{0pt}{}[]
\titlespacing*{\paragraph}{0pt}{6pt}{0pt} \newcounter{proofparagraphcounter}
\newcommand{\proofparagraph}[1]{ \refstepcounter{proofparagraphcounter}%
\paragraph{Part \theproofparagraphcounter : #1}}%

\SetKwInOut{Input}{Input}%
\SetKwInOut{Output}{Output}%

\newtheoremstyle{break}%
{\topsep}{\topsep}%
{\itshape}{}%
{\bfseries}{\textbf{.}}%
{.4em}%
{\thmname{#1}\thmnumber{ #2}%
\thmnote{ \rm(#3)\addcontentsline{toc}{subsubsection}{{\it#1 #2}\/: #3}}}

\newtheoremstyle{breakplainproof}
{\topsep}{\topsep}%
{}{}%
{\bfseries}{\textbf{.}}%
{.4em}%
{\thmname{#1}\thmnumber{ #2}%
\thmnote{ \rm(#3)\addcontentsline{toc}{subsubsection}{{\it#1}\/: #3}}}

\theoremstyle{break}
\newtheorem{theorem}{Theorem}
\newtheorem{lemma}{Lemma}
\newtheorem{assumption}{Assumption}

\newtheorem{definition}{Definition}

\theoremstyle{breakplainproof}
\let\proof\relax
\newtheorem*{proof}{Proof}
\AtBeginEnvironment{proof}{\setcounter{proofparagraphcounter}{0}}%
\AtEndEnvironment{proof}{\null\hfill\qedsymbol}%

\theoremstyle{remark}

\makeatletter%
\begin{document}

\title{
  Inference with Mondrian Random Forests
}

\author{
  Matias D.\ Cattaneo\textsuperscript{1,$\ast$}
  \and
  Jason M.\ Klusowski\textsuperscript{1}
  \and
  William G.\ Underwood\textsuperscript{2}
}

\maketitle

\footnotetext[1]{
  Department of Operations Research
  and Financial Engineering,
  Princeton University
}
\footnotetext[2]{
  Statistical Laboratory,
  University of Cambridge
}
\let\thefootnote\relax
\footnotetext[1]{
  \textsuperscript{*}Corresponding author:
  \href{mailto:cattaneo@princeton.edu}{\texttt{cattaneo@princeton.edu}}
}
\addtocounter{footnote}{-1}\let\thefootnote\svthefootnote

\setcounter{page}{0}\thispagestyle{empty}

\begin{abstract}
Random forests are popular methods for regression and classification analysis,
and many different variants have been proposed in recent years. One interesting
example is the Mondrian random forest, in which the underlying constituent
trees are constructed via a Mondrian process. We give precise bias and variance
characterizations, along with a Berry--Esseen-type central limit theorem, for
the Mondrian random forest regression estimator. By combining these results
with a carefully crafted debiasing approach and an accurate variance estimator,
we present valid statistical inference methods for the unknown regression
function. These methods come with explicitly characterized error bounds in
terms of the sample size, tree complexity parameter, and number of trees in the
forest, and include coverage error rates for feasible confidence interval
estimators. Our debiasing procedure for the Mondrian random forest also
allows it to achieve the minimax-optimal point estimation convergence rate in
mean squared error for multivariate $\beta$-H{\"o}lder regression functions,
for all $\beta > 0$, provided that the underlying tuning parameters are chosen
appropriately. Efficient and implementable algorithms are devised for both
batch and online learning settings, and we study the computational
complexity of different Mondrian random forest implementations. Finally,
simulations with synthetic data validate our theory and methodology,
demonstrating their excellent finite-sample properties.
 \end{abstract}


\clearpage
\maketitle
\tableofcontents
\clearpage
\pagebreak

\clearpage

\section{Introduction}

Random forests, first introduced by \citet{breiman2001random},
are a workhorse in modern machine learning
for regression and classification tasks.
Their desirable traits include computational efficiency
in big data settings (via parallelization and greedy heuristics), simplicity of
configuration and
amenability to tuning parameter selection, ability to
adapt to latent structure in high-dimensional data sets, and flexibility in
handling mixed data types, among other virtues.
Random forests have also achieved great empirical successes in many
fields of study, including
healthcare, finance, online commerce, causal inference,
text analysis, bioinformatics, image classification, and ecology.

Since Breiman introduced random forests over twenty years ago, the study of
their statistical properties remains an active area of research. Many
fundamental questions about Breiman's random forests
remain unanswered, owing in part to the subtle ingredients
present in the estimation procedure which
make standard analytical tools ineffective.
These technical difficulties stem from the way the constituent trees greedily
partition the covariate space, utilizing both the covariate and response
data. This creates complicated dependencies on the data that are often
exceedingly hard to untangle without overly stringent assumptions, thereby
hampering theoretical progress.

To address the aforementioned technical challenges while retaining the
phenomenology of Breiman's random forests, a variety of stylized
versions of random forest procedures have been proposed and
studied in the literature.
Early proposals include centered random forests
\citep{biau2012analysis, arnould2023interpolation}
and median random forests
\citep{duroux2018impact, arnould2023interpolation}.
Each tree in a centered random forest is constructed by first choosing a
covariate uniformly at random and then splitting the cell at the midpoint along
the direction of the chosen covariate. Median random forests operate in a
similar way, but involve the covariate data by splitting at the empirical
median along the direction of the randomly chosen covariate.
Known as purely random forests, these procedures simplify Breiman's
original, more data-adaptive version by growing trees that partition
the covariate space in a way that is statistically independent
of the response data.

Yet another variant of random forests, Mondrian random forests
\citep{lakshminarayanan2014mondrian},
have received significant attention from the statistics
and machine learning communities
in recent years
\citep{ma2020isolation, mourtada2020minimax, mourtada2021amf,%
  scillitoe2021uncertainty, vicuna2021reducing, gao2022towards,%
  oreilly2022stochastic, baptista2024trim, oreilly2024minimax,%
oreilly2024statistical, zhan2024non, osborne2025uniformly}.
Like other purely random forest variants,
Mondrian random forests offer a simplified modification of
Breiman's original proposal
in which the partition is generated independently of the data
and according to a canonical stochastic process known as
the Mondrian process \citep{roy2008mondrian}.
The Mondrian process takes a single tuning parameter $\lambda > 0$
known as the ``lifetime''
and enjoys various mathematical properties.
These properties allow Mondrian random forests to be
fitted in an online manner \citep{lakshminarayanan2014mondrian,mourtada2021amf}
and permit a rigorous statistical analysis,
while also retaining some of the appealing features
of other random forest methods. The lifetime parameter $\lambda$,
in analogy with the number of refinements of a
data-adaptive recursive partitioning algorithm,
governs the extent to which the response data is smoothed, with a large
$\lambda$ resulting
in a more complicated partition and therefore less smoothing.
However, unlike a data-adaptive partition which modulates the smoothing
intensity based on local data characteristics, the Mondrian process
applies a uniform smoothing effect globally across all covariates.
Thus, one can draw parallels between (axis-aligned) data-dependent
partitioning schemes and more flexible---albeit, more involved---versions of
the Mondrian process which employ adaptive directional smoothing,
with unique lifetimes $\lambda_j$ learned
for each covariate, permitting tailored smoothing of the responses.

This paper studies the theoretical
statistical properties of Mondrian random forests with
an emphasis on \emph{inference} techniques specific to this procedure.
Although Mondrian random forests are not state-of-the-art for
high-dimensional problems in practice,
we focus on this purely random forest variant not only because of its
importance in the development of random forest theory in general,
but also because the Mondrian process
(along with generalizations including the oblique Mondrian process
and stationary random tessellation processes) is, to date,
the only known randomized recursive tree mechanism
for which the resulting random forest is
minimax-optimal for point estimation over a class of
smooth multivariate regression functions, without requiring sample splitting
\citep{mourtada2020minimax,oreilly2024minimax,oreilly2024statistical}.
In fact, when the covariate dimension exceeds one,
the aforementioned centered and median random forests are both minimax
\emph{suboptimal}, due to their large biases, over the class of Lipschitz
smooth regression functions \citep{klusowski2021sharp}.
It is therefore natural to
focus our study of inference for random forests on versions that at the very
least exhibit competitive bias and variance, as this will have important
implications for the trade-off between confidence and precision.
Moreover, recent studies of generalized random forests
and distributional random forests
identify similar fundamental challenges, namely
those of establishing Gaussian approximations
and combining them with strategies for bias reduction \citep{naf2023confidence}.

Despite their recent popularity, relatively little is
known about the formal
statistical properties of Mondrian random forests.
Focusing on nonparametric regression, \citet{mourtada2020minimax}
recently showed that Mondrian forests containing just a single tree
(called a Mondrian tree) can be minimax-optimal in integrated mean squared
error
whenever the regression function is
$\beta$-H{\"o}lder continuous for some $\beta \in (0, 1]$.
The authors also showed that, when appropriately tuned,
large Mondrian random forests can be similarly
minimax-optimal for $\beta \in (0, 2]$, while the constituent trees cannot.
See also \citet{oreilly2022stochastic} for analogous results on more general
Mondrian tree and forest constructions.
These results formally demonstrate the value of ensembling with random forests
from a point estimation perspective.
No results are currently available in the literature for
statistical inference using Mondrian random forests.

As already mentioned, a different strand of the literature studies the
statistical properties of Breiman's
random forests which form ensembles of
\emph{adaptive} decision trees. In such models,
each constituent tree is constructed with a greedy algorithm that
recursively optimizes a goodness-of-fit metric (such as mean squared error)
using both the covariates and response data;
a leading example in practice is the celebrated
Classification and Regression Tree (CART)
methodology \citep{breiman1984, breiman2001random}.
The underlying complexity of the resulting
procedures make their formal theoretical analysis quite difficult, and
therefore only a more restricted set of results is currently available in the
literature. In terms of estimation theory, \citet{Scornet-Biau-Vert_2015_AOS}
established consistency of
adaptive random forests for additive models with a fixed number of covariates,
and \citet{Chi-Vossler-Fan-Lv_2022_AOS}, \citet{klusowski2021universal}, and
\citet{cattaneo2024convergence} provided rates of convergence for
models with a growing number of covariates, under
different assumptions on the statistical and
algorithmic features of the constituent decision trees.
A framework for tuning tree depths via data-adaptive early stopping was
developed by \citet{miftachov2025early}.
In contrast, formal
inference theory is far less developed, since there are arguably no
satisfactory
theoretical results for fully adaptive decision
tree or random forest methods. For example, \citet{wager2018estimation} provide
asymptotic estimation and
inference results for adaptive random forests, but
they employ sample splitting
(i.e., the so-called ``honesty'' property where the partitioning and,
  separately, the output in the terminal cells are
formed using independent subsamples), and make assumptions
that rule out procedures commonly used in practice such as CART and
other sum-of-squares based splitting criteria;
cf., the so-called ``$\alpha$-regularity'' condition
\citep{cattaneo2024pointwise, cattaneo2025honest}. From a broad perspective,
our paper connects with this distinct thread in the
random forest literature by demonstrating optimal estimation
and inference results for non-adaptive Mondrian random forests with explicit
probability deviation guarantees. Furthermore, by providing a clean
theoretical framework, our results may inform and guide the
development of more adaptive Mondrian random forests methods, a
direction we return to briefly in the conclusion of this paper.

\subsection{Contributions}

Our paper contributes to the literature on
the foundational statistical properties
of Mondrian random forest regression estimation with two main results.
Firstly, we give a central limit theorem
for the classical Mondrian random forest point estimator
under weak conditions,
and propose valid large-sample inference procedures employing a consistent
standard error estimator.
We establish these results by
deploying a restricted moments version of the Berry--Esseen theorem for
independent but not identically distributed (i.n.i.d.)
random variables \citep[Theorem~5.7]{Petrov_1995_Book}
because we need to handle delicate
probabilistic features of the Mondrian random forest estimator.
In particular, we deal with the existence of
Mondrian cells which are ``too small''
and lead to a reduced effective (local) sample size
for some trees in the forest.
Such pathological cells are in fact typical in Mondrian random forests and
complicate the probability limits of certain sample averages; in fact, small
Mondrian random forests (or indeed single Mondrian trees) remain random even
in the limit due to the lack of ensembling. The presence of such small cells
renders inapplicable prior distributional approximation results for
partitioning-based estimators in the literature
\citep{huang2003local,cattaneo2020large}, since the commonly required
quasi-uniformity assumption on the underlying partitioning scheme (cf.,
$\alpha$-regularity in the adaptive random forest literature) is violated
by partitions generated using the Mondrian process. We circumvent this
technical challenge by establishing new theoretical results for Mondrian
partitions and their associated Mondrian trees and forests, which may be of
independent interest.
Our distributional approximation does not rely on sample splitting;
unlike approaches based on ``honest'' trees,
the estimator is fit using
the entire sample of covariates and responses simultaneously.

The second main contribution of our paper is to propose a debiasing approach
for the Mondrian random forest point estimator.
We accomplish this by first precisely characterizing the probability limit
of the large sample conditional bias, and then applying a debiasing procedure
based on the generalized jackknife \citep{schucany1977improvement}.
We thus exhibit a Mondrian random forest variant which is minimax-optimal
in pointwise mean squared error when the
regression function is $\beta$-H{\"o}lder for any $\beta > 0$.
Our method works by generating an ensemble of Mondrian random forests
carefully chosen to have smaller misspecification bias
when extra smoothness is available,
resulting in minimax optimality even for $\beta > 2$.
This result complements \citet{mourtada2020minimax} by demonstrating the
existence of a class of Mondrian random forests that can efficiently exploit
the additional smoothness of the unknown regression function for
minimax-optimal point estimation.
Our proposed debiasing procedure is also useful when conducting
statistical inference because it provides a principled method
for ensuring that the bias is negligible relative
to the standard deviation of the estimator.
More specifically, we use our debiasing approach to construct
valid confidence intervals based on robust bias correction
\citep{calonico2018jasa,calonico2022bernoulli},
and include an explicit bound on their coverage error probability.

For the purposes of implementation, we propose techniques for
tuning parameter selection and demonstrate the practical applicability
and accuracy of our methodology through empirical studies with simulated data.
We also discuss applications to batch and online learning settings,
presenting computationally efficient algorithms
along with bounds for their average case time complexity.

\subsection{Organization}

Section~\ref{sec:setup} gives the
assumptions on the data generating process,
using a H{\"o}lder smoothness condition on the regression function
to control the bias of various estimators.
We also introduce the Mondrian process
and use it to define the Mondrian random forest
estimator, stating the assumptions
on its lifetime parameter and the number of trees.

Section~\ref{sec:inference} presents our first set of main results.
We begin by precisely characterizing the bias of
the Mondrian random forest estimator in Lemma~\ref{lem:bias}, with
the aim of subsequently applying a debiasing procedure.
We similarly analyze the variance of this estimator
(Lemma~\ref{lem:variance}), and deduce its rate of convergence in
Theorem~\ref{thm:rate}.
Next, we present our Berry--Esseen-type
central limit theorem for the centered Mondrian
random forest estimator under weak conditions as Theorem~\ref{thm:clt},
and discuss implications for lifetime parameter selection.
To enable valid feasible statistical inference,
we provide a consistent variance estimator
in Lemma~\ref{lem:variance_estimation}, and use it to
construct confidence intervals in Theorem~\ref{thm:confidence}.

Section~\ref{sec:debiased} introduces our proposed debiased Mondrian random
forests,
a family of estimators
based on linear combinations of Mondrian random forests
with varying lifetime parameters.
These parameters are carefully chosen to annihilate leading terms
in our bias characterization,
yielding an estimator with superior bias properties
(Lemma~\ref{lem:bias_debiased}).
We also study the variance of this debiased estimator
(Lemma~\ref{lem:variance_debiased}), and derive
its rate of convergence in Theorem~\ref{thm:minimax}.
The resulting rate is
shown to be minimax-optimal in mean squared error for each
H{\"o}lder parameter $\beta > 0$,
under regularity conditions.
Furthermore, Theorem~\ref{thm:clt_debiased} verifies that a
Berry--Esseen theorem holds for the debiased Mondrian random forest.
We again discuss the implications for the lifetime parameter,
and provide a consistent variance estimator
(Lemma~\ref{lem:variance_estimation_debiased})
for constructing confidence intervals
(Theorem~\ref{thm:confidence_debiased}).

Section~\ref{sec:implementation} discusses
implementation details and empirical results,
beginning by presenting a data-driven approach to
selecting the crucial lifetime parameter using polynomial estimation.
We also give advice on choosing the number of trees,
as well as other parameters associated with the debiasing procedure.
Empirical simulation results are presented using synthetic data,
demonstrating the practical value of our methods for optimal point estimation
and feasible robust bias-corrected inference.

Section~\ref{sec:compute} considers the computational aspects
of our methodology, presenting algorithmic procedures
with precisely characterized average case time complexity bounds for
the batch setting (Algorithm~\ref{alg:batch}, Lemma~\ref{lem:compute_batch})
and for online learning regimes
(Algorithm~\ref{alg:online}, Lemma~\ref{lem:compute_online}).

Concluding remarks are given in Section~\ref{sec:conclusion},
while \iftoggle{journal}{the supplementary material
\citep{cattaneo2025mondriansupplement} contains}{the appendices contain}
all the mathematical proofs of our theoretical results%
\iftoggle{journal}{}{ (Appendix~\ref{sec:proofs})},
alongside additional empirical studies%
\iftoggle{journal}{}{ (Appendix~\ref{sec:additional_empirical})}.

\subsection{Notation}

We write $\|\cdot\|_2$ for the usual Euclidean $\ell^2$ norm on $\R^d$.
The natural numbers are $\N = \{0, 1, 2, \ldots \}$.
We use $a \wedge b$ for the minimum and $a \vee b$ for the maximum
of two real numbers.
For non-negative sequences
$a_n$ and $b_n$, we write
$a_n \lesssim b_n$
to indicate that
$a_n / b_n$ is bounded for $n\geq 1$.
If $a_n \lesssim b_n \lesssim a_n$,
we write $a_n \asymp b_n$.
For random non-negative sequences
$A_n$ and $B_n$, similarly we write
$A_n \lesssim_\P B_n$
if $A_n / B_n$ is bounded in probability.
Let $\Phi: \R \to \R$ be the
cumulative distribution function of the standard normal
distribution, and for $\alpha \in (0, 1)$, let
$q_{\alpha}$ be the normal quantile satisfying $\Phi(q_{\alpha}) = \alpha$.

\section{Setup}
\label{sec:setup}

When using a Mondrian random forest, there are two sources of randomness.
The first is the data,
and here we consider the nonparametric regression setting
with $d$-dimensional covariates.
The second source is injected purposely from a collection of
independent trees drawn from a Mondrian process
using a specified lifetime parameter.

\subsection{Data generation}

We begin with a definition of H{\"o}lder continuity,
which is used to determine a target class of regression functions,
and which participates in controlling the bias of various estimators.

\begin{definition}[H{\"o}lder continuity]%

  Take $\beta > 0$ and define
  $\flbeta = \lceil \beta - 1 \rceil$
  as the largest integer strictly less than $\beta$.
  We say a function $g: [0,1]^d \to \R$
  is $\beta$-H{\"o}lder continuous and
  we write $g \in \cH^\beta$ if
  $g$ is $\flbeta$ times differentiable and
  $\max_{|\nu| = \flbeta}
  \left| \partial^\nu g(x) - \partial^{\nu} g(x') \right|
  \leq C \|x-x'\|_2^{\beta - \flbeta}$
  for some constant $C > 0$
  and all $x, x' \in [0,1]^d$.
  Here, $\nu \in \N^d$ is a multi-index with
  $|\nu| = \sum_{j=1}^d \nu_j$ and
  $\partial^{\nu} g(x) =
  \partial^{|\nu|} g(x) \big/ \prod_{j=1}^d \partial x_j^{\nu_j}$.

\end{definition}

Throughout this paper,
we assume that the data satisfies the following assumption.

\begin{assumption}[Data generation]%
  \label{ass:data}

  Fix $d \geq 1$ and
  let $(X_i, Y_i)$ be independent and identically distributed (i.i.d.) samples
  from a distribution on $\R^d \times \R$,
  writing $\bX = (X_1, \ldots, X_n)$
  and $\bY = (Y_1, \ldots, Y_n)$.
  Suppose $X_i$ has Lebesgue density function $f(x)$
  on $[0,1]^d$ which is bounded away from zero
  and satisfies $f \in \cH^{\betaf}$ for some $\betaf > 0$.
  Suppose $\E[Y_i^2 \mid X_i]$ is bounded,
  let $\mu(X_i) = \E[Y_i \mid X_i]$,
  and assume $\mu \in \cH^{\betamu}$ where $\betamu > 0$.
  Let $\varepsilon_i = Y_i - \mu(X_i)$ and assume
  $\sigma^2(X_i) = \E[\varepsilon_i^2 \mid X_i]$ is
  bounded away from zero with $\sigma^2 \in \cH^{\betasigma}$
  for some $\betasigma > 0$.
  Set $\beta = \betamu \wedge (\betaf + 1)$.

\end{assumption}

Some comments are in order surrounding Assumption~\ref{ass:data}.
The requirement that the covariate density $f(x)$
should be strictly positive on all of $[0,1]^d$ may seem restrictive,
particularly when $d$ is moderately large.
However, since our theory is presented pointwise in
the design point $x$,
it is sufficient for this condition to hold only on some neighborhood of $x$.
To see this, note that continuity implies the density
is positive on some hypercube containing $x$. Upon rescaling
the covariates, this hypercube can be mapped onto $[0,1]^d$.
The same argument holds for the H{\"o}lder smoothness
assumptions, and for the upper and lower bounds
on the conditional variance function.

The parameter $\beta$ represents the effective
degree of smoothness which is captured by the
(suitably debiased) Mondrian random forest.
Its definition is motivated as follows: firstly,
we take $\beta \leq \betamu$ in order to compare our rates
of convergence with classical results for $\beta$-H\"older regression
functions. Secondly, due to the presence of design bias,
we require in our analysis that the density
function $f(x)$ should also be smooth
(though not necessarily as smooth as $\mu$),
imposing $\beta \leq \betaf + 1$. Our proofs characterize the roles of each
of the smoothness parameters $\betamu$, $\betaf$ and $\betasigma$ precisely,
though our main results depend only on $\beta$.
By allowing for $\beta_f < 1$, we strictly generalize
the Lipschitz density assumption of \citet[Theorem~3]{mourtada2020minimax}.

\subsection{The Mondrian process}

The Mondrian process was introduced by
\citet{roy2008mondrian} and offers a canonical method
for generating random rectangular partitions,
which can be used as the trees for a random forest
\citep{lakshminarayanan2014mondrian}.
For the reader's convenience, we give a brief description of this process
here; see \citet[Section~3]{mourtada2020minimax}
for a more complete construction.

For a fixed dimension $d$ and lifetime parameter $\lambda > 0$,
the Mondrian process is a stochastic process taking values
in the set of finite rectangular partitions of $[0,1]^d$.
For a rectangle $D = \prod_{j=1}^d [a_j, b_j] \subseteq [0,1]^d$,
we denote the side aligned with dimension $j$ by $D_j = [a_j, b_j]$,
write $D_j^- = a_j$ and $D_j^+ = b_j$ for its left and right
endpoints respectively,
and use $|D_j| = D_j^+ - D_j^-$ for its length.
The volume of $D$ is $|D| = \prod_{j=1}^{d} |D_j|$
and its linear dimension is $|D|_1 = \sum_{j=1}^{d} |D_j|$.

To sample a partition $T$ from the Mondrian process
$\cM \big( [0,1]^d, \lambda \big)$,
start at time $t=0$ with the trivial partition
of $[0,1]^d$ which has no splits.
Then repeatedly apply the following procedure to each cell $D$
in the partition.
Let $t_D$ be the time at which the cell was formed,
and sample $E_D \sim \Exp(|D|_1)$,
where $\Exp(a)$ is the exponential distribution on $[0, \infty)$
with Lebesgue density $a e^{-a x}$.
If $t_D + E_D \leq \lambda$, then split $D$.
This is done by first selecting a split dimension $J$ with
$\P(J=j) = |D_j| / |D|_1$, and then sampling a split location
$S_J \sim \Unif\big[D_J^-, D_J^+\big]$.
The cell $D$ splits into the two new cells
$\{x \in D : x_J \leq S_J\}$ and $\{x \in D : x_J > S_J\}$,
each with formation time $t_D + E_D$.
The output of this sampling procedure
is the partition $T$ consisting of the cells $D$
which were not split because $t_D + E_D > \lambda$.
The cell in $T$ containing a point $x \in [0,1]^d$
is written $T(x)$.
Figure~\ref{fig:mondrian_process} shows typical realizations of
$T \sim \cM\big( [0,1]^d, \lambda \big)$ for
$d=2$ and with different lifetime parameters $\lambda$.
\begin{figure}[ht]
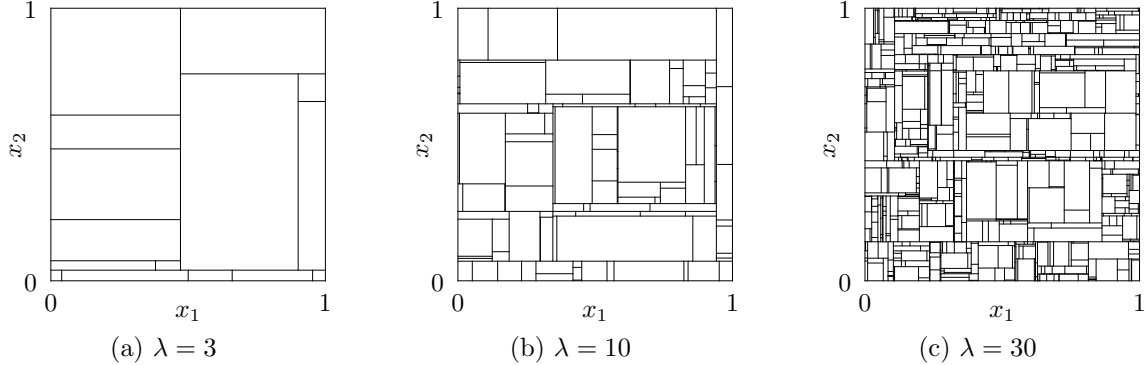

  \iftoggle{journal}{\vspace*{0mm}}{\vspace*{-4mm}}
  \centering
  \begin{subfigure}{0.32\textwidth}
    \centering
    \resizebox{0.97\textwidth}{!}{
\begingroup%
\makeatletter%
%
\makeatother%
\endgroup%
     }
    \caption{$\lambda = 3$}
  \end{subfigure}
  \begin{subfigure}{0.32\textwidth}
    \centering
    \resizebox{0.97\textwidth}{!}{
\begingroup%
\makeatletter%
%
%
\makeatother%
\endgroup%
     }
    \caption{$\lambda = 10$}
  \end{subfigure}
  \begin{subfigure}{0.32\textwidth}
    \centering
    \resizebox{0.97\textwidth}{!}{
\begingroup%
\makeatletter%
%
%
\makeatother%
\endgroup%
     }
    \caption{$\lambda = 30$}
  \end{subfigure}
  \caption{
    The Mondrian process
    $T \sim \cM \big( [0,1]^d, \lambda \big)$
    with $d=2$ and lifetime parameters $\lambda$.
  }
  \label{fig:mondrian_process}
  \vspace*{-5mm}
\end{figure}

\subsection{Mondrian random forests}

We define the Mondrian random forest estimator
\eqref{eq:estimator}
as in \citet{lakshminarayanan2014mondrian} and
\citet{mourtada2020minimax},
and will later extend it to a debiased version
in Section~\ref{sec:debiased}.
For a lifetime parameter $\lambda > 0$ and
forest size $B \geq 1$,
let $\bT = (T_1, \ldots, T_B)$ be a Mondrian forest
where $T_b \sim \cM\big([0,1]^d, \lambda\big)$ are
mutually independent Mondrian trees which are
independent of the data.
For $x \in [0,1]^d$, we write
$N_b(x) = \sum_{i=1}^{n} \I\{X_i \in T_b(x)\}$
for the number of samples in $T_b(x)$,
with $\I$ denoting an indicator function.
Then the Mondrian random forest estimator of $\mu(x)$ is
\begin{align}
  \label{eq:estimator}
  \hat\mu(x)
  =
  \frac{1}{B}
  \sum_{b=1}^B
  \frac{\sum_{i=1}^n Y_i \, \I\big\{ X_i \in T_b(x) \big\}}
  {N_b(x)}.
\end{align}
If there are no samples $X_i$ in $T_b(x)$ then $N_b(x) = 0$,
so we define $0/0 = 0$
(see
  \iftoggle{journal}{the supplementary material
  \citep{cattaneo2025mondriansupplement}}{Appendix~\ref{sec:proofs}}
for details).
To ensure the bias and variance of the Mondrian random forest estimator
converge to zero (see Section~\ref{sec:inference}),
and to avoid boundary issues,
we impose some basic conditions on
$x$, $\lambda$, and $B$ in Assumption~\ref{ass:estimator}.
We emphasize that our results are for the
low-dimensional nonparametric regime where
$d$ does not depend on the sample size $n$.

\begin{assumption}[Mondrian random forest estimator]%
  \label{ass:estimator}
  Suppose $x \in (0,1)^d$ is a fixed interior point of the support of $X_i$,
  and also that $\lambda \gtrsim (\log n)^3$,
  $n / \lambda^d \to \infty$, and $B \gtrsim (\log n)^d$.
\end{assumption}

The requirement that $n / \lambda^d \to \infty$ ensures that the number of data
points
$N_b(x)$ falling inside a typical Mondrian cell $T_b(x)$,
and hence the effective sample size of the Mondrian random forest,
diverges in large samples.
Assumption~\ref{ass:estimator} also implies that
the size of the forest $B$ should grow with $n$.
In our theory, the forest size $B$ plays three roles.
Firstly, we assume $B \gtrsim (\log n)^d$
to ensure sufficient averaging across independent trees; this
``large forest'' condition is used in the central limit theorem proof
(Theorems~\ref{thm:clt} and~\ref{thm:clt_debiased}) to overcome
heavy tail issues from very small Mondrian cells, and not to control
bias. Secondly, the Berry--Esseen bounds for our centered estimators
(Theorems~\ref{thm:clt} and~\ref{thm:clt_debiased}) include a $1/B$
term, so increasing $B$ improves the Gaussian approximation. Thirdly,
for feasible inference and mean squared error optimality,
$B$ must be large enough that the finite-forest squared bias terms
(e.g., $1/(\lambda^{2(1\wedge\beta)}B)$; see Lemma~\ref{lem:bias} and
Theorems~\ref{thm:rate}, \ref{thm:confidence}, and~\ref{thm:minimax})
are negligible relative to the leading ``variance'' term of the
infinite-forest estimator, which scales as $\lambda^d/n$
(Lemma~\ref{lem:variance}), necessitating that $B$ grows
as a positive fractional power of $n$.
Thus, for the purpose of both meeting our
statistical assumptions and mitigating the computational burden, we
recommend choosing $B \asymp \sqrt n$ for Mondrian random forests;
selecting $B$ for our debiased estimator requires a different set of
conditions (see Sections~\ref{sec:debiased}
and~\ref{sec:implementation}). Large forests usually do not present
significant computational challenges in practice as the ensemble
estimator is easily parallelizable over the trees (see
Section~\ref{sec:compute} for more discussion). We will emphasize
explicitly where the ``large forest'' condition is important for our theory.

\section{Inference with Mondrian random forests}%
\label{sec:inference}

Our analysis begins with a standard conditional bias--variance decomposition
for the Mondrian random forest estimator:
\begin{align}
  \nonumber
  \hat\mu(x) - \mu(x)
  &=
  \Big( \hat\mu(x) - \E \big[ \hat \mu(x) \mid \bX, \bT \big]\Big)
  + \Big( \E \big[ \hat \mu(x) \mid \bX, \bT \big] - \mu(x)\Big) \\
  \label{eq:bias_variance_bias}
  &=
  \frac{1}{B}
  \sum_{b=1}^B
  \frac{\sum_{i=1}^n \big(\mu(X_i) - \mu(x)\big) \,
  \I\big\{ X_i \in T_b(x) \big\}}
  {N_b(x)} \\
  &\qquad+
  \label{eq:bias_variance_variance}
  \frac{1}{B}
  \sum_{b=1}^B
  \frac{\sum_{i=1}^n \varepsilon_i \, \I\big\{ X_i \in T_b(x) \big\}}
  {N_b(x)}.
\end{align}

Our approach to estimation and inference is as follows. Firstly, we
precisely characterize the probability limit of the ``bias'' term
\eqref{eq:bias_variance_bias}, and compute the second
conditional moment of the
``variance'' term \eqref{eq:bias_variance_variance}. This allows us to
understand the bias--variance trade-off, and to derive upper bounds on
the rate of convergence for the Mondrian random forest point estimator.

Secondly, we provide a central limit theorem
for the ``variance'' term \eqref{eq:bias_variance_variance}.
By ensuring that the standard deviation dominates the conditional bias,
we may conclude that a corresponding central limit theorem
holds for the Mondrian random forest \eqref{eq:estimator}.
With an appropriate estimator for the variance,
we then establish procedures for valid and feasible statistical inference
on the unknown regression function $\mu(x)$.

\subsection{Bias and variance characterizations}%

We begin with \eqref{eq:bias_variance_bias}, which captures the bias of the
Mondrian random forest estimator conditional on the covariates $\bX$ and the
forest $\bT$. The next lemma demonstrates that this conditional
bias converges in $L^2$ at a certain rate, and provides a
precise characterization of the resulting non-random limiting bias.

\begin{lemma}[Bias]%
  \label{lem:bias}
  Suppose Assumptions~\ref{ass:data} and~\ref{ass:estimator} hold.
  For each $1 \leq r \leq \lfloor \flbeta / 2 \rfloor$ there exists
  $B_r(x) \in \R$, which is a function of
  the derivatives of $f$ and $\mu$ at $x$ up to order $2r$,
  with
  \begin{align}
    \label{eq:bias}
    \E \left[
      \Bigg(
        \E \bigl[ \hat \mu(x) \mid \bX, \bT \bigr]
        - \mu(x)
        - \sum_{r=1}^{\lfloor \flbeta / 2 \rfloor}
        \frac{B_r(x)}{\lambda^{2r}}
      \Bigg)^2
    \right]
    &\lesssim
    \frac{1}{\lambda^{2\beta}}
    + \frac{1}{\lambda^{2(1 \wedge \beta)}  B}
    + \frac{1}{\lambda^{2(1 \wedge \beta)}}
    \frac{\lambda^d}{n}.
  \end{align}
  Whenever $\beta > 2$, the leading bias is the quadratic term
  \begin{align*}
    \frac{B_1(x)}{\lambda^2}
    &=
    \frac{1}{2 \lambda^2}
    \sum_{j=1}^d \frac{\partial^2 \mu(x)}{\partial x_j^2}
    + \frac{1}{2 \lambda^2}
    \frac{1}{f(x)}
    \sum_{j=1}^{d} \frac{\partial \mu(x)}{\partial x_j}
    \frac{\partial f(x)}{\partial x_j}.
  \end{align*}
  If $X_i \sim \Unif\big([0,1]^d\big)$ then $f(x) = 1$,
  and using multi-index notation we have
  \begin{align*}
    \frac{B_r(x)}{\lambda^{2r}}
    &=
    \frac{1}{\lambda^{2r}}
    \sum_{|\nu|=r}
    \partial^{2 \nu} \mu(x)
    \prod_{j=1}^d
    \frac{1}{\nu_j + 1}.
  \end{align*}
\end{lemma}

The bias characterization in Lemma~\ref{lem:bias} incorporates some high-degree
polynomial terms in the lifetime parameter $\lambda$ which for now may seem
ignorable. The magnitude of the bias is determined by the leading term
in \eqref{eq:bias}, typically of order $1/\lambda^2$ whenever $\beta \geq 2$.
This suffices for ensuring a negligible contribution from the bias with an
appropriate choice of lifetime. However, the advantage of specifying
higher-order terms will become apparent in Section~\ref{sec:debiased}, where we
construct a debiased Mondrian random forest estimator, directly targeting and
annihilating the higher-order terms in order to furnish superior estimation and
inference properties. We also demonstrate numerically the detrimental role of
bias in estimation and inference in Section~\ref{sec:implementation}.

In Lemma~\ref{lem:bias} we give some explicit examples
of calculating the limiting bias when $\beta > 2$
or $X_i$ are uniformly distributed.
The general form of $B_r(x)$ is provided in
\iftoggle{journal}{the supplementary material
\citep{cattaneo2025mondriansupplement}}{Appendix~\ref{sec:proofs}}
but is somewhat unwieldy except
in specific situations.
Nonetheless, the most important properties are that $B_r(x)$
are non-random and
do not depend on the lifetime $\lambda$;
these are crucial features for our debiasing procedure given in
Section~\ref{sec:debiased}.
If the forest size $B$ does not diverge to infinity
then we suffer the first-order conditional bias term
$1 / \big(\lambda^{1 \wedge \beta} \sqrt B\big)$.
This phenomenon was explained by \citet{mourtada2020minimax},
who noted that it allows individual Mondrian trees ($B=1$) to achieve
minimax optimality in integrated mean squared error
only when $\beta \in (0, 1]$.
In contrast, large forests remove this first-order bias
through ensemble averaging
and as such are optimal for all $\beta \in (0, 2]$.

We now turn to \eqref{eq:bias_variance_variance}, which captures the stochastic
part of the Mondrian random forest. Lemma~\ref{lem:variance} determines the
probability limit of the scaled conditional variance of this term, alongside
its $L^2$ convergence rate. First, define
\begin{align*}
  \tilde \Sigma(x)
  &= \frac{n}{\lambda^d}
  \Var\big[
    \hat \mu(x) \mid \bX, \bT
  \big]
  &&\text{and}
  &\Sigma(x)
  &=
  \frac{\sigma^2(x)}{f(x)}
  \left(
    \frac{4 - 4 \log 2}{3 }
  \right)^d.
\end{align*}

\begin{lemma}[Variance]
  \label{lem:variance}
  Suppose Assumptions~\ref{ass:data} and~\ref{ass:estimator} hold. Then
  \begin{align*}
    \E \left[
      \big(
        \tilde \Sigma(x)
        - \Sigma(x)
      \big)^2
    \right]
    &\lesssim
    \frac{\lambda^d}{n}
    + \frac{1}{B}
    + \frac{1}{\lambda^{2(1 \wedge \betaf \wedge \betasigma)}}.
  \end{align*}
\end{lemma}
As $n/\lambda^d \to \infty$, $B \to \infty$ and
$\lambda \to \infty$ by Assumption~\ref{ass:estimator},
it follows from Lemma~\ref{lem:variance} that
\begin{align}
  \label{eq:variance_bound}
  \E\Bigl[
    \Var\bigl[\hat\mu(x) \mid \bX, \bT\bigr]
  \Bigr]
  &=
  \frac{\lambda^d}{n}
  \E\bigl[\tilde \Sigma(x)\bigr]
  \lesssim
  \frac{\lambda^d}{n}
  \biggl(
    \Sigma(x)
    + \sqrt{\frac{\lambda^d}{n}}
    + \frac{1}{\sqrt B}
    + \frac{1}{\lambda^{1 \wedge \betaf \wedge \betasigma}}
  \biggr)
  \lesssim
  \frac{\lambda^d}{n}.
\end{align}
An upper bound on the $L^2$ rate of convergence of the Mondrian random forest
estimator can therefore
be deduced from the bias--variance decomposition,
Lemma~\ref{lem:bias}, and \eqref{eq:variance_bound}.
This rate of convergence depends on the sequence of lifetime parameters
$\lambda$;
for optimal point estimation, we may balance the
contributions from the bias and from the standard deviation by ensuring that
$1/\lambda^{2 \wedge \beta}
+ 1/\big(\lambda^{1 \wedge \beta} \sqrt B\big)
\asymp \sqrt{\lambda^d / n}$, or equivalently if
$\lambda \asymp n^{\frac{1}{d + 2(2 \wedge \beta)}}$
and $B \gtrsim n^{\frac{2(2 \wedge \beta) - 2(1 \wedge \beta)}
{d + 2(2 \wedge \beta)}}$.
We formalize these deductions in Theorem~\ref{thm:rate}
and note that they imply that the Mondrian random forest is
rate-minimax-optimal \citep{stone1982optimal}
in pointwise mean squared error
for $\beta$-H{\"o}lder functions with $\beta \in (0,2]$;
a corresponding result for integrated mean squared error
was provided by \citet[Theorem~2]{mourtada2020minimax}.

\begin{theorem}[Mean squared error]%
  \label{thm:rate}
  Suppose Assumptions~\ref{ass:data} and~\ref{ass:estimator} hold. Then
  \begin{align*}
    \E \left[ \big( \hat \mu(x) - \mu(x) \big)^2 \right]
    \lesssim
    \frac{\lambda^d}{n}
    + \frac{1}{\lambda^{2(2 \wedge \beta)}}
    + \frac{1}{\lambda^{2(1 \wedge \beta)} B}.
  \end{align*}
  If further
  $\lambda \asymp n^{\frac{1}{d + 2(2 \wedge \beta)}}$
  and $B \gtrsim n^{\frac{2(2 \wedge \beta) - 2(1 \wedge \beta)}
  {d + 2(2 \wedge \beta)}}$, then
  \begin{align*}
    \E \left[ \big( \hat \mu(x) - \mu(x) \big)^2 \right]
    \lesssim
    n^{- \frac{2 (2 \wedge \beta)}{d + 2(2 \wedge \beta)}}.
  \end{align*}
\end{theorem}

We take this opportunity to contrast Mondrian random forests with classical
nonparametric local smoothing methods.
For example, the lifetime $\lambda$ plays a similar role to the inverse
bandwidth for kernel smoothing
as it determines both the effective sample size $n / \lambda^d$
and the scale of localization $1/\lambda$, and thus also the
associated rate of convergence. Likewise, $1/\lambda$ controls the diameter of
a typical cell in Mondrian partition-based smoothing.
However, due to the Mondrian process
construction, some cells are typically ``too small''
(equivalent to an insufficiently large bandwidth)
to give an appropriate effective sample size.
In the same manner, classical methods based on non-random partitioning
such as spline estimators
typically impose a quasi-uniformity assumption to ensure
all the cells are of comparable size \citep{huang2003local,cattaneo2020large},
a property which does not hold
for the Mondrian process (not even with high probability).

\subsection{Central limit theorem}

Having discussed the point estimation properties of the
Mondrian random forest estimator, we present a central limit theorem
which forms the core of our methodology for performing statistical
inference. As well as establishing asymptotic normality of
the appropriately centered and scaled estimator, we also provide a rate
of convergence in terms of a Berry--Esseen-style bound on the
Kolmogorov--Smirnov distance from the normal distribution.
In addition to precisely quantifying the quality of the Gaussian distributional
approximation, this allows us to obtain explicit bounds on the
coverage error of feasible confidence intervals.

Before stating the theorem,
we highlight some of the challenges involved in establishing such a result.
At first glance, the summands in \eqref{eq:bias_variance_variance}
appear independent over $1 \leq i \leq n$, conditional
on the forest $\bT$, depending only on $X_i$ and $\varepsilon_i$.
However, the $N_b(x)$ appearing in the denominator depends on all $X_i$
simultaneously, violating this independence assumption
and rendering classical central limit theorems inapplicable.
A natural preliminary attempt to resolve this issue is to observe that
\begin{align*}
  N_b(x)&=
  \sum_{i=1}^{n} \I\big\{X_i \in T_b(x)\big\}
  \approx
  n \, \P \big( X_i \in T_b(x) \mid T_b \big)
  \approx
  n f(x) |T_b(x)|
\end{align*}
with high probability.
One could attempt to use this by
approximating the estimator with an average of
i.i.d.\ random variables, or by
employing a central limit theorem
conditional on $\bX$ and $\bT$.
However, such an approach fails because
$\E \left[ 1 / |T_b(x)|^2 \right] = \infty$;
the possible existence of small cells causes the law of the
inverse cell volume to have heavy tails.
For similar reasons, attempts to directly establish a central limit theorem
based on $2 + \delta$ moments, such as the classical Lyapunov central limit
theorem,
are ineffective.

We circumvent these problems by directly analyzing
$\I\{N_b(x) \geq 1\} / N_b(x)$.
We establish concentration properties
for this non-linear function of $X_i$ via the Efron--Stein inequality
\citep[Section 3.1]{boucheron2013concentration}
along with a sequence of delicate preliminary lemmas regarding
inverse moments of truncated (conditional) binomial random variables.
In particular, we show that
$\E \left[ \I \{N_b(x) \geq 1\} / N_b(x) \right]
\lesssim \lambda^d / n$ and
$\E \left[ \I \{N_b(x) \geq 1\} / N_b(x)^2 \right]
\lesssim \lambda^{2d} (\log n)^d / n^2$.
Asymptotic normality is then established by a careful application of
a Berry--Esseen theorem \citep{Petrov_1995_Book}
conditional on $(\bX, \bT)$.
\iftoggle{journal}{The supplementary material
\citep{cattaneo2025mondriansupplement}}{Appendix~\ref{sec:proofs}}
provides all the technical details.

The following theorem gives our
Berry--Esseen-type central limit theorem for
the centered (zero mean conditional on the covariates and the trees)
``variance'' term from \eqref{eq:bias_variance_variance},
scaled and standardized by its conditional variance $\tilde \Sigma(x)$.
Note that on the event $\tilde\Sigma(x) = 0$, we also have
$\hat\mu(x) = 0$ and $\E \left[ \hat\mu(x) \mid \bX, \bT \right] = 0$,
so continue to define $0 / 0 = 0$.

\begin{theorem}[Central limit theorem]
  \label{thm:clt}
  If Assumptions~\ref{ass:data}
  and~\ref{ass:estimator} hold, and
  $\E[|Y_i|^{2+\delta} \mid X_i ]$ is bounded almost surely
  with $\delta > 0$, then
  \begin{align}
    \label{eq:berry_esseen}
    \sup_{t\in\mathbb{R}}
    \left|
    \P \left(
      \sqrt{\frac{n}{\lambda^d}}
      \frac{\hat \mu(x) - \E\left[ \hat \mu(x) \mid \bX, \bT \right]}
      {\sqrt{\smash[b]{\tilde \Sigma(x)}}}
      \leq t
    \right)
    - \Phi(t)
    \right|
    &\lesssim
    \left( \frac{\lambda^d}{n} \right)^{\frac{1 \wedge \delta}{2}}
    + \frac{1}{B}.
  \end{align}
\end{theorem}

We make some remarks on Theorem~\ref{thm:clt}.
Firstly, since $n/\lambda^d$ is the effective sample size and
$Y_i$ has only $2 + \delta$ finite moments,
the first term in \eqref{eq:berry_esseen} is likely to be unimprovable
\citep[Theorem~3.4.1]{ibragimov1975independent}. In particular,
we attain the classical Berry--Esseen rate when
$\E[|Y_i|^3 \mid \bX]$ is bounded and $B \gtrsim \sqrt{n / \lambda^d}$.

The condition of $B \gtrsim (\log n)^d$ is central to our proof of
Theorem~\ref{thm:clt},
ensuring sufficient ``mixing'' of different Mondrian cells to escape the
heavy-tailed phenomenon detailed in the preceding discussion.
For concreteness, the large forest condition allows us to deal
with expressions such as
$\E \left[ 1 / (|T_b(x)| |T_{b'}(x)|) \right] =
\E \left[ 1 / |T_b(x)| \right] \E \left[ 1 / |T_{b'}(x)| \right]
\approx \lambda^{2d} < \infty$
where $b \neq b'$, by independence of the trees, rather than
the ``no ensembling'' single tree analog
$\E \left[ 1 / |T_b(x)|^2 \right] = \infty$.

Nonetheless, it is not clear whether the $1/B$ term is
strictly necessary in \eqref{eq:berry_esseen} or if it is an
artifact of the proof.
When $B$ is bounded, $\tilde\Sigma(x)$ remains random in the limit,
and in fact it is not difficult to show that in this regime we have
that $\E\big[\{\tilde \Sigma(x)\}^2\big] \geq (\log n)^d$, which diverges
(cf.,\ Lemma~\ref{lem:variance}).
While these mildly pathological properties may not necessarily
render the central limit theorem invalid, they certainly highlight some
issues associated with inference based on a single tree or a small forest.

Theorem~\ref{thm:clt} applies only to the centered Mondrian random forest
estimator;
in order for it to be useful in a feasible inference setting, we must
combine it with methods for controlling the conditional bias
(see Lemma~\ref{lem:bias}).
In Section~\ref{sec:debiased} we will show
how the estimator can be debiased, giving weaker lifetime conditions for
inference, improved rates of convergence, and superior coverage
guarantees, whenever additional smoothness is available.

\subsection{Confidence intervals}

We demonstrate how to use our previous results to construct
valid confidence intervals for the regression function $\mu(x)$.
To do this, there are two preliminary issues which must be resolved.
Firstly, the Berry--Esseen central limit theorem presented in
Theorem~\ref{thm:clt} is stated for the Mondrian random forest
estimator $\hat\mu(x)$
centered at its conditional expectation
$\E \left[\hat \mu(x) \mid \bX, \bT \right]$,
rather than at the true value $\mu(x)$. As such,
we use Lemma~\ref{lem:bias} to ensure that the bias
$\E \left[\hat \mu(x) \mid \bX, \bT \right] - \mu(x)$ is taken into
account when establishing procedures for inference.
Specifically, the bias should shrink faster than the standard deviation;
this requires
$1 / \lambda^{2 \wedge \beta}
+ 1 / \big(\lambda^{1 \wedge \beta} \sqrt B\big)
\ll \sqrt{\lambda^d / n}$,
which is satisfied by imposing the restrictions
$\lambda \gg n^{\frac{1}{d + 2(2 \wedge \beta)}}$
and $B \gg n^{\frac{2(2 \wedge \beta) - 2(1 \wedge \beta)}
{d + 2(2 \wedge \beta)}}$
on the lifetime $\lambda$ and forest size $B$.

The second issue is that
the variances $\tilde\Sigma(x)$ and $\Sigma(x)$
depend on the unknown quantities $\sigma^2(x)$ and $f(x)$.
To conduct feasible inference, we must therefore provide
a consistent variance estimator. To this end, define
\begin{align}
  \label{eq:sigma2_hat}
  \hat\sigma^2(x)
  &=
  \frac{1}{B}
  \sum_{b=1}^{B}
  \sum_{i=1}^n
  \frac{\big(Y_i - \hat \mu(x)\big)^2 \, \I\{X_i \in T_b(x)\}}
  {N_b(x)}, \\
  \nonumber
  \hat\Sigma(x)
  &=
  \hat\sigma^2(x)
  \frac{n}{\lambda^d}
  \sum_{i=1}^n
  \left(
    \frac{1}{B}
    \sum_{b=1}^B
    \frac{\I\{X_i \in T_b(x)\}}{N_b(x)}
  \right)^2.
\end{align}
\begin{lemma}[Variance estimation]%
  \label{lem:variance_estimation}
  If Assumptions~\ref{ass:data}
  and~\ref{ass:estimator} hold, and
  $\E[|Y_i|^{2+\delta} \mid X_i ]$ is bounded almost surely with $\delta > 0$,
  then
  \begin{align*}
    \left(
      \E \left[
        \big|
        \hat\Sigma(x)
        - \Sigma(x)
        \big|^{\frac{2 - \I\{\delta < 2\}}{2}}
      \right]
    \right)^{\frac{2}{2 - \I\{\delta < 2\}}}
    &\lesssim
    \left(\frac{\lambda^d}{n}\right)^{\frac{1}{2}
    - \frac{\I\{\delta < 2\}}{2 + \delta}}
    + \frac{1}{\sqrt B}
    + \frac{1}{\lambda^{1 \wedge \betamu \wedge \betaf \wedge \betasigma}}.
  \end{align*}
\end{lemma}

For a confidence level $\alpha \in (0, 1)$,
Theorem~\ref{thm:confidence} shows how to construct an asymptotically valid
$100 (1-\alpha)\%$ confidence interval
for the regression function $\mu(x)$.
The restrictions on the lifetime $\lambda$ and forest size $B$
are the same as those previously discussed,
and an explicit upper bound on the coverage error rate is provided.
Define the interval estimator
\begin{align*}
  \CI(x) = \left[
    \hat \mu(x)
    - \sqrt{\frac{\lambda^d}{n}} \hat \Sigma(x)^{1/2} \,
    q_{1 - \alpha / 2}, \
    \hat \mu(x)
    - \sqrt{\frac{\lambda^d}{n}} \hat \Sigma(x)^{1/2} \,
    q_{\alpha / 2}
  \right].
\end{align*}
\begin{theorem}[Confidence intervals]%
  \label{thm:confidence}
  If Assumptions~\ref{ass:data}
  and~\ref{ass:estimator} hold, and
  $\E[|Y_i|^{2+\delta} \mid X_i ]$ is bounded almost surely
  with $\delta > 0$, then
  \begin{align*}
    &\left| \P \big( \mu(x) \in \CI(x) \big) - (1 - \alpha) \right| \\
    &\quad\lesssim
    \frac{n}{\lambda^d} \frac{1}{\lambda^{2(2 \wedge \beta)}}
    + \Biggl(
      \left(
        \frac{\lambda^d}{n}
      \right)^{1 - \frac{2 \, \I\{\delta < 2\}}{2 + \delta}}
      + \frac{1}{B}
      + \frac{1}
      {\lambda^{2(1 \wedge \betamu \wedge \betaf \wedge \betasigma)}}
      + \frac{n}{\lambda^d}
      \frac{1}{\lambda^{2(1 \wedge \beta)} B}
    \Biggr)^{\frac{1}{5 + 2 \, \I\{\delta < 2\}}}.
  \end{align*}
\end{theorem}

When coupled with an appropriate lifetime selection method
(see Section~\ref{sec:implementation}),
Theorem~\ref{thm:confidence} gives a feasible procedure
for uncertainty quantification in Mondrian random forests.
Our procedure requires no adjustment of the original Mondrian random
forest estimator beyond ensuring that the bias is negligible,
and in particular does not rely on sample splitting.
The construction of confidence intervals is just one corollary
of the result given in Theorem~\ref{thm:clt};
other applications include hypothesis testing based on the value
of $\mu(x)$ at a design point $x$ by inversion of the confidence interval,
as well as specification testing by comparison with
a $\sqrt{n}$-consistent parametric regression estimator.
The construction of simultaneous confidence
intervals for finitely many points $x_1, \ldots, x_D$
can be accomplished either using standard multiple testing corrections
or by first establishing a multivariate
central limit theorem using the Cram{\'e}r--Wold device
and formulating a consistent variance matrix estimator.

\section{Debiased Mondrian random forests}%
\label{sec:debiased}

We give our next main contribution:
a novel variant of the Mondrian random forest estimator
that corrects for higher-order bias with an approach based on
generalized jackknifing \citep{schucany1977improvement}.
This estimator retains the basic form of a Mondrian random forest
in the sense that it is a linear combination of Mondrian tree estimators,
but in this section we allow for non-identical linear coefficients,
some of which may be negative, and for differing
lifetime parameters across the trees.
Since the basic Mondrian random forest estimator is a special
case of this more general debiased version,
we will discuss only the latter throughout the rest of the paper.

We use the explicit form of the bias given in Lemma~\ref{lem:bias}
to construct the debiased Mondrian forest estimator as follows,
letting $J \geq 0$ be the bias correction order.
With $J=0$ we preserve the original Mondrian random forest,
with $J=1$ we remove second-order bias,
and with $J = \lfloor\flbeta / 2 \rfloor$ we remove bias terms
up to and including order $2 \lfloor\flbeta / 2 \rfloor$,
giving the maximum possible bias reduction achievable
in the H{\"o}lder class $\cH^\beta$ \citep{stone1982optimal}.
As such, only bias terms of order $1/\lambda^\beta$ will remain.

For $0 \leq r \leq J$, let $\hat \mu_r(x)$ be a Mondrian forest estimator
based on the trees
$T_{b r} \sim \cM\big([0,1]^d, \lambda_r \big)$
for $1 \leq b \leq B$, where $\lambda_r = a_r \lambda$ for some $a_r > 0$
and $\lambda > 0$.
We write $\bT$ to denote the collection of all the trees,
and suppose they are mutually independent.
We find values of $a_r$ along with coefficients $\omega_r \in \R$ which
annihilate the leading $J$ bias terms of the
debiased Mondrian random forest estimator
\begin{align}
  \label{eq:debiased}
  \hat \mu_\rd(x)
  &=
  \sum_{r=0}^J
  \omega_r \hat \mu_r(x)
  = \sum_{r=0}^{J}
  \omega_r
  \frac{1}{B}
  \sum_{b=1}^B
  \frac{\sum_{i=1}^n Y_i \,
  \I\big\{ X_i \in T_{b r}(x) \big\}}
  {N_{b r}(x)}.
\end{align}
This ensemble estimator retains the ``forest'' structure
of the original estimators,
but with varying lifetime parameters $\lambda_r$ and coefficients $\omega_r$.
Thus, referring to \eqref{eq:bias}, we desire
\begin{align*}
  \sum_{r=0}^{J}
  \omega_r
  \left(
    \mu(x)
    + \sum_{s=1}^{J}
    \frac{B_{s}(x)}{a_r^{2s} \lambda^{2s}}
  \right)
  &= \mu(x)
\end{align*}
for all $\lambda$, or equivalently the system of linear equations
$\sum_{r=0}^{J} \omega_r = 1$
and $\sum_{r=0}^{J} \omega_r a_r^{-2s} = 0$
for each $1 \leq s \leq J$.
We solve these as follows.
Define the $(J+1) \times (J+1)$ Vandermonde matrix
$A_{r s} = a_{r-1}^{2-2s}$,
let $\omega = (\omega_0, \ldots, \omega_J)^\T \in \R^{J+1}$
and take $e_0 = (1, 0, \ldots, 0)^\T \in \R^{J+1}$.
Then a solution for the debiasing coefficients is given by
$\omega = A^{-1} e_0$ whenever $A$ is non-singular.
In practice we can take $a_r$ to be a fixed geometric or arithmetic sequence
to ensure this is the case,
appealing to the Vandermonde determinant formula:
$\det A = \prod_{0 \leq r < s \leq J} (a_r^{-2} - a_s^{-2})
\neq 0$ whenever $a_r$ are distinct.
For example, one could set
$a_r = (1 + \gamma)^r$ or $a_r = 1 + \gamma r$ for some $\gamma > 0$.
Because we assume $\beta$, and therefore the choice of $J$, do not
depend on $n$, there is no need to quantify
the invertibility of $A$ by, for example, bounding its eigenvalues
away from zero as a function of $J$ and the choice of $a_r$.

The debiased Mondrian random forest estimator defined in \eqref{eq:debiased}
is a linear combination of standard Mondrian random forests,
and as such contains both a sum over $0 \leq r \leq J$,
representing the debiasing procedure,
and a sum over $1 \leq b \leq B$, representing the forest averaging.
We have been interpreting this estimator as a debiased
version of the standard Mondrian random forest given in \eqref{eq:estimator},
but it is equally valid to swap the order of these sums.
This gives rise to an alternative point of view:
we replace each Mondrian random tree with a ``debiased'' version,
and then take a forest of such modified trees.
This perspective is perhaps more in line with existing
techniques for constructing
randomized ensembles, where the outermost operation
represents a $B$-fold average
of randomized base learners (not necessarily locally constant decision trees),
each of which has a small bias component
\citep{caruana2004ensemble, zhou2019deep}.

\subsection{Bias and variance characterizations}

In Lemma~\ref{lem:bias_debiased}
we verify that this debiasing procedure does indeed annihilate the desired
bias terms; it is a direct consequence of Lemma~\ref{lem:bias} and
the construction of the debiased Mondrian random forest estimator
$\hat\mu_\rd(x)$.

\begin{lemma}[Bias of the debiased estimator]%
  \label{lem:bias_debiased}
  Suppose Assumptions~\ref{ass:data} and~\ref{ass:estimator} hold.
  Then in the notation of Lemma~\ref{lem:bias} and with
  $\bar\omega = \sum_{r=0}^J \omega_r a_r^{-2J - 2}$,
  \begin{align*}
    &\E \left[
      \left(
        \E \big[ \hat \mu_\rd(x) \mid \bX, \bT \big]
        - \mu(x)
        - \I\{2J+2 < \beta \}
        \frac{\bar\omega B_{J+1}(x)}{\lambda^{2J + 2}}
      \right)^2
    \right] \\
    &\qquad\lesssim
    \frac{1}{\lambda^{2((2J + 4) \wedge \beta)}}
    + \frac{1}{\lambda^{2(1 \wedge \beta)} B}
    + \frac{1}{\lambda^{2(1 \wedge \beta)}}
    \frac{\lambda^d}{n}.
  \end{align*}
\end{lemma}

Lemma~\ref{lem:bias_debiased} has the following consequence:
the leading bias term is characterized in terms of
$B_{J+1}(x)$ whenever $J < \beta/2 - 1$,
or equivalently $J < \lfloor \flbeta/2 \rfloor$,
that is, the debiasing order
$J$ does not exhaust the H{\"o}lder smoothness $\beta$.
If this condition does not hold, then the estimator is
fully debiased; the resulting leading bias
term is bounded above by $1/\lambda^\beta$ up to constants
but its form is left unspecified.

The following lemma controls the
variance of the debiased Mondrian random forest estimator.
With $\ell_{r r'} = 2 a_r ( 1 - a_r
\log(1 + a_{r'} / a_{r}) / a_{r'} ) / 3$, define
\begin{align*}
  \tilde \Sigma_\rd(x)
  &= \frac{n}{\lambda^d}
  \Var\big[
    \hat \mu_\rd(x) \mid \bX, \bT
  \big]
  &&\text{and}
  &\Sigma_\rd(x)
  &=
  \frac{\sigma^2(x)}{f(x)}
  \sum_{r=0}^{J}
  \sum_{r'=0}^{J}
  \omega_r
  \omega_{r'}
  \left( \ell_{r r'} + \ell_{r' r} \right)^d.
\end{align*}
\begin{lemma}[Variance of the debiased estimator]
  \label{lem:variance_debiased}
  Supposing Assumptions~\ref{ass:data} and~\ref{ass:estimator} hold,
  \begin{align*}
    \E \left[
      \big(
        \tilde \Sigma_\rd(x)
        - \Sigma_\rd(x)
      \big)^2
    \right]
    &\lesssim
    \frac{\lambda^d}{n}
    + \frac{1}{B}
    + \frac{1}{\lambda^{2(1 \wedge \betaf \wedge \betasigma)}}.
  \end{align*}
\end{lemma}

As in \eqref{eq:variance_bound}, it follows from
Lemma~\ref{lem:variance_debiased} that
\begin{equation}
  \label{eq:variance_bound_debiased}
  \E\Bigl[
    \Var\bigl[\hat\mu_\rd(x) \mid \bX, \bT\bigr]
  \Bigr]
  \lesssim
  \frac{\lambda^d}{n}.
\end{equation}

\subsection{Minimax optimality}

Our next main result, Theorem~\ref{thm:minimax}, shows that
when using an appropriate sequence of lifetime parameters $\lambda$,
the debiased Mondrian random forest estimator
achieves, up to constants, the minimax-optimal rate of convergence
for pointwise mean squared error
estimation of a $d$-dimensional regression function $\mu \in \cH^\beta$
\citep{stone1982optimal}.
This result holds for all $d \geq 1$ and all $\beta > 0$,
complementing a previous result
(see Theorem~\ref{thm:rate})
established only for $\beta \in (0, 2]$
and in integrated mean squared error
by \citet{mourtada2020minimax}.

\begin{theorem}[Mean squared error of the debiased estimator]%
  \label{thm:minimax}
  Grant Assumptions~\ref{ass:data} and~\ref{ass:estimator}. Then
  \begin{align*}
    \E \left[
      \big(
        \hat \mu_\rd(x)
        - \mu(x)
      \big)^2
    \right]
    &\lesssim
    \frac{\lambda^d}{n}
    + \frac{1}{\lambda^{2((2 J + 2) \wedge \beta)}}
    + \frac{1}{\lambda^{2(1 \wedge \beta)} B}.
  \end{align*}
  Thus with $J \geq \lfloor \flbeta / 2 \rfloor$,
  $\lambda \asymp n^{\frac{1}{d + 2 \beta}}$
  and $B \gtrsim n^{\frac{2 \beta - 2 (1 \wedge \beta)}{d + 2 \beta}}$,
  we have
  \begin{align*}
    \E \left[
      \big(
        \hat \mu_\rd(x)
        - \mu(x)
      \big)^2
    \right]
    &\lesssim
    n^{- \frac{2 \beta}{d + 2 \beta}}.
  \end{align*}
\end{theorem}

The sequence of lifetime parameters $\lambda$ required in
Theorem~\ref{thm:minimax} is
chosen to balance the bias and standard deviation bounds implied by
Lemma~\ref{lem:bias_debiased} and \eqref{eq:variance_bound_debiased}
respectively, in order to minimize the pointwise mean squared error.
While selecting an optimal debiasing order $J$ needs only
knowledge of an upper bound on the smoothness $\beta$,
choosing an optimal sequence of $\lambda$ values
does assume that $\beta$ is known a priori.
The problem of adapting to $\beta$ from data is beyond the scope of this paper;
we provide some practical advice
for tuning parameter selection
in Section~\ref{sec:implementation}.

Theorem~\ref{thm:minimax} complements the minimaxity results
proven by \citet{mourtada2020minimax} for
Mondrian trees (with $\beta \leq 1$) and for Mondrian random forests
(with $\beta \leq 2$), with one modification:
our version is stated in pointwise rather than integrated
mean squared error.
This is because our debiasing procedure is designed to handle
interior smoothing bias and as such does not provide any correction
for boundary bias.
We leave the development of such boundary corrections to future work,
but constructions similar to higher-order boundary-correcting
kernels should be possible.
If the region of integration is a compact set
in the interior of $[0,1]^d$
then we do obtain an optimal integrated mean squared error bound:
if $a \in (0, 1/2)$ is fixed then under the
same conditions as Theorem~\ref{thm:minimax},
with appropriate tuning of $\lambda$ and $B$,
\begin{align*}
  \E \left[
    \int_{[a, 1-a]^d}
    \big(
      \hat \mu_\rd(x)
      - \mu(x)
    \big)^2
    \diffi x
  \right]
  &\lesssim
  \frac{\lambda^d}{n}
  + \frac{1}{\lambda^{2\beta}}
  + \frac{1}{\lambda^{2(1 \wedge \beta)} B}
  \lesssim
  n^{-\frac{2 \beta}{d + 2 \beta}}.
\end{align*}

\subsection{Central limit theorem}

In Theorem~\ref{thm:clt_debiased}, we verify that a central
limit theorem holds for the debiased
random forest estimator $\hat\mu_\rd(x)$.
The strategy and challenges associated with proving
Theorem~\ref{thm:clt_debiased} are identical to those discussed earlier
surrounding Theorem~\ref{thm:clt}.
In fact in \iftoggle{journal}{the supplementary material
\citep{cattaneo2025mondriansupplement}}{Appendix~\ref{sec:proofs}}
we provide a direct proof only for Theorem~\ref{thm:clt_debiased}
and deduce Theorem~\ref{thm:clt} as a special case.
Again on the event $\tilde\Sigma_\rd(x) = 0$, we also have
$\hat\mu_\rd(x) = 0$ and $\E \left[ \hat\mu_\rd(x) \mid \bX, \bT \right] = 0$,
so we take $0/0=0$.

\begin{theorem}[Central limit theorem with debiasing]
  \label{thm:clt_debiased}
  Suppose Assumptions~\ref{ass:data}
  and~\ref{ass:estimator} hold, and
  $\E[|Y_i|^{2+\delta} \mid X_i ]$ is bounded almost surely
  with $\delta > 0$. Then
  \begin{align*}
    \sup_{t\in\mathbb{R}}
    \left|
    \P \left(
      \sqrt{\frac{n}{\lambda^d}}
      \frac{\hat \mu_\rd(x) - \E\left[ \hat \mu_\rd(x) \mid \bX, \bT \right]}
      {\sqrt{\smash[b]{\tilde \Sigma_\rd(x)}}}
      \leq t
    \right)
    - \Phi(t)
    \right|
    &\lesssim
    \left( \frac{\lambda^d}{n} \right)^{\frac{1 \wedge \delta}{2}}
    + \frac{1}{B}.
  \end{align*}
\end{theorem}

\subsection{Confidence intervals}

As before, to conduct valid feasible inference we must ensure that the bias
(now significantly reduced due to our debiasing procedure)
is negligible when compared to the standard deviation
(which is of the same order as before).
We treat here the general ``partial debiasing''
setting where either the debiasing order $J$ or
the H{\"o}lder smoothness $\beta$ may determine the magnitude
of the bias, which is $1 / \lambda^{(2 J + 2) \wedge \beta}$.
For optimal results, one should take
$J \geq \lfloor \flbeta/2 \rfloor$ to ensure total debiasing,
as in Theorem~\ref{thm:minimax}.
We thus require
$1 / \lambda^{(2J+2) \wedge \beta}
+ 1 / \big(\lambda^{1 \wedge \beta} \sqrt B\big)
\ll \sqrt{\lambda^d / n}$,
satisfied by imposing
$\lambda \gg n^{\frac{1}{d + 2 ((2 J + 2) \wedge \beta)}}$
and $B \gg n^{\frac{2 ((2 J + 2) \wedge \beta) - 2 (1 \wedge \beta)}
{d + 2 ((2 J + 2) \wedge \beta)}}$
on the lifetime parameter $\lambda$
and forest size $B$.

Once again, we propose a variance estimator
and show that it is consistent. With $\hat\sigma^2(x)$ as in
\eqref{eq:sigma2_hat} in Section~\ref{sec:inference}, define
\begin{align}
  \label{eq:Sigma_hat_d}
  \hat\Sigma_\rd(x)
  &=
  \hat\sigma^2(x)
  \frac{n}{\lambda^d}
  \sum_{i=1}^n
  \left(
    \sum_{r=0}^J
    \omega_r
    \frac{1}{B}
    \sum_{b=1}^B
    \frac{\I\{X_i \in T_{b r}(x)\}}
    {N_{b r}(x)}
  \right)^2.
\end{align}
\begin{lemma}[Variance estimation for the debiased estimator]%
  \label{lem:variance_estimation_debiased}
  Suppose Assumptions~\ref{ass:data}
  and~\ref{ass:estimator} hold, and
  $\E[|Y_i|^{2 + \delta} \mid X_i ]$ is bounded almost surely
  with $\delta > 0$. Then
  \begin{align*}
    \left(
      \E \left[
        \big|
        \hat\Sigma_\rd(x)
        - \Sigma_\rd(x)
        \big|^{\frac{2 - \I\{\delta < 2\}}{2}}
      \right]
    \right)^{\frac{2}{2 - \I\{\delta < 2\}}}
    &\lesssim
    \left(\frac{\lambda^d}{n}\right)^{\frac{1}{2}
    - \frac{\I\{\delta < 2\}}{2 + \delta}}
    + \frac{1}{\sqrt B}
    + \frac{1}{\lambda^{1 \wedge \betamu \wedge \betaf \wedge \betasigma}}.
  \end{align*}
\end{lemma}

In analogy to Section~\ref{sec:inference},
we now demonstrate the construction of feasible valid confidence
intervals using the debiased Mondrian random forest estimator
in Theorem~\ref{thm:confidence_debiased}. Consider the
debiased $100(1-\alpha)\%$ confidence interval estimator
\begin{align}
  \label{eq:CI_d}
  \CI_\rd(x) = \left[
    \hat \mu_\rd(x)
    - \sqrt{\frac{\lambda^d}{n}} \hat \Sigma_\rd(x)^{1/2}\,
    q_{1 - \alpha / 2}, \
    \hat \mu_\rd(x)
    - \sqrt{\frac{\lambda^d}{n}} \hat \Sigma_\rd(x)^{1/2}\,
    q_{\alpha / 2}
  \right].
\end{align}

\begin{theorem}[Confidence intervals with debiasing]
  \label{thm:confidence_debiased}
  Suppose Assumptions~\ref{ass:data}
  and~\ref{ass:estimator} hold, and
  $\E[|Y_i|^{2+\delta} \mid X_i ]$ is bounded almost surely
  with $\delta > 0$. Then
  \begin{align*}
    &\left| \P \big( \mu(x) \in \CI_\rd(x) \big) - (1 - \alpha) \right| \\
    &\quad\lesssim
    \frac{n}{\lambda^d} \frac{1}{\lambda^{2((2 J + 2) \wedge \beta)}}
    + \Biggl(
      \hspace*{-1mm}
      \left(
        \frac{\lambda^d}{n}
      \right)^{1 - \frac{2 \, \I\{\delta < 2\}}{2 + \delta}}
      \hspace*{-3mm}
      + \frac{1}{B}
      + \frac{1}
      {\lambda^{2(1 \wedge \betamu \wedge \betaf \wedge \betasigma)}}
      + \frac{n}{\lambda^d}
      \frac{1}{\lambda^{2(1 \wedge \beta)} B}
    \Biggr)^{\hspace*{-1mm}\frac{1}{5 + 2 \, \I\{\delta < 2\}}}.
  \end{align*}

\end{theorem}

One important benefit of our debiasing technique is made clear in
Theorem~\ref{thm:confidence_debiased}: the restrictions imposed
on the lifetime parameter $\lambda$ are substantially relaxed,
especially in smooth classes with large $\beta$.
As well as the high-level of benefit of relaxed conditions,
this is also useful for practical selection of appropriate lifetimes
for estimation and inference respectively;
see Section~\ref{sec:implementation} for more details.
Nonetheless, such improvements do not come without concession.
The limiting variance of the debiased estimator
is typically larger than that of the unbiased version
in small samples
(the extent of this increase depends on the choice of the
debiasing parameters $a_r$),
leading to wider confidence intervals and larger estimation error,
despite the theoretical asymptotic improvements.
Nonetheless, the empirical results in Section~\ref{sec:implementation}
demonstrate that the debiasing effect can overcome the increased
variance with moderate sample sizes.
Because we employ symmetric confidence intervals, the coverage error depends
on the squared bias
$1 / \lambda^{2((2J+2)\wedge\beta)}$,
whereas the corresponding Berry--Esseen rate would depend on the
(larger) linear bias $1 / \lambda^{(2J+2)\wedge\beta}$.

\section{Implementation and empirical results}%
\label{sec:implementation}

We discuss procedures for selecting the parameters
involved in fitting a debiased Mondrian random forest;
namely the base lifetime parameter $\lambda$,
the number of trees in each forest $B$,
the order of the bias correction $J$,
and the debiasing scale parameters $a_r$ for $0 \leq r \leq J$.
We then provide empirical results with simulated data
to demonstrate the effectiveness of our methods.

\subsection{Reproducibility}

All code used to generate results in this section
(e.g., model training and parameter tuning) is available at
\url{https://github.com/WGUNDERWOOD/MondrianForests.jl}.

\subsection{Tuning parameter selection}
\label{sec:parameter_selection}

The most important parameter is the base Mondrian lifetime $\lambda$,
which plays the role of a complexity parameter and thus governs the overall
bias--variance trade-off of the estimator.
Correct tuning of $\lambda$ is especially important in two main respects:
firstly, in order to use the central limit theorem established in
Theorem~\ref{thm:clt_debiased}, we must have that the bias converges to zero,
requiring $\lambda \gg n^{\frac{1}{d + 2((2 J + 2) \wedge \beta)}}$.
Secondly, the minimax optimality result of Theorem~\ref{thm:minimax} is valid
only in the regime $\lambda \asymp n^{\frac{1}{d + 2\beta}}$,
and so $\lambda$ requires careful determination in practice.
For clarity, in this section we use the notation
$\hat\mu_\rd(x; \lambda, J)$ for the debiased Mondrian random forest estimator
implemented with lifetime $\lambda$ and debiasing order $J$,
as in \eqref{eq:debiased}.
Similarly, we write $\hat\Sigma_\rd(x; \lambda, J)$ for the
associated variance estimator \eqref{eq:Sigma_hat_d}.

For minimax-optimal point estimation when $\beta$ is known
(for example, when the data come from noisy measurements
of a smooth physical system),
choose any sequence $\lambda \asymp n^{\frac{1}{d + 2\beta}}$
and use $\hat\mu_\rd(x; \lambda, J)$ with
$J = \lfloor \flbeta / 2 \rfloor$,
following the theory given in Theorem~\ref{thm:minimax}.
For an explicit example of how to choose the lifetime,
one can instead use
$\hat\mu_\rd\big(x; \hat\lambda_{J-1}, J-1\big)$
so that the leading bias is explicitly characterized by
Lemma~\ref{lem:bias_debiased},
and with $\hat\lambda_{J-1}$ as defined below.
This estimator is however not minimax-optimal as
the debiasing order of $J-1 < J$
does not satisfy the conditions of Theorem~\ref{thm:minimax}.

For performing inference, a more careful procedure is required;
we suggest the following, when $\beta > 2$ is known.
Set $J = \lfloor \flbeta / 2 \rfloor$ as before,
and use $\hat\mu_\rd\big(x; \hat\lambda_{J-1}, J\big)$
and $\hat\Sigma_\rd\big(x; \hat\lambda_{J-1}, J\big)$
to construct a confidence interval \eqref{eq:CI_d},
so that one selects a lifetime tailored
for a more biased estimator than that which is actually used.
This results in an inflated lifetime estimate, guaranteeing
the resulting bias is negligible when it is
plugged into the fully debiased estimator.
This approach to tuning parameter selection
and debiasing for nonparametric
inference corresponds to an application of robust bias correction
\citep{calonico2018jasa,calonico2022bernoulli},
where the point estimator is bias-corrected
and the robust standard error estimator incorporates the additional
variability introduced by the correction.
This gives a refined distributional approximation
but may not exhaust the underlying
smoothness of the regression function.
An alternative approach based on Lepski's method
\citep{lepski1992asymptotically,birge2001alternative}
could be developed with the latter goal in mind,
when $\beta$ is unknown.

It remains to propose a concrete method for computing $\hat\lambda_{J}$ in
finite samples; we suggest a procedure based on minimizing the
asymptotic mean squared error (AMSE) using plug-in selection with
polynomial estimation, building on classical ideas from the nonparametric
smoothing literature.
Expressions for the AMSE are available as direct consequences of
Lemmas~\ref{lem:bias_debiased} and~\ref{lem:variance_debiased},
provided that $J < \lfloor \flbeta/2 \rfloor$
so the H{\"o}lder smoothness is not fully exhausted.

\subsubsection*{Selecting the lifetime parameter
\texorpdfstring{$\lambda$}{lambda} with polynomial estimation}%

For implementation, we propose a simple rule-of-thumb approach. Suppose that
$X_i \sim \Unif\big([0,1]^d\big)$
and that the leading bias of $\hat\mu_\rd(x)$ is well approximated by an
additively separable function so that,
writing $\partial^{2 J + 2}_j \mu(x)$
for $\partial^{2 J + 2}_j \mu(x) / \partial x_j^{2 J + 2}$,
the asymptotic bias is
\begin{align*}
  \ABias(x; \lambda, J)
  = \frac{\bar \omega B_{J+1}(x)}{\lambda^{2 J + 2}}
  = \frac{1}{\lambda^{2 J + 2}}
  \frac{\bar \omega }{J + 2}
  \sum_{j=1}^d
  \partial^{2 J + 2}_j \mu(x).
\end{align*}
Suppose that the model is homoscedastic so
$\sigma^2(x) = \sigma^2$ and
the asymptotic variance of $\hat\mu_\rd$ is
\begin{align*}
  \AVar(x; \lambda, J)
  = \frac{\lambda^d}{n}
  \Sigma_\rd(x)
  &=
  \frac{\lambda^d \sigma^2}{n}
  \sum_{r=0}^{J}
  \sum_{r'=0}^{J}
  \omega_r
  \omega_{r'}
  \left( \ell_{r r'} + \ell_{r' r} \right)^d.
\end{align*}
The asymptotic mean squared error is therefore
\begin{align*}
  \AMSE(x; \lambda, J)
  &=
  \frac{1}{\lambda^{4 J + 4}}
  \frac{\bar \omega^2}{(J + 2)^2}
  \left(
    \sum_{j=1}^d
    \partial^{2 J + 2}_j \mu(x)
  \right)^2
  \hspace*{-1mm}
  + \frac{\lambda^d \sigma^2}{n}
  \sum_{r=0}^{J}
  \sum_{r'=0}^{J}
  \omega_r
  \omega_{r'}
  \left( \ell_{r r'} + \ell_{r' r} \right)^d.
\end{align*}
Minimizing over $\lambda > 0$ yields the AMSE-optimal lifetime parameter
\begin{align*}
  \lambda_{J}
  &=
  \left(
    \frac{
      \frac{(4 J + 4) \bar \omega^2}{(J + 2)^2}
      n \left(
        \sum_{j=1}^d
        \partial^{2 J + 2}_j \mu(x)
      \right)^2
    }{
      d \sigma^2
      \sum_{r=0}^{J}
      \sum_{r'=0}^{J}
      \omega_r
      \omega_{r'}
      \left( \ell_{r r'} + \ell_{r' r} \right)^d
    }
  \right)^{\frac{1}{4 J + 4 + d}}.
\end{align*}
An estimator of $\lambda_J$ is given by the plug-in procedure
\begin{align*}
  \hat\lambda_J
  &=
  \left(
    \frac{
      \frac{(4 J + 4) \bar \omega^2}{(J + 2)^2}
      n \left(
        \sum_{j=1}^d
        \partial^{2 J + 2}_j \hat\mu(x)
      \right)^2
    }{
      d \hat\sigma^2
      \sum_{r=0}^{J}
      \sum_{r'=0}^{J}
      \omega_r
      \omega_{r'}
      \left( \ell_{r r'} + \ell_{r' r} \right)^d
    }
  \right)^{\frac{1}{4 J + 4 + d}}
\end{align*}
for some preliminary estimators
$\partial^{2 J + 2}_j \hat\mu(x)$ and $\hat\sigma^2$.
These can be obtained by fitting a global polynomial regression
to the data $(\bX, \bY)$ of order $2 J + 4$ without interaction terms.
To do this, define the $n \times ((2 J + 4)d + 1)$ design matrix
$\bP$ with rows given by
\begin{align*}
  P(X_i) = \left(
    1, X_{i1}, X_{i1}^2, \ldots, X_{i1}^{2 J + 4},
    X_{i2}, X_{i2}^2, \ldots, X_{i2}^{2 J + 4},
    \ldots,
    X_{id}, X_{id}^2, \ldots, X_{id}^{2 J + 4}
  \right).
\end{align*}
Then the derivative estimator is
\begin{align*}
  \partial^{2 J + 2}_j \hat\mu(x)
  &=
  \partial^{2 J + 2}_j P(x)
  \big( \bP^\T \bP \big)^{-1} \bP^\T \bY \\
  \nonumber
  &=
  (2J + 2)!
  \left(
    0_{1 + (j-1)(2 J + 4) + (2J + 1)},
    1, x_j, x_j^2 / 2,
    0_{(d-j)(2 J + 4)}
  \right)
  \big( \bP^\T \bP \big)^{-1} \bP^\T \bY,
\end{align*}
and the variance estimator $\hat\sigma^2$ is
the based on the residual sum of squared errors of this model:
\begin{align*}
  \hat\sigma^2
  &=
  \frac{1}{n - (2J + 4)d - 1}
  \big(
    \bY^\T \bY
    - \bY^\T \bP \big( \bP^\T \bP \big)^{-1} \bP^\T \bY
  \big).
\end{align*}

\subsubsection*{Choosing the number \texorpdfstring{$B$}{B} of trees
in each forest}%

The next parameter to choose is the number of trees in each forest.
If no debiasing is applied, we suggest taking
$B \asymp \sqrt{n}$ to ensure the coverage error in
Theorem~\ref{thm:confidence} converges to zero.
If debiasing is used then we recommend setting
$B \asymp n^{\frac{2J-1}{2J}}$,
consistent with Theorem~\ref{thm:minimax} and
Theorem~\ref{thm:confidence_debiased}.

\subsubsection*{Setting the debiasing order \texorpdfstring{$J$}{J}}%

Deciding how many orders of bias to remove
requires knowledge of the H{\"o}lder smoothness of $\mu$ and $f$,
which is in practice very difficult to estimate statistically.
As such we recommend removing only the first one or two
bias terms, taking $J \in \{0,1,2\}$
to avoid inflating the variance of the estimator.

\subsubsection*{Selecting the debiasing scalars \texorpdfstring{$a_r$}{ar}}%

As in Section~\ref{sec:debiased},
take a fixed geometric or arithmetic sequence.
For example, $a_r = (1+\gamma)^r$
or $a_r = 1 + \gamma r$ where $\gamma > 0$;
we suggest $a_r = (3/2)^r$.

\subsection{Empirical results}

To demonstrate the empirical properties of our proposed
estimation and inference methodology, we present results with simulated
data. Throughout this section we use the data generating process given by
uniform covariates $X_i \sim \Unif[0, 1]^d$ for $d \in \{1, 2\}$,
a sinusoidal regression function
$\mu(x) = \sum_{j=1}^{d} \sin(\pi x_j)$, and homoscedastic normal errors
$\varepsilon_i \sim \cN(0, \sigma^2)$ with $\sigma = 3/10$.
We thus ensure that the regression function is smooth
and appropriately bounded but is not a polynomial,
so the bias terms given in Lemma~\ref{lem:bias} do not vanish.
The additive structure is not essential; similar performance is expected
for more general smooth models with dependence
between the covariates, assuming a sufficiently large sample size
and densities bounded away from zero.
We focus on estimation at the design point
$x = \left(1/2, \ldots, 1/2\right) \in \R^d$, and use
$n=1000$ samples and $B=800$ trees in each forest.
We demonstrate our procedures both
with and without debiasing by setting $J = 1$ and $J = 0$ respectively,
and when $J = 1$ we use the debiasing scalars $(a_1, a_2) = (1, 3/2)$
suggested in Section~\ref{sec:debiased}, yielding debiasing coefficients of
$(\omega_0, \omega_1) = (-4/5, 9/5)$.
For lifetime selection (LS),
we first show our estimator $\hat \lambda_J$
based on polynomial regression
(Section~\ref{sec:parameter_selection}), and then present
the infeasible oracle lifetime $\lambda_J$
which exactly minimizes the asymptotic mean squared error.
To illustrate robustness with respect to this tuning parameter, we repeat the
same experiments but rescaling $\lambda_J$ by a lifetime multiplier
$\LM \in \{1 \pm \ell/5 : 0 \leq \ell \leq 2\}$.
We further exhibit the robust bias correction (BC) approach
discussed in Section~\ref{sec:parameter_selection} by using a debiased
estimator ($J=1$) with the AMSE-optimal lifetime parameter $\lambda_0$.

For each such estimator, we present the
empirical root mean squared error (RMSE),
bias, standard deviation (SD),
and absolute bias/SD ratio, based on $3000$ repeats.
We also show the estimated standard deviation
$\widehat\SD = \sqrt{\lambda^d/n}\, \hat\Sigma_\rd(x)$,
as well as the estimated variance of the errors
$\hat\sigma^2(x)$. Since oracle properties available,
we give the asymptotic bias
(ABias) and asymptotic standard deviation (ASD).
Finally, we present the empirical coverage rate (CR) of
nominal $95\%$ confidence intervals
along with their empirical average widths (CIW).

We now derive the asymptotic oracle properties of our estimators.
Firstly, by Lemma~\ref{lem:variance},
the asymptotic variance of the estimator without debiasing is
\begin{align*}
  \AVar
  = \frac{\lambda^d}{n}
  \Sigma(x)
  &=
  \frac{\lambda^d}{n}
  \frac{\sigma^2(x)}{f(x)}
  \left(
    \frac{4 - 4 \log 2}{3}
  \right)^d
  \approx \frac{\lambda^d \sigma^2}{n} 0.4091^d.
\end{align*}
By Lemma~\ref{lem:variance_debiased},
$\ell_{0 0} = \frac{2}{3} \left( 1 - \log 2 \right)$,
$\ell_{0 1} = \frac{2}{3} \left( 1 - \frac{2}{3} \log \frac{5}{2} \right)$,
$\ell_{1 0} = 1 - \frac{3}{2} \log \frac{5}{3}$,
and $\ell_{1 1} = 1 - \log 2$, so
the asymptotic variance of its debiased counterpart is
\begin{align*}
  \AVar_\rd
  &=
  \frac{\lambda^d}{n}
  \Sigma_\rd(x)
  =
  \frac{\lambda^d}{n}
  \frac{\sigma^2(x)}{f(x)}
  \sum_{r=0}^{J}
  \sum_{r'=0}^{J}
  \omega_r
  \omega_{r'}
  \left( \ell_{r r'} + \ell_{r' r} \right)^d \\
  &\approx
  \frac{\lambda^d \sigma^2}{n}
  \left(
    0.64 \cdot 0.4091^d
    - 2.88 \cdot 0.4932^d
    + 3.24 \cdot 0.6137^d
  \right).
\end{align*}
We similarly establish the asymptotic biases.
Without debiasing, by Lemma~\ref{lem:bias},
\begin{align*}
  \ABias
  &=
  \frac{1}{\lambda^{2}}
  \sum_{|\nu|=1}
  \partial^{2 \nu} \mu(x)
  \prod_{j=1}^d
  \frac{1}{\nu_j + 1}
  =
  \frac{1}{2 \lambda^2}
  \sum_{j=1}^d
  \frac{\partial^2 \mu(x)}{\partial x_j^2}
  =
  - \frac{\pi^2}{2 \lambda^2}
  \sum_{j=1}^d
  \sin(\pi x_j)
  =
  - \frac{\pi^2 d}{2 \lambda^2}.
\end{align*}
For the debiased estimator, with
$\bar\omega = \omega_0 + \omega_1 a_1^{-4}
= -4/9$,
we recover
\begin{align*}
  \ABias_\rd
  &=
  \frac{\bar\omega B_{2}(x)}{\lambda^{4}}
  =
  -\frac{4}{9 \lambda^{4}}
  \sum_{|\nu|=2}
  \partial^{2 \nu} \mu(x)
  \prod_{j=1}^d
  \frac{1}{\nu_j + 1}
  =
  -\frac{4}{27 \lambda^{4}}
  \sum_{j=1}^d
  \frac{\partial^{4} \mu(x)}{\partial^4 x_j}
  =
  -\frac{4 \pi^4 d}{27 \lambda^{4}}.
\end{align*}

Table~\ref{tab:d1_n1000_B800} gives results in the one-dimensional
setting ($d=1$). Firstly, observe that the
polynomial lifetime estimator appears to be moderately
accurate, displaying some oversmoothing when fitting a
polynomial of order $4$ (for $\hat\lambda_0$)
and some undersmoothing with a polynomial of order $6$
(with $\hat\lambda_1$).
The effects of debiasing on RMSE are clear, with
the appropriately tuned debiased Mondrian forest
($J=1$, $\lambda_1$) providing the best results
(Theorem~\ref{thm:minimax}).
Likewise, the effect of debiasing is apparent when
using an undersmoothed lifetime
($J=1$, $\lambda_0$), with the bias being significantly
reduced (see Lemma~\ref{lem:bias_debiased})
at the expense of a larger standard deviation.
The variance estimator performs well,
with $\widehat\SD$ a good approximation for the finite-sample
SD, and $\hat\sigma^2$ similarly sits close to $\sigma^2 = 0.09$.
The value of robust bias correction ($J=1$, $\lambda_0$)
for statistical inference is clear, with the coverage rates
clustering around the nominal $95\%$ even with perturbed
lifetime values (see Theorem~\ref{thm:confidence_debiased}).
In contrast, the no-debiasing estimator
($J=0$, $\lambda_0$) fails to attain correct coverage,
while its fully debiased counterpart
($J=1$,~$\lambda_1$) lacks robustness,
reaching the nominal level only with larger lifetime values.
Accurate coverage is at the expense of wider
confidence intervals, but the differences are not large.

\begin{table}[H]
  \begin{center}
    \begin{scriptsize}
\begin{tabular}{|c|cc|cc|cccc|cc|ccc|cc|}
  \hline
  & $J$ & LS & LM & $\lambda$ & RMSE & Bias & SD & Bias/SD &
  $\SDhattab$ & $\hat\sigma^2$ & ARMSE &
  ABias & ASD & CR & CIW \\
  \hline
  \multirow{6}{*}{\rotatebox{90}{No debiasing}}&0& $\hat\lambda_{0}$& 1.0&
  14.73& 0.0351& -0.0250& 0.0247& 1.0123& 0.0236& 0.0931& 0.0369& -0.0270&
  0.0232& 82.5\%& 0.093\\
  && $\lambda_{0}$& 1.2& 23.10& 0.0307& -0.0092& 0.0293& 0.3130& 0.0292&
  0.0894& 0.0306& -0.0092& 0.0292& 93.6\%& 0.114\\
  &&& 1.1& 21.18& 0.0300& -0.0109& 0.0280& 0.3888& 0.0280& 0.0897& 0.0300&
  -0.0110& 0.0279& 93.4\%& 0.110\\
  &&& 1.0& 19.25& 0.0297& -0.0131& 0.0267& 0.4909& 0.0267& 0.0901& 0.0298&
  -0.0133& 0.0266& 92.9\%& 0.105\\
  &&& 0.9& 17.33& 0.0300& -0.0160& 0.0253& 0.6326& 0.0254& 0.0907& 0.0301&
  -0.0164& 0.0253& 90.7\%& 0.100\\
  &&& 0.8& 15.40& 0.0312& -0.0201& 0.0238& 0.8438& 0.0241& 0.0916& 0.0316&
  -0.0208& 0.0238& 87.5\%& 0.095\\
  \hline
  \multirow{6}{*}{\rotatebox{90}{Debiasing}}&1& $\hat\lambda_{1}$& 1.0& 11.14&
  0.0301& -0.0031& 0.0300& 0.1036& 0.0302& 0.1002& 0.0296& -0.0026& 0.0287&
  95.0\%& 0.119\\
  && $\lambda_{1}$& 1.2& 7.86& 0.0255& -0.0070& 0.0246& 0.2835& 0.0269& 0.1103&
  0.0245& -0.0038& 0.0242& 95.9\%& 0.106\\
  &&& 1.1& 7.21& 0.0255& -0.0095& 0.0236& 0.4031& 0.0263& 0.1147& 0.0238&
  -0.0053& 0.0232& 95.2\%& 0.103\\
  &&& 1.0& 6.55& 0.0264& -0.0135& 0.0227& 0.5950& 0.0256& 0.1198& 0.0235&
  -0.0078& 0.0221& 94.3\%& 0.100\\
  &&& 0.9& 5.90& 0.0288& -0.0191& 0.0216& 0.8817& 0.0249& 0.1259& 0.0241&
  -0.0119& 0.0210& 90.6\%& 0.097\\
  &&& 0.8& 5.24& 0.0343& -0.0274& 0.0206& 1.3346& 0.0240& 0.1329& 0.0275&
  -0.0191& 0.0198& 82.0\%& 0.094\\
  \hline
  \multirow{6}{*}{\rotatebox{90}{Robust BC}}&1& $\hat\lambda_{0}$& 1.0& 14.73&
  0.0334& -0.0014& 0.0333& 0.0405& 0.0339& 0.0940& 0.0336& -0.0011& 0.0330&
  95.3\%& 0.133\\
  && $\lambda_{0}$& 1.2& 23.10& 0.0420& -0.0004& 0.0420& 0.0105& 0.0419&
  0.0898& 0.0415& -0.0001& 0.0415& 94.8\%& 0.164\\
  &&& 1.1& 21.18& 0.0401& -0.0003& 0.0401& 0.0078& 0.0402& 0.0901& 0.0398&
  -0.0001& 0.0398& 95.0\%& 0.158\\
  &&& 1.0& 19.25& 0.0381& -0.0004& 0.0381& 0.0115& 0.0383& 0.0905& 0.0379&
  -0.0001& 0.0379& 94.7\%& 0.150\\
  &&& 0.9& 17.33& 0.0362& -0.0003& 0.0362& 0.0084& 0.0365& 0.0912& 0.0360&
  -0.0002& 0.0360& 95.0\%& 0.143\\
  &&& 0.8& 15.40& 0.0341& -0.0005& 0.0341& 0.0139& 0.0346& 0.0922& 0.0339&
  -0.0003& 0.0339& 95.3\%& 0.136\\
  \hline
\end{tabular}
     \end{scriptsize}
  \end{center}
  \vspace*{-3mm}
  \caption{Simulation results with
  $d=1$, $n=1000$, and $B=800$, over $3000$ repeats}
  \label{tab:d1_n1000_B800}
  \vspace*{-2mm}
\end{table}

Table~\ref{tab:d2_n1000_B800} presents analogous results in the
two-dimensional setting ($d=2$).
The debiased estimator ($J=1$, $\lambda_1$) again achieves
the best RMSE, and the undersmoothed estimator
($J=1$, $\lambda_0$)
similarly displays the smallest bias/SD ratio.
Coverage rates are generally worse than in Table~\ref{tab:d1_n1000_B800},
mostly due to the increased difficulty posed by the curse
of dimensionality and a reduced effective sample size.
Nonetheless, inference based on robust bias correction continues to exhibit a
pronounced improvement in coverage when compared to standard non-debiased
methods, and again shows a moderate increase in confidence interval widths.

\vspace*{2mm}
\begin{table}[ht]
  \begin{center}
    \begin{scriptsize}
\begin{tabular}{|c|cc|cc|cccc|cc|ccc|cc|}
  \hline
  & $J$ & LS & LM & $\lambda$ & RMSE & Bias & SD & Bias/SD &
  $\SDhattab$ & $\hat\sigma^2$ & ARMSE &
  ABias & ASD & CR & CIW \\
  \hline
  \multirow{6}{*}{\rotatebox{90}{No debiasing}}&0& $\hat\lambda_{0}$& 1.0&
  12.35& 0.0805& -0.0646& 0.0481& 1.3432& 0.0481& 0.0989& 0.0828& -0.0666&
  0.0479& 71.1\%& 0.189\\
  && $\lambda_{0}$& 1.2& 18.39& 0.0758& -0.0310& 0.0692& 0.4481& 0.0627&
  0.0882& 0.0771& -0.0292& 0.0714& 88.2\%& 0.246\\
  &&& 1.1& 16.85& 0.0735& -0.0361& 0.0640& 0.5650& 0.0593& 0.0898& 0.0741&
  -0.0347& 0.0654& 87.0\%& 0.233\\
  &&& 1.0& 15.32& 0.0726& -0.0427& 0.0587& 0.7280& 0.0558& 0.0919& 0.0728&
  -0.0420& 0.0595& 84.9\%& 0.219\\
  &&& 0.9& 13.79& 0.0743& -0.0518& 0.0532& 0.9740& 0.0520& 0.0947& 0.0746&
  -0.0519& 0.0535& 80.8\%& 0.204\\
  &&& 0.8& 12.26& 0.0796& -0.0637& 0.0477& 1.3346& 0.0478& 0.0985& 0.0811&
  -0.0657& 0.0476& 71.6\%& 0.188\\
  \hline
  \multirow{6}{*}{\rotatebox{90}{Debiasing}}&1& $\hat\lambda_{1}$& 1.0& 9.20&
  0.0726& -0.0144& 0.0712& 0.2020& 0.0746& 0.1277& 0.0723& -0.0086& 0.0691&
  95.1\%& 0.292\\
  && $\lambda_{1}$& 1.2& 7.18& 0.0584& -0.0217& 0.0542& 0.3999& 0.0672& 0.1490&
  0.0550& -0.0108& 0.0540& 96.3\%& 0.263\\
  &&& 1.1& 6.58& 0.0577& -0.0283& 0.0503& 0.5620& 0.0644& 0.1602& 0.0518&
  -0.0154& 0.0495& 95.8\%& 0.252\\
  &&& 1.0& 5.99& 0.0596& -0.0381& 0.0459& 0.8299& 0.0613& 0.1733& 0.0503&
  -0.0225& 0.0450& 94.0\%& 0.240\\
  &&& 0.9& 5.39& 0.0664& -0.0516& 0.0418& 1.2332& 0.0578& 0.1879& 0.0530&
  -0.0343& 0.0405& 90.4\%& 0.227\\
  &&& 0.8& 4.79& 0.0797& -0.0704& 0.0373& 1.8873& 0.0538& 0.2044& 0.0656&
  -0.0549& 0.0360& 79.6\%& 0.211\\
  \hline
  \multirow{6}{*}{\rotatebox{90}{Robust BC}}&1& $\hat\lambda_{0}$& 1.0& 12.35&
  0.0889& -0.0053& 0.0888& 0.0598& 0.0854& 0.1047& 0.0928& -0.0014& 0.0927&
  94.9\%& 0.335\\
  && $\lambda_{0}$& 1.2& 18.39& 0.1208& -0.0032& 0.1208& 0.0265& 0.0971&
  0.0925& 0.1381& -0.0003& 0.1381& 89.4\%& 0.380\\
  &&& 1.1& 16.85& 0.1135& -0.0040& 0.1134& 0.0351& 0.0953& 0.0941& 0.1266&
  -0.0004& 0.1266& 90.8\%& 0.373\\
  &&& 1.0& 15.32& 0.1056& -0.0039& 0.1055& 0.0367& 0.0927& 0.0964& 0.1151&
  -0.0005& 0.1151& 92.6\%& 0.363\\
  &&& 0.9& 13.79& 0.0974& -0.0042& 0.0973& 0.0427& 0.0893& 0.0994& 0.1036&
  -0.0008& 0.1036& 93.8\%& 0.350\\
  &&& 0.8& 12.26& 0.0883& -0.0047& 0.0882& 0.0532& 0.0853& 0.1041& 0.0921&
  -0.0013& 0.0921& 94.9\%& 0.334\\
  \hline
\end{tabular}
     \end{scriptsize}
  \end{center}
  \vspace*{-2mm}
  \caption{Simulation results with
  $d=2$, $n=1000$, and $B=800$, over $3000$ repeats}
  \label{tab:d2_n1000_B800}
\end{table}

\section{Computational complexity and application to online learning}%
\label{sec:compute}

We discuss some computational aspects of
(debiased) Mondrian random forests. We firstly consider the batch setting,
where all of the data is available simultaneously, and secondly investigate
the online regime, where data arrives sequentially and the model must be
incrementally updated \citep{lakshminarayanan2014mondrian}.
Mondrian random forests have several properties that make them well suited
for online learning:
(i) in \citet{mourtada2021amf} it was shown that some
online Mondrian forest variants maintain statistical consistency, achieving
the same asymptotic error rates as their batch counterparts under certain
conditions;
(ii) as we will demonstrate (Lemma~\ref{lem:compute_online}),
online Mondrian forest algorithms exploiting the Markov property
of the Mondrian process are computationally efficient,
therefore scaling to large streaming datasets; and
(iii) the random nature of splits in
Mondrian trees allows the forest to naturally adapt to changes in the
underlying data distribution over time (concept drift), without requiring
explicit drift detection or model reset mechanisms.

Some potential applications of online Mondrian forests with uncertainty
quantification include
real-time prediction and monitoring in industrial processes
\citep{gomes2017adaptive},
adaptive pricing and recommendation systems
\citep{krauss2017deep,li2018multi},
online anomaly detection with confidence levels
\citep{martindale2020ensemble},
and streaming data analysis for the natural sciences
\citep{abdulsalam2010classification}.

The inference procedures developed in this
paper extend to the online setting, allowing for uncertainty quantification in
streaming data applications. However, care must be taken in situations where
the underlying distribution may change over time, or where
validity of the inferential
procedures is required to hold uniformly over the data arrival times.
Developing rigorous statistical inference tools for online
Mondrian forests in those more complicated time-dependent regimes is an
interesting direction for future
work.

The core of our computational approaches for batch and online learning
comprise several main ideas; these enable substantial improvements over
naive algorithms based on the equations presented in previous sections.
The first of these is to keep track of which data points are
``local'' to the evaluation point $x$, according to the forest
$(T_{b r}(x) : 1 \leq b \leq B, 0 \leq r \leq J)$.
Define the \emph{union cell} $U(x) \subseteq [0,1]^d$
and \emph{active indices} $I(x) \subseteq \{1, \ldots, n\}$ by
\begin{align}
  \label{eq:compute_local}
  U(x)
  &= \prod_{j=1}^d \bigcup_{b=1}^B \bigcup_{r=0}^J T_{b r}(x)_j
  &&\text{and}
  &I(x)
  &= \{1 \leq i \leq n : X_i \in U(x)\}
\end{align}
respectively,
noting that any data point contributing to $\hat\mu_\rd(x)$
or $\hat\Sigma_\rd(x)$ satisfies $X_i \in U(x)$ and $i \in I(x)$.
As the lifetime parameter $\lambda$ grows, the volume of $U(x)$
and the proportion of contributing samples
$|I(x)| / n$ both converge to zero in expectation,
lowering the effective sample size
and significantly decreasing the amount of computation necessary.
Further, $U(x)$ can be efficiently computed with a divide-and-conquer
approach whenever multiple parallel processors are available.

The second main idea is to observe that the estimators $\hat\mu_\rd(x)$
and $\hat\sigma^2(x)$ can be expressed as ratios of sums.
More precisely, firstly define
\begin{gather}
  \nonumber
  N_{b r}(x)
  = \sum_{i \in I(x)} \I\{X_i \in T_{b r}(x)\},\qquad   S_{b r}(x)
  = \sum_{i \in I(x)} Y_i \, \I\{X_i \in T_{b r}(x)\}, \\
  \label{eq:compute_sums}
  V_{b r}(x)
  = \sum_{i \in I(x)} Y_i^2 \, \I\{X_i \in T_{b r}(x)\},
\end{gather}
which are efficient to update as new samples arrive;
furthermore, they can be computed separately for each
$b$ and $r$ in parallel. Then,
\begin{align}
  \label{eq:compute_mu_sigma2}
  \hat \mu_\rd(x)
  &= \sum_{r=0}^{J} \omega_r \frac{1}{B} \sum_{b=1}^B
  \frac{S_{b r}(x)} {N_{b r}(x)}
  &&\text{and}
  &\hat\sigma^2(x)
  &= \frac{1}{B} \sum_{b=1}^{B} \frac{V_{b 0}(x)}{N_{b 0}(x)}.
\end{align}

The third observation is that the estimators depend on the trees
only through the cell $T_{b r}(x)$. Since \citet{mourtada2020minimax}
characterize the exact distribution of this quantity, it can be sampled
without needing to grow an entire Mondrian tree. Further, the memoryless
property of the exponential distribution
(and thus also of the Mondrian process) means that in the online setting,
only a small fraction of the cells typically need to be updated.

The fourth and final concept is to avoid fitting the relatively computationally
expensive $\hat\Sigma_\rd(x)$ too often. This estimator does not readily admit
a ``ratio of sums'' formulation, and hence is not efficient to update
incrementally. Our recommendation is to instead only update this term
after $K \geq 1$
new data points have arrived on average. Note however that using the active
indices $I(x)$ still permits an improvement over the naive approach, since
\begin{align}
  \label{eq:compute_Sigma_hat}
  \hat\Sigma_\rd(x)
  &=
  \hat\sigma^2(x)
  \frac{n}{\lambda^d}
  \sum_{i \in I(x)}
  \left(
    \sum_{r=0}^J
    \omega_r
    \frac{1}{B}
    \sum_{b=1}^B
    \frac{\I\{X_i \in T_{b r}(x)\}}
    {N_{b r}(x)}
  \right)^2.
\end{align}

Before discussing the online learning setting in more detail,
we present our efficient procedure for batch estimation and inference
in Algorithm~\ref{alg:batch}.

\begin{algorithm}[ht]
  \small
  \iftoggle{journal}{\setstretch{1.3}}{\setstretch{1.2}}
  \caption{Batch learning with Mondrian random forests}
  \label{alg:batch}

  \KwIn{%
    Data $(X_i, Y_i)$ for $1 \leq i \leq n$,
    forest size $B \geq 1$,
    debiasing order $J \geq 0$.
  }

  Select $\lambda$ using one of the methods from
  Section~\ref{sec:parameter_selection}. \\

  Construct the union cell $U(x)$ and active indices $I(x)$
  as in \eqref{eq:compute_local}. \\

  Calculate $N_{b r}(x)$, $S_{b r}(x)$ and $V_{b r}(x)$
  for each $1 \leq b \leq B$ and $0 \leq r \leq J$
  as in \eqref{eq:compute_sums}. \\

  Compute $\hat\mu_\rd(x)$ and $\hat\sigma^2(x)$
  with \eqref{eq:compute_mu_sigma2}. \\

  Calculate $\hat\Sigma_\rd(x)$ and $\CI_\rd(x)$ using
  \eqref{eq:compute_Sigma_hat} and \eqref{eq:CI_d} respectively.

\end{algorithm}
\iftoggle{journal}{\vspace*{-3mm}}{}

The following lemma bounds the average case
time complexity of our batch learning procedure
(Algorithm~\ref{alg:batch}),
under the same assumptions made
throughout the paper.

\begin{lemma}[Computational complexity of batch learning]%
  \label{lem:compute_batch}
  Suppose Assumptions~\ref{ass:data}
  and~\ref{ass:estimator} hold.
  Then the average case time complexity of Algorithm~\ref{alg:batch} is
  \begin{align*}
    \E \left[ \mathcal{T}_{\mathrm b} \right]
    \lesssim
    d (J+1)
    \big(
      n d (J+1) + B
    \big)
    + \frac{n B d (J+1) \log (2 B (J+1))^d}{\lambda^d}.
  \end{align*}
\end{lemma}

We now turn to the online learning setting, making the following assumptions.
Firstly, suppose that
a (debiased) Mondrian random forest with $B$ trees has already been
fitted to a data set with $n$ samples, using a lifetime of $\lambda$, and that
this original data set is still available. Assume that the union cell $U(x)$,
the index set $I(x)$, and the point estimates
$\hat\mu_\rd(x)$ and $\hat\Sigma_\rd(x)$
have been computed, as well as the trees $T_{b r}(x)$ and the
quantities $S_{b r}(x)$, $N_{b r}(x)$ and $V_{b r}(x)$ for $1 \leq b \leq B$
and $0 \leq r \leq J$. A~new data set with $k$ samples then arrives,
where $1 \leq k \leq n$, and we must
produce updated estimates $\hat\mu^*_\rd(x)$, $\hat\Sigma_\rd^*(x)$
and $\CI_\rd^*(x)$ based on all $n + k$ samples.
Our randomized procedure for doing this is
described below, using a star to indicate
updated quantities, and summarized in Algorithm~\ref{alg:online}.

The new sample size is $n^* = n+k$, so
the first step is to update $B$.
As recommended in Section~\ref{sec:implementation}, we take
$B \asymp n^\xi$ for some $\xi \in (0, 1)$; therefore set
$B^* = \big\lfloor (n+k)^\xi B / n^\xi \big\rfloor$.
Next, we update the lifetime parameter $\lambda$. To avoid excessive
computation, we suggest the following:
with probability $1 \wedge (k/K)$, use the methods from
Section~\ref{sec:parameter_selection} to compute a new lifetime parameter
$\lambda^* \geq \lambda$ using all of the data.
Otherwise, note that $\lambda \asymp n^\zeta$ for some $\zeta \in (0, 1/d)$
(for example $\zeta = 1 / (d + 2\beta)$
under the conditions of Theorem~\ref{thm:minimax})
and set $\lambda^* = (n+k)^\zeta \lambda / n^\zeta$.
Next, to update the trees $T_{b r}(x)$,
sample $E_{b r j 1}$ and $E_{b r j 2}$ i.i.d.\ $\Exp(1)$, and set
\begin{align}
  \label{eq:update_trees}
  T_{b r}^*(x)_j^-
  &= T_{b r}(x)_j^- \vee
  \left(x_j - \frac{E_{b r j 1}}{\lambda^* - \lambda}\right),
  &T_{b r}^*(x)_j^+
  &= T_{b r}(x)_j^+ \wedge
  \left(x_j + \frac{E_{b r j 2}}{\lambda^* - \lambda}\right).
\end{align}
Since $B^* \geq B$, we also generate new trees
$T_{b r}^*(x)$ for $B + 1 \leq b \leq B^*$ and $0 \leq r \leq J$ using
\begin{align}
  \label{eq:compute_new_trees}
  T_{b r}^*(x)_j^-
  &= 0 \lor \left( x_j - \frac{E_{b r j 1}}{a_r \lambda}\right),
  & T_{b r}^*(x)_j^+
  &= 1 \land \left( x_j + \frac{E_{b r j 2}}{a_r \lambda}\right).
\end{align}
Computing $U^*(x)$ is simple, applying
\eqref{eq:compute_local} to $T_{b r}^*(x)$.
To update $I(x)$, set
\begin{align}
  \label{eq:compute_indices}
  I^*(x) =
  \begin{cases}
    \{i \in I(x) \cup \{n+1, \ldots, n+k\} : X_i \in U^*(x)\}
    &\text{if }
    U^*(x) \subseteq U(x), \\
    \{1 \leq i \leq n + k : X_i \in U^*(x)\}
    & \text{otherwise}.
  \end{cases}
\end{align}
For $N_{b r}(x)$, and analogously for $S_{b r}(x)$ and $V_{b r}(x)$,
apply the following method:
\begin{align}
  \label{eq:compute_update_sums}
  N_{b r}^*(x)
  &=
  \begin{cases}
    N_{b r}(x) + \sum_{i \in V^*(x), i > n} \I\{X_i \in T_b(x)\}
    &\text{if }
    b \leq B \text{ and } T_{b r}(x) = T_{b r}^*(x) \\
    \sum_{i \in I^*(x)} \I\{X_i \in T_{b r}^*(x)\}
    & \text{otherwise}.
  \end{cases}
\end{align}

Finally, $\hat\mu_\rd^*(x)$ and $\hat\sigma^{2*}(x)$ are computed using
\eqref{eq:compute_mu_sigma2}. With probability
$1 \wedge (k/K)$, recalculate $\hat\Sigma_\rd^*(x)$
with \eqref{eq:compute_Sigma_hat}; otherwise set
$\hat\Sigma_\rd^*(x) = \hat\Sigma_\rd(x)$.
The confidence interval $\CI_\rd^*(x)$ can then
be constructed with \eqref{eq:CI_d}.
The following algorithm summarizes our online methodology.

\begin{algorithm}[ht]
  \small
  \iftoggle{journal}{\setstretch{1.3}}{\setstretch{1.2}}
  \caption{Online learning with Mondrian random forests}
  \label{alg:online}

  \KwIn{%
    Data $(X_i, Y_i)$ for $1 \leq i \leq n$,
    forest size $B \geq 1$, debiasing order $J \geq 0$,
    lifetime $\lambda$,
    forest exponent $\xi \in (0, 1)$, lifetime exponent $\zeta \in (0, 1/d)$,
    active region $U(x)$, active indices $I(x)$,
    trees $T_{b r}(x)$ and $N_{b r}(x)$, $S_{b r}(x)$, $V_{b r}(x)$
    for $1 \leq b \leq B$ and $0 \leq r \leq J$,
    new data $(X_i, Y_i)$ for $n+1 \leq i \leq n+k$,
    recalculation gap $K \geq 1$.
  }

  Get the updated number of trees
  $B^* = \big\lfloor (n+k)^\xi B / n^\xi \big\rfloor$.

  With probability $1 \wedge (k / K)$, select $\lambda^*$ as in
  Section~\ref{sec:parameter_selection};
  otherwise, set $\lambda^* = (n+k)^\zeta \lambda / n^\zeta$.

  Generate the incrementally updated forest $T_{b r}^*(x)$ as in
  \eqref{eq:update_trees} and \eqref{eq:compute_new_trees}.

  Construct the updated union cell $U^*(x)$ and active indices $I^*(x)$
  as in \eqref{eq:compute_indices}.

  Calculate $N_{b r}^*(x)$, $S_{b r}^*(x)$, and $V_{b r}^*(x)$
  as in \eqref{eq:compute_update_sums}
  and derive $\hat\mu^*_\rd(x)$ and $\sigma^{2*}(x)$
  from \eqref{eq:compute_sums}.

  With probability $1 \wedge (k/K)$,
  recalculate $\hat\Sigma^*_\rd(x)$ using \eqref{eq:compute_Sigma_hat};
  otherwise, set $\hat\Sigma^*_\rd(x) = \hat\Sigma_\rd(x)$.

  Compute $\CI^*_\rd(x)$ using \eqref{eq:CI_d} with
  $\hat\mu^*_\rd(x)$ and $\hat\Sigma^*_\rd(x)$.

\end{algorithm}

Lemma~\ref{lem:compute_online} bounds the average case
time complexity of our online computational procedure
presented in Algorithm~\ref{alg:online}.

\begin{lemma}[Computational complexity of online learning]
  \label{lem:compute_online}
  Suppose Assumptions~\ref{ass:data}
  and~\ref{ass:estimator} hold.
  Then the average case time complexity of Algorithm~\ref{alg:online} is
  \begin{align*}
    \E \left[ \mathcal{T}_{\mathrm o} \right]
    &\lesssim
    d (J + 1) \!
    \left( \frac{k n d (J + 1)}{K} + k d + B \right)
    + \frac{d (J+1) \log (2 B (J+1))^d}{\lambda^d}
    \left( \! n + B k + \frac{n B}{K} \right).
  \end{align*}
\end{lemma}

Lemma~\ref{lem:compute_batch} already demonstrated that
Algorithm~\ref{alg:batch}
is more efficient than the naive approach of computing
$\hat\mu_\rd(x)$ and $\hat\Sigma_\rd(x)$
directly with \eqref{eq:debiased} and \eqref{eq:Sigma_hat_d}, respectively,
which each have a time complexity of $n (J+1) B$.
The reason for this is that by first constructing the active indices $I(x)$,
we avoid iterating over the entire sample for each tree in
Algorithm~\ref{alg:batch}. Lemma~\ref{lem:compute_online} formalizes the
improvement achieved by Algorithm~\ref{alg:online} in online settings,
relative to the batch estimation approach of Algorithm~\ref{alg:batch}. Most
importantly, the terms involving the product $n B$
are reduced to $n + B k + n B / K$, offering a substantial speed-up
in large forests when the new sample size $k$ is much smaller
than that of the existing data $n$, and when $K$ is large to avoid
regularly estimating the lifetime $\lambda$ and variance $\hat\Sigma_\rd(x)$.

\section{Conclusion}%
\label{sec:conclusion}

We presented a Berry--Esseen theorem under mild conditions
for the Mondrian random forest estimator,
and showed how it can be used to perform statistical inference
on an unknown nonparametric regression function.
We introduced debiased versions of Mondrian random forests, exploiting
higher-order smoothness, and
demonstrated their advantages for statistical inference and their
minimax optimality properties.
We discussed tuning parameter selection, enabling fully
feasible and practical estimation and inference procedures,
and demonstrated the empirical performance of our proposed methodology.
Finally, we developed efficient algorithms for
batch and online settings.

There are several potential avenues for future work on inference with
Mondrian random forests. The development of data-adaptive partitioning
schemes is one such important direction, and could be implemented
perhaps by allowing the lifetime
parameter $\lambda$ to vary across different covariates, yielding
the $d$-dimensional parameter $(\lambda_1, \ldots, \lambda_d)$.
One approach to designing such methodology might involve adapting sparse,
greedy algorithms for non-parametric regression, similar to those described by
\citet{lafferty2008rodeo}, to the context of axis-aligned partitioning
estimators. Specifically, by examining how changes in each $\lambda_j$ affect
the Mondrian forest estimator, e.g., via an estimate of $
\frac{\partial}{\partial \lambda_j} \mathbb{E}[\hat\mu(x)] $, these parameters
can be dynamically adjusted to
more effectively learn low-dimensional structure in the regression function.
Alternatively, one might formulate a Goldenshluger--Lepski-type procedure
\citep{goldenshluger2008universal}
for multiple tuning parameter selection.
Another potential line of research would consist of proposing further
strategies for debiasing Mondrian random forests (and related estimators);
an approach based on within-cell local polynomial smoothing, for example,
may serve to eliminate both design bias and boundary bias,
as well as allowing for less restrictive
conditions on the covariate density function and the regression function.

\section{Acknowledgments and funding}

The authors would like to thank the Editor,
two anonymous reviewers,
Rajita Chandak, Jianqing Fan,
Henry Reeve, Richard Samworth, and Sanjeev Kulkarni for
insightful comments and suggestions. Klusowski would like to thank the
participants
at the $11$th Bernoulli-IMS World Congress in Probability and Statistics in
Bochum, Germany for valuable feedback.
Underwood would like to thank the attendees
at the University of Pittsburgh,
University of Illinois Urbana-Champaign, and University of Michigan
Statistics Seminars,
as well as the members of the University of Cambridge
Statistical Laboratory.

Cattaneo was supported in part by the National Science Foundation through
DMS-2210561 and SES-2241575.
Klusowski was supported in part by the National Science Foundation
through CAREER DMS-2239448, DMS-2054808, and HDR TRIPODS CCF-1934924.
 
\appendix
\section{Proofs and technical results}%
\iftoggle{journal}{}{\label{sec:proofs}}

In this section we present the full proofs of all our results, and also state
some useful technical preliminary and intermediate lemmas.
We use the
following simplified notation for convenience, whenever it is appropriate:
write $\I_{i b}(x) = \I \left\{ X_i \in T_b(x) \right\}$ and $N_b(x) =
\sum_{i=1}^{n} \I_{i b}(x)$, as well as $\I_b(x) = \I \left\{ N_b(x) \geq 1
\right\}$.
We use $C$ to denote a positive constant whose value may change
from line to line, and write
$a_n = O(b_n)$ for $a_n \lesssim b_n$.
We begin with an overview of the main proof strategies and a discussion of the
challenges involved in Section~\ref{sec:overview_proofs}.
We then give some preliminary lemmas in Section~\ref{sec:lemmas},
and present the proofs for
Section~\ref{sec:inference}
(including Lemma~\ref{lem:bias}, Lemma~\ref{lem:variance},
  Theorem~\ref{thm:rate},
  Theorem~\ref{thm:clt}, Lemma~\ref{lem:variance_estimation}, and
Theorem~\ref{thm:confidence})
in Section~\ref{sec:proofs_inference};
the proofs for Section~\ref{sec:debiased}
(including Lemma~\ref{lem:bias_debiased}, Lemma~\ref{lem:variance_debiased},
  Theorem~\ref{thm:minimax}, Theorem~\ref{thm:clt_debiased},
  Lemma~\ref{lem:variance_estimation_debiased}, and
Theorem~\ref{thm:confidence_debiased})
in Section~\ref{sec:proofs_debiased};
and the proofs for Section~\ref{sec:compute}
(including Lemma~\ref{lem:compute_batch} and Lemma~\ref{lem:compute_online})
in Section~\ref{sec:proofs_compute}.

\subsection{Overview of proof strategies}%
\label{sec:overview_proofs}

This section provides some insight into
the general approach we use to establish our main results.
We focus on the technical innovations
forming the core of our arguments, and refer the reader to
the upcoming sections for full proofs.

\subsubsection{Preliminary technical results}

The starting point for our proofs is a result characterizing
the distribution of the shape of a Mondrian cell $T(x)$.
This property is a consequence of the fact that the restriction of a
Mondrian process to a subcell remains a Mondrian process
\citep{mourtada2020minimax}. We have
\begin{align*}
  |T(x)_j|
  &= \left( \frac{E_{j1}}{\lambda} \wedge x_j \right)
  + \left( \frac{E_{j2}}{\lambda} \wedge (1-x_j) \right)
\end{align*}
for all $1 \leq j \leq d$,
recalling that $T(x)_j$ is the side of the cell $T(x)$
aligned with axis $j$,
and where $E_{j1}$ and $E_{j2}$
are mutually independent $\Exp(1)$ random variables.
Our assumptions that $x \in (0,1)$ and $\lambda \to \infty$
mean that the ``boundary terms'' $x_j$ and $1-x_j$ are
eventually ignorable and so
$|T(x)_j| = (E_{j1} + E_{j2}) / \lambda$
with high probability.
Controlling the size of the largest cell in the forest
containing $x$ is now straightforward with a union
bound, giving
\begin{align*}
  \max_{1 \leq b \leq B}
  \max_{1 \leq j \leq d}
  |T_b(x)_j|
  \lesssim_\P \frac{\log B}{\lambda}.
\end{align*}
This shows that, up to logarithmic terms, none of the cells
in the forest at $x$ are significantly larger than average,
ensuring that the Mondrian random forest estimator is
``localized'' around $x$ on the scale of $1/\lambda$,
an important property for our bias characterization.

Having provided upper bounds for the sizes of Mondrian cells,
we also must establish some lower bounds in order to
ensure a sufficient effective sample size and to
quantify the ``small cells'' phenomenon mentioned previously.
The first step towards this is to bound the first two moments
of the truncated inverse Mondrian cell volume; we show that
\begin{align*}
  \E\left[
    1 \wedge
    \frac{1}{n |T(x)|}
  \right]
  &\asymp
  \frac{\lambda^d}{n}
  &&\text{and}
  &\E\left[
    1 \wedge
    \frac{1}{n^2 |T(x)|^2}
  \right]
  &\asymp
  \frac{\lambda^{2d} (\log n)^d}{n^2}.
\end{align*}
These bounds are computed using the exact distribution of $|T(x)|$.
Note that $\E\left[ 1 / |T(x)|^2 \right] = \infty$ because
$1 / (E_{j1} + E_{j2})$ has only $2 - \delta$ finite moments,
so the truncation is crucial here.
Since we have ``almost two'' moments, this truncation is
at the expense of only a logarithmic term.
Nonetheless, third and higher truncated moments will not enjoy such tight
bounds, demonstrating both the fragility of this result and
the inadequacy of
tools such as the Lyapunov central limit theorem which require
$2 + \delta$ marginal moments.

To conclude this investigation into the ``small cell'' phenomenon,
we apply the previous bounds to
ensure that the empirical effective sample sizes
$N_b(x) = \sum_{i=1}^{n} \I \left\{ X_i \in T_b(x) \right\}$
are approximately of the order $n / \lambda^d$ in an appropriate sense;
we demonstrate that
\begin{align*}
  \E\left[
    \frac{\I\{N_b(x) \geq 1\}}{N_b(x)}
  \right]
  &\lesssim
  \frac{\lambda^d}{n}
  &&\text{and}
  &\E\left[
    \frac{\I\{N_b(x) \geq 1\}}{N_b(x)^2}
  \right]
  &\lesssim
  \frac{\lambda^{2d} (\log n)^d}{n^2},
\end{align*}
as well as ``mixed'' bounds
$\E \left[
  \I\{N_b(x) \geq 1\} \I\{N_{b'}(x) \geq 1\}
  / (N_b(x) N_{b'}(x))
\right]
\lesssim \lambda^{2d} / n^2$
when $b \neq b'$, which arise from covariance terms across
multiple trees.
The proof of this result is involved and technical, and proceeds by
induction. The idea is to construct a class of subcells by taking
all possible intersections of the cells in $T_b$ and $T_{b'}$
(we show two trees here for clarity; there may be more)
and noting that each $N_b(x)$ is the sum of the number of points
in each such ``refined cell'' intersected with $T_b(x)$.
We then swap out each refined cell one at a time and replace
the number of data points it contains with its volume multiplied by $n f(x)$,
showing that the expectation on the left hand side does not increase too much
using a moment bound for inverse binomial random variables
based on Bernstein's inequality.
By induction and independence of the trees,
eventually the problem is reduced to computing moments
of truncated inverse Mondrian cell volumes, as above.

\subsubsection{Bias characterization}

Our first substantial result is the bias characterization
given as Lemma~\ref{lem:bias}, in which we precisely
characterize the probability limit of the conditional bias
\begin{align*}
  \E \left[ \hat \mu(x) \mid \bX, \bT \right]
  - \mu(x)
  &=
  \frac{1}{B}
  \sum_{b=1}^B
  \sum_{i=1}^n \big( \mu(X_i) - \mu(x) \big)
  \frac{\I\{X_i \in T_b(x)\}}{N_b(x)}.
\end{align*}
The first step in this proof is to pass to the ``infinite forest''
limit by taking an expectation conditional on $\bX$, or equivalently
marginalizing over $\bT$, applying the conditional Markov inequality
to see
\begin{align*}
  \big|
  \E \left[ \hat \mu(x) \mid \bX, \bT \right]
  - \E \left[ \hat \mu(x) \mid \bX \right]
  \big|
  &\lesssim_\P
  \frac{1}{\lambda^{1 \wedge \betamu} \sqrt B}.
\end{align*}
While this may seem a crude approximation, it is already known that
fixed-size Mondrian forests have suboptimal bias properties
when compared to forests with a diverging number of trees.
In fact, when $\beta \geq 1$,
the error $1 / \big(\lambda^{1 \wedge \beta} \sqrt B\big)$
exactly accounts for the first-order bias of individual Mondrian trees
\citep{mourtada2020minimax}.

Next we show that $\E \left[ \hat \mu(x) \mid \bX \right]$
converges in probability to its expectation, using
the Efron--Stein theorem to handle this non-linear function of
the i.i.d.\ variables $X_i$.
The important insight here is that replacing a sample
$X_i$ with an independent copy
$\tilde X_i$ can change the value of
$N_b(x)$ by at most one.
Further, this can happen only on the event
$\{ X_i \in T_{b}(x) \} \cup \{ \tilde X_i \in T_{b}(x) \}$,
which occurs with probability on the order $1/\lambda^d$
(the expected cell volume) for each tree $1 \leq b \leq B$.
The H\"older property of $\mu$
and the upper bound on the maximum cell size then give
$|\mu(X_i) - \mu(x)|
\lesssim \max_{1 \leq j \leq d} |T_b(x)_j|^{1 \wedge \betamu}
\lesssim_\P 1 / \lambda^{1 \wedge \betamu}$
whenever $X_i \in T_b(x)$,
so we combine this with moment bounds for the denominator $N_b(x)$ to see
\begin{align*}
  \left|
  \E \left[ \hat \mu(x) \mid \bX \right]
  - \E \left[ \hat \mu(x) \right]
  \right|
  \lesssim_\P
  \frac{1}{\lambda^{1 \wedge \betamu}}
  \sqrt{\frac{\lambda^d}{n}}.
\end{align*}

The next step is to approximate the resulting non-random bias
$\E \left[ \hat \mu(x) \right] - \mu(x)$ as a polynomial in $1/\lambda$.
To this end, we firstly apply a concentration-type result for the binomial
distribution to deduce that
\begin{align*}
  \E \left[ \frac{\I\{N_b(x) \geq 1\}}{N_b(x)} \Bigm| \bT \right]
  \approx
  \frac{1}{n \int_{T_b(x)} f(s) \diffi s}
\end{align*}
in an appropriate sense, and hence,
by conditioning on $\bT$ and $\bX$ without $X_i$,
\begin{align}
  \label{eq:bias_ratio}
  \E \left[ \hat \mu(x) \right]
  - \mu(x)
  &\approx
  \E \left[
    \frac{\int_{T_b(x)} (\mu(s) - \mu(x)) f(s) \diffi s}
    {\int_{T_b(x)} f(s) \diffi s}
  \right].
\end{align}
Next we apply the multivariate version of Taylor's theorem to the integrands
in both the numerator and the denominator in \eqref{eq:bias_ratio},
and then apply the Maclaurin series of $1 / (1+x)$
and the multinomial theorem to recover a single polynomial in $1/\lambda$.
The error term is on the order of $1/\lambda^\beta$ and depends on
the smoothness of $\mu$ and $f$,
and the polynomial coefficients are given by
various expectations involving exponential random variables.
The final step is to verify using symmetry of Mondrian cells that all
the odd monomial coefficients are zero,
and to calculate some explicit examples of the
form of the limiting bias.

\subsubsection{Central limit theorem}

To prove our second main result (Theorem~\ref{thm:clt}),
we apply a version of the Berry--Esseen theorem for i.n.i.d. random variables,
conditional on $(\bX, \bT)$,
which only requires $2 + \delta$ moments.
Define the variables
\begin{align*}
  S_i(x)
  &=
  \sqrt{\frac{n}{\lambda^d}}
  \frac{1}{B} \sum_{b=1}^B
  \frac{\I \{X_i \in T_b(x)\} \varepsilon_i}
  {N_{b}(x)},
\end{align*}
which are independent and zero-mean given $(\bX, \bT)$,
and further satisfy
\begin{align*}
  \sqrt{\frac{n}{\lambda^d}}
  \big(
    \hat\mu(x)
    - \E\left[
      \hat\mu(x) \mid \bX, \bT
    \right]
  \big)
  = \sum_{i=1}^n S_i(x).
\end{align*}
Thus by \citet[Theorem~5.7]{Petrov_1995_Book}, conditional on $(\bX, \bT)$,
taking a marginal expectation,
\begin{align*}
  \sup_{t \in \R}
  \bigg|
  \P \bigg(
    \tilde \Sigma(x)^{-\frac{1}{2}}
    \sum_{i=1}^{n} S_i
    \leq t
  \bigg)
  - \Phi(t)
  \bigg|
  &\lesssim
  \E \bigg[
    1 \wedge
    \bigg(
      \tilde \Sigma(x)^{- \frac{2 + (\delta \wedge 1)}{2}}
      \sum_{i=1}^{n}
      \E \Big[ |S_i|^{2 + (\delta \wedge 1)} \mid \bX, \bT \Big]
    \bigg)
  \bigg].
\end{align*}
Bounding the right-hand side
now reduces to establishing properties of $\tilde\Sigma(x)$
and its large-sample limit $\Sigma(x)$.
To this end, we again use the Efron--Stein theorem
to bound $\Var \big[ \tilde\Sigma(x) \big]$ and then
apply a careful sequence of approximations to control
$\E \big[ \tilde\Sigma(x) \big] - \Sigma(x)$.
The final task is to calculate the limiting variance $\Sigma(x)$.
It is a straightforward but tedious exercise to verify
that each denominator $N_b(x)$ can be replaced by $n f(x) |T_b(x)|$,
yielding
\begin{align*}
  \Sigma(x)
  &=
  \frac{\sigma^2(x)}{f(x)}
  \lim_{\lambda \to \infty}
  \frac{1}{\lambda^d}
  \E \left[
    \frac{|T_{b}(x) \cap T_{b'}(x)|}
    {|T_{b}(x)| \, |T_{b'}(x)|}
  \right]
  = \frac{\sigma^2(x)}{f(x)}
  \left(
    \E \left[
      \frac{(E_{1} \wedge E'_{1}) + (E_{2} \wedge E'_{2})}
      {(E_{1} + E_{2}) (E'_{1} + E'_{2})}
    \right]
  \right)^d,
\end{align*}
where $E_1$, $E_2$, $E'_1$, and $E'_2$ are independent $\Exp(1)$,
by the cell shape distribution and independence of the trees.
This final expectation is calculated by integration, using
various incomplete gamma function identities.

\subsubsection{Confidence intervals}

While Theorem~\ref{thm:clt} gives a distributional approximation for the
infeasible $t$-statistic, in order to construct confidence intervals we must
instead approximate the corresponding feasible $t$-statistic.
To do this, first observe that
if $\tau$ and $\hat \tau$ are real-valued random variables
and $\varepsilon > 0$, then the following anti-concentration inequality holds:
\begin{align*}
  \sup_{t \in \R}
  \left| \P \big( \hat \tau \leq t \big) - \Phi(t) \right|
  &\leq
  \sup_{t \in \R}
  \left| \P \big( \tau \leq t \big) - \Phi(t) \right|
  + \varepsilon \sqrt{2 / \pi}
  + \P \big( |\hat \tau - \tau| > \varepsilon \big).
\end{align*}
We apply this result to
\begin{align*}
  \hat\tau =
  \sqrt{\frac{n}{\lambda^d}}
  \left(
    \frac{\hat \mu(x) - \mu(x)}
    {\sqrt{\smash[b]{\hat \Sigma(x)}}}
    - \frac{\E\left[ \hat \mu(x) \right] - \mu(x)}
    {\sqrt{\Sigma(x)}}
  \right)
  &&\text{and}
  && \tau = \sqrt{\frac{n}{\lambda^d}}
  \frac{\hat \mu- \E\left[ \hat \mu(x) \mid \bX, \bT \right]}
  {\sqrt{\smash[b]{\tilde \Sigma(x)}}},
\end{align*}
bounding $\P \big( |\hat \tau - \tau| > \varepsilon \big)$ using
our established results on $\hat \mu(x)$, $\tilde \Sigma(x)$
and $\hat \Sigma(x)$.
Exploiting symmetry of the resulting confidence interval
permits a quadratic dependence of the coverage error on the bias.

\subsection{Preliminary lemmas}
\label{sec:lemmas}

We begin by bounding the maximum size of any cell
in a Mondrian forest containing $x$.
This result is used regularly throughout many of our other proofs,
and captures the ``localizing'' behavior of the Mondrian random
forest estimator, showing that Mondrian cells have side lengths
at most on the order of $1/\lambda$.
For distributions $P_1$ and $P_2$ on $\R$, we write
$P_1 \leq P_2$ if $P_2$ stochastically dominates $P_1$; that is,
if there exist random variables $X_1 \sim P_1$ and $X_2 \sim P_2$
such that $X_1 \leq X_2$ almost surely.
Likewise, if $X_1$ is a real-valued random variable
and $P_2$ is a law on $\R$,
we write $X_1 \leq P_2$ if there exists $X_2 \sim P_2$
such that $X_1 \leq X_2$ almost surely,
on a possibly enlarged probability space.
Observe that if $n_1, n_2 \in \N$ with $n_1 \leq n_2$
and $p_1, p_2 \in [0, 1]$ with $p_1 \leq p_2$, then
$\Bin(n_1, p_1) \leq \Bin(n_2, p_2)$.

\begin{lemma}[Upper bound on the largest cell in a Mondrian forest]%
  \label{lem:largest_cell}
  Let $T_1, \ldots, T_B \sim \cM\big([0,1]^d\!, \lambda\big)$
  and take $x \in (0,1)^d$. Then for all $t > 0$
  \begin{align*}
    \P \left(
      \max_{1 \leq b \leq B}
      \max_{1 \leq j \leq d}
      |T_b(x)_j|
      \geq \frac{t}{\lambda}
    \right)
    &\leq
    2dB e^{-t/2}.
  \end{align*}

\end{lemma}

\begin{proof}[Lemma~\ref{lem:largest_cell}]
  We use the explicit distribution of the shape of Mondrian cells
  given by \citet[Proposition~1]{mourtada2020minimax}.
  In particular, we have
  $|T_b(x)_j| = \left( \frac{E_{bj1}}{\lambda} \wedge x_j \right)
  + \left( \frac{E_{bj2}}{\lambda} \wedge (1-x_j) \right)$
  where $E_{bj1}$ and $E_{bj2}$
  are independent $\Exp(1)$ random variables for
  $1 \leq b \leq B$ and $1 \leq j \leq d$.
  Thus $|T_b(x)_j| \leq \frac{E_{bj1} + E_{bj2}}{\lambda}$
  and so by a union bound
  \begin{align*}
    \P \left(
      \max_{1 \leq b \leq B}
      \max_{1 \leq j \leq d}
      |T_b(x)_j|
      \geq \frac{t}{\lambda}
    \right)
    &\leq
    \P \left(
      \max_{1 \leq b \leq B}
      \max_{1 \leq j \leq d}
      (E_{bj1} \vee E_{bj2})
      \geq \frac{t}{2}
    \right) \\
    &\leq
    2dB\,
    \P \left(
      E_{bj1}
      \geq \frac{t}{2}
    \right)
    \leq
    2dB e^{-t/2}.
  \end{align*}
\end{proof}

The next result is another ``localization'' result, this time
showing that the union over the forest
of the cells $T_b(x)$ containing $x$ do not contain ``too many''
samples $X_i$.
In other words, the Mondrian random forest estimator fitted at $x$
should only depend on $n/\lambda^d$ (the effective sample size)
data points up to logarithmic terms.

\begin{lemma}[Upper bound on the number of active data points]%
  \label{lem:active_data}
  Suppose that Assumptions~\ref{ass:data} and~\ref{ass:estimator} hold
  and define
  $N_{\cup}(x) =
  \sum_{i=1}^{n} \I \left\{ X_i \in \bigcup_{b=1}^{B} T_b(x) \right\}$.
  Then for $t > 0$ and $n \geq \lambda^d$,
  with $\|f\|_\infty = \sup_{x \in [0,1]^d} f(x)$,
  \begin{align*}
    \P \left( N_{\cup}(x) > t^{d+1}
      \frac{n}{\lambda^d}
      \|f\|_\infty
    \right)
    &\leq
    4 d B e^{-t/4}.
  \end{align*}
\end{lemma}

\begin{proof}[Lemma~\ref{lem:active_data}]

  Note that
  \begin{align*}
    N_\cup(x) \sim
    \Bin\biggl(n, \int_{\bigcup_{b=1}^{B} T_b(x)} f(s) \diffi s \biggr)
    \leq \Bin\biggl(n, 2^d \max_{1 \leq b \leq B} \max_{1 \leq j \leq d}
    |T_b(x)_j|^d \|f\|_\infty \biggr)
  \end{align*}
  conditionally on $\bT$.
  If $N \sim \Bin(n,p)$ then, by Bernstein's inequality,
  $\P\left( N \geq (1 + t) n p\right)
  \leq \exp\left(-\frac{t^2 n^2 p^2 / 2}{n p(1-p) + t n p / 3}\right)
  \leq \exp\left(-\frac{3t^2 n p}{6 + 2t}\right)$.
  Thus for $t \geq 2$,
  \begin{align*}
    \P \left( N_{\cup}(x) > (1+t) n \frac{2^d t^d}{\lambda^d}
      \|f\|_\infty
      \Bigm| \max_{1 \leq b \leq B} \max_{1 \leq j \leq d}
      |T_j(x)| \leq \frac{t}{\lambda}
    \right)
    &\leq
    \exp\left(- \frac{2^d t^{d} n}{\lambda^d}\right).
  \end{align*}
  By Lemma~\ref{lem:largest_cell},
  $\P \left( \max_{1 \leq b \leq B} \max_{1 \leq j \leq d}
  |T_j(x)| > \frac{t}{\lambda} \right)
  \leq 2 d B e^{-t/2}$.
  Hence
  \begin{align*}
    &\P \left( N_{\cup}(x) > 2^{d+1} t^{d+1} \frac{n}{\lambda^d}
      \|f\|_\infty
    \right) \\
    &\quad\leq
    \P \left( N_{\cup}(x) > 2 t n \frac{2^d t^d}{\lambda^d}
      \|f\|_\infty
      \Bigm| \max_{1 \leq b \leq B} \max_{1 \leq j \leq d}
      |T_j(x)| \leq \frac{t}{\lambda}
    \right)
    + \P \left( \max_{1 \leq b \leq B} \max_{1 \leq j \leq d}
      |T_j(x)| > \frac{t}{\lambda}
    \right) \\
    &\quad\leq
    \exp\left(- \frac{2^d t^{d} n}{\lambda^d}\right)
    + 2 d B e^{-t/2}.
  \end{align*}
  Noting the result is trivial for $t < 4$ and replacing $t$ by $t/2$
  gives that for $n \geq \lambda^d$,
  \begin{align*}
    \P \left( N_{\cup}(x) > t^{d+1}
      \frac{n}{\lambda^d}
      \|f\|_\infty
    \right)
    &\leq
    4 d B e^{-t/4}.
  \end{align*}
\end{proof}

Next we give a series of results culminating in a
generalized moment bound for the denominator appearing
in the Mondrian random forest estimator.
We begin by providing a moment bound for the truncated inverse binomial
distribution, which will be useful for controlling
$\frac{\I_b(x)}{N_b(x)} \leq 1 \wedge \frac{1}{N_b(x)}$
because conditional on $T_b$ we have
$N_b(x) \sim \Bin \left( n, \int_{T_b(x)} f(s) \diffi s \right)$.
Our constants could be suboptimal but they are sufficient
for our applications.

\begin{lemma}[An inverse moment bound for the binomial distribution]%
  \label{lem:binomial_bound}
  For $n \geq 1$ and $p \in [0,1]$, let $N \sim \Bin(n, p)$.
  Take $a_1, \ldots, a_k \geq 0$ and $l_1, \ldots, l_k \geq 1$.
  Then with $L = \sum_{j=1}^{k} l_j$,
  \begin{align*}
    \E\left[
      \prod_{j=1}^k
      \left(
        1 \wedge
        \frac{1}{N + a_j}
      \right)^{l_j}
    \right]
    &\leq
    (9 L)^{2 L}
    \prod_{j=1}^k
    \left(
      1 \wedge
      \frac{1}{n p + a_j}
    \right)^{l_j}.
  \end{align*}
\end{lemma}

\begin{proof}[Lemma~\ref{lem:binomial_bound}]
  By Bernstein's inequality,
  $\P\left( N \leq n p - t \right)
  \leq \exp\left(-\frac{3t^2}{6n p + 2t}\right)$.
  Therefore we have
  $\P\left( N \leq n p/4 \right)
  \leq \exp\left(-\frac{27 n^2 p^2 / 16}{6n p + 3 n p / 2}\right)
  = e^{-9 n p / 40}$.
  Partitioning by this event gives
  \begin{align*}
    \E\left[
      \prod_{j=1}^k
      \left(
        1 \wedge
        \frac{1}{N + a_j}
      \right)^{l_j}
    \right]
    &\leq
    e^{-9 n p / 40}
    \prod_{j=1}^k
    \frac{1}{1 \vee a_j^{l_j}}
    + \prod_{j=1}^k
    \frac{1}{1 \vee (\frac{n p}{4} + a_j)^{l_j}} \\
    &\leq
    \prod_{j=1}^k
    \frac{1}{1 \vee \left(\frac{9 n p}{40 k l_j} + a_j\right)^{l_j}}
    + \prod_{j=1}^k
    \frac{1}{1 \vee (\frac{n p}{4} + a_j)^{l_j}} \\
    &\leq
    2 \prod_{j=1}^k
    \frac{1}{1 \vee \left(\frac{9 n p}{40 k l_j} + a_j\right)^{l_j}}
    \leq
    2 \prod_{j=1}^k
    \frac{(40 k l_j /9)^{l_j}}{1 \vee \left(n p + a_j\right)^{l_j}} \\
    &\leq
    (9 L)^{2 L}
    \prod_{j=1}^k
    \left(
      1 \wedge
      \frac{1}{n p + a_j}
    \right)^{l_j}.
  \end{align*}
\end{proof}

Our next result is probably the most technically involved in the paper,
allowing one to bound moments of
(products of) $\frac{\I_b(x)}{N_b(x)}$ by the corresponding moments of
(products of) $\frac{1}{n |T_b(x)|}$, again based on the heuristic
that $N_b(x)$ is conditionally binomial so concentrates around
its conditional expectation
$n \int_{T_b(x)} f(x) \diffi s \asymp n |T_b(x)|$.
By independence of the trees,
the latter expected products then factorize
since the dependence on the data $X_i$ has been eliminated.
The proof is complicated, and relies on the following induction procedure.
First we consider the common refinement consisting of the
subcells $\cR$ generated by all possible intersections
of $T_b(x)$ over the selected trees
(say $T_{b}(x), T_{b'}(x), T_{b''}(x)$
though there could be arbitrarily many).
Note that $N_b(x)$ is the sum of the number of
samples $X_i$ in each such subcell in $\cR$.
We then apply Lemma~\ref{lem:binomial_bound} repeatedly
to each subcell in $\cR$ in turn, replacing
the number of samples $X_i$ in that subcell with its volume
multiplied by the sample size $n$,
and controlling the error incurred at each step.
We record the subcells which have been ``checked'' in this manner
using the class $\cD \subseteq \cR$ and proceed by finite induction,
beginning with $\cD = \emptyset$ and ending at $\cD = \cR$.

\begin{lemma}[Generalized moment bound for
  Mondrian random forest denominators]%
  \label{lem:moment_denominator}

  Suppose Assumptions~\ref{ass:data}
  and~\ref{ass:estimator} hold.
  Let $T_b \sim \cM\big([0,1]^d, \lambda\big)$
  be independent and $k_b \geq 1$ for $1 \leq b \leq B_0$.
  Then with $k = \sum_{b=1}^{B_0} k_b$,
  for sufficiently large $n$,
  \begin{align*}
    \E\left[
      \prod_{b=1}^{B_0}
      \frac{\I_b(x)}{N_b(x)^{k_b}}
    \right]
    &\leq
    \left( \frac{36k}{\inf_{x \in [0,1]^d} f(x)} \right)^{2^{2k}}
    \prod_{b=1}^{B_0}
    \E \left[
      1 \wedge
      \frac{1}{(n |T_b(x)|)^{k_b}}
    \right].
  \end{align*}
\end{lemma}

\begin{proof}[Lemma~\ref{lem:moment_denominator}]

  Define the common refinement of
  $\left\{ T_b(x) : 1 \leq b \leq {B_0} \right\}$ as
  the class of sets
  \begin{align*}
    \cR
    &= \left\{ \bigcap_{b=1}^{{B_0}} D_b :
      D_b \in
      \big\{ T_b(x), T_b(x)^\c \big\}
    \right\}
    \bigsetminus
    \left\{
      \emptyset,\,
      \bigcap_{b=1}^{{B_0}}
      T_b(x)^\c
    \right\}
  \end{align*}
  where $T_b(x)^\c = [0, 1]^d \setminus T_b(x)$,
  and let $\cD \subset \cR$.
  We will proceed by induction on the elements of $\cD$,
  which represents the subcells we have checked,
  starting from $\cD = \emptyset$ and finishing at $\cD = \cR$.
  For $D \in \cR$ let
  $\cA(D) = \left\{ 1 \leq b \leq {B_0} : D \subseteq T_b(x) \right\}$
  be the indices of the trees which are active on subcell $D$,
  and for $1 \leq b \leq {B_0}$ let
  $\cA(b) = \left\{ D \in \cR : D \subseteq T_b(x) \right\}$
  be the subcells which are contained in $T_b(x)$,
  so that $b \in \cA(D) \iff D \in \cA(b)$.
  For a subcell $D \in \cR$,
  write $N_b(D) = \sum_{i=1}^{n} \I \left\{ X_i \in D \right\}$
  so that $N_b(x) = \sum_{D \in \cA(b)} N_b(D)$.
  Note that for any $D \in \cR \setminus \cD$,
  \begin{align*}
    &\E \left[
      \prod_{b=1}^{B_0}
      \frac{1}{
        1 \vee \left(
          \sum_{D' \in \cA(b) \setminus \cD}
          N_b(D')
          + n \sum_{D' \in \cA(b) \cap \cD}
          |D'|
        \right)^{k_b}
      }
    \right] \\
    &=
    \E \left[
      \prod_{b \notin \cA(D)}
      \frac{1}{
        1 \vee \left(
          \sum_{D' \in \cA(b) \setminus \cD}
          N_b(D')
          + n \sum_{D' \in \cA(b) \cap \cD}
          |D'|
        \right)^{k_b}
      } \right. \\
      &\left.
      \ \times\,\E\hspace*{-1mm}\left[
        \prod_{b \in \cA(D)}
        \frac{1}{
          1 \vee \left(
            \sum_{D' \in \cA(b) \setminus \cD}
            N_b(D')
            + n \sum_{D' \in \cA(b) \cap \cD}
            |D'|
          \right)^{k_b}
        }
        \biggm|
        \bT,
        N_b(D') : D' \in \cR
        \hspace*{-1mm}
        \setminus
        \hspace*{-1mm}
        (\cD \cup \{D\})
      \right]
    \right]\hspace*{-1mm}.
  \end{align*}
  Now the inner conditional expectation is over $N_b(D)$ only.
  Since $f$ is bounded away from zero,
  \begin{align*}
    N_b(D)
    &\sim \Bin\left(
      n - \sum_{D' \in \cR \setminus (\cD \cup \{D\})} N_b(D'), \
      \frac{\int_{D} f(s) \diffi s}
      {1 - \int_{\bigcup \left( \cR \setminus \cD \right) \setminus D}
      f(s) \diffi s}
    \right) \\
    &\geq \Bin\left(
      n - \sum_{D' \in \cR \setminus (\cD \cup \{D\})} N_b(D'), \
      |D| \inf_{x \in [0,1]^d} f(x)
    \right)
  \end{align*}
  conditional on $\bT$ and
  $N_b(D') : D' \in \cR \setminus (\cD \cup \{D\})$.
  Further, by Lemma~\ref{lem:active_data} with $n \geq \lambda^d$,
  \begin{align*}
    \P \left(
      \sum_{D' \in \cR \setminus (\cD \cup \{D\})} N_b(D')
    > t^{d+1} \frac{n}{\lambda^d} \|f\|_\infty \right)
    &\leq
    \P \left( N_{\cup}(x) > t^{d+1}
      \frac{n}{\lambda^d}
      \|f\|_\infty
    \right)
    \leq
    4 d B_0 e^{-t/4}.
  \end{align*}
  Thus
  $N_b(D) \geq \Bin(n/2, |D| \inf_x f(x))$
  conditional on
  $\left\{ \bT, N_b(D') : D' \in \cR \setminus (\cD \cup \{D\}) \right\}$
  with probability at least
  $1 - 4 d B_0 e^{\frac{-\sqrt \lambda}{8 \|f\|_\infty}}$.
  So by Lemma~\ref{lem:binomial_bound},
  \begin{align*}
    &\E \left[
      \prod_{b \in \cA(D)}
      \frac{1}{
        1 \vee \left(
          \sum_{D' \in \cA(b) \setminus \cD}
          N_b(D')
          + n \sum_{D' \in \cA(b) \cap \cD}
          |D'|
        \right)^{k_b}
      }
      \biggm|
      \bT,
      N_b(D') : D' \in \cR \setminus (\cD \cup \{D\})
    \right] \\
    &\quad\leq
    (9k)^{2k} \,
    \E \left[
      \prod_{b \in \cA(D)}
      \frac{1}{
        1 \vee \left(
          \sum_{D' \in \cA(b) \setminus (\cD \cup \{D\})}
          N_b(D')
          + n |D| \inf_x f(x) / 2
          + n \sum_{D' \in \cA(b) \cap \cD}
          |D'|
      \right)^{k_b}}
    \right] \\
    &\qquad+
    4 d B_0 e^{\frac{-\sqrt \lambda}{8 \|f\|_\infty}} \\
    &\quad\leq
    \left( \frac{18k}{\inf_x f(x)} \right)^{2k}
    \E
    \left[
      \prod_{b \in \cA(D)}
      \frac{1}{
        1 \vee \left(
          \sum_{D' \in \cA(b) \setminus (\cD \cup \{D\})}
          N_b(D')
          + n \sum_{D' \in \cA(b) \cap (\cD \cup \{D\})}
          |D'|
      \right)^{k_b}}
    \right] \\
    &\qquad+
    4 d B_0 e^{\frac{-\sqrt \lambda}{8 \|f\|_\infty}}.
  \end{align*}
  Therefore plugging this back into the marginal expectation yields
  \begin{align*}
    &\E\left[
      \prod_{b=1}^{B_0}
      \frac{1}{
        1 \vee \left(
          \sum_{D' \in \cA(b) \setminus \cD}
          N_b(D')
          + n \sum_{D' \in \cA(b) \cap \cD}
          |D'|
        \right)^{k_b}
      }
    \right] \\
    &\quad\leq
    \left( \frac{18k}{\inf_x f(x)} \right)^{2k}
    \E \left[
      \prod_{b=1}^{B_0}
      \frac{1}{
        1 \vee \left(
          \sum_{D' \in \cA(b) \setminus (\cD \cup \{D\})}
          N_b(D')
          + n \sum_{D' \in \cA(b) \cap (\cD \cup \{D\})}
          |D'|
      \right)^{k_b}}
    \right] \\
    &\qquad+
    4 d B_0 e^{\frac{-\sqrt \lambda}{8 \|f\|_\infty}}.
  \end{align*}
  Now we apply induction,
  starting with $\cD = \emptyset$ and
  adding $D \in \cR \setminus \cD$ to $\cD$ until
  $\cD = \cR$.
  This takes at most $|\cR| \leq 2^k$ steps and yields
  \begin{align*}
    \E\left[
      \prod_{b=1}^{B_0}
      \frac{\I_b(x)}{N_b(x)^{k_b}}
    \right]
    &\leq
    \E\left[
      \prod_{b=1}^{B_0}
      \frac{1}{1 \vee N_b(x)^{k_b}}
    \right]
    =
    \E\left[
      \prod_{b=1}^{B_0}
      \frac{1}{1 \vee \left( \sum_{D \in \cA(b)} N_b(D) \right)^{k_b}}
    \right]
    \leq \cdots \\
    &\leq
    \left( \frac{18k}{\inf_x f(x)} \right)^{2^{2k}}
    \left(
      \prod_{b=1}^{B_0}
      \,\E \left[
        \frac{1}{1 \vee (n |T_b(x)|)^{k_b}}
      \right]
      + 4 d B_0 2^{k} e^{\frac{-\sqrt \lambda}{8 \|f\|_\infty}}
    \right),
  \end{align*}
  where the expectation factorizes due to independence of $T_b(x)$.
  The last step is to remove the trailing exponential term.
  To do this, note that by Jensen's inequality,
  \begin{align*}
    \prod_{b=1}^{B_0}
    \,\E \left[
      \frac{1}{1 \vee (n |T_b(x)|)^{k_b}}
    \right]
    &\geq
    \prod_{b=1}^{B_0}
    \frac{1}
    {\E \left[ 1 \vee (n |T_b(x)|)^{k_b} \right]}
    \geq
    \prod_{b=1}^{B_0}
    \frac{1}{n^{k_b}}
    = n^{-k}
  \end{align*}
  while the assumption of $\lambda \gtrsim (\log n)^3$ implies
  $\lambda \geq (\log n)^3 / C^2$ eventually for some $C > 0$, giving
  \begin{align*}
    4 d B_0 2^k e^{\frac{-\sqrt \lambda}{8 \|f\|_\infty}}
    \leq
    4 d B_0 2^k e^{\frac{ - (\log n)^{3/2}}{8 C \|f\|_\infty}}
    \leq
    4 d B_0 2^k e^{- k \log n - \log (4 d B_0 2^k)}
    \leq
    n^{-k}
  \end{align*}
  for sufficiently large $n$
  because $B_0$, $d$, and $k$ are fixed.
\end{proof}

Now that moments of (products of) $\frac{\I_b(x)}{N_b(x)}$
have been bounded by moments of
(products of) $\frac{1}{n |T_b(x)|}$, we establish further
explicit bounds for these in the next result.
Note that the problem has been reduced to determining
properties of Mondrian cells, so once again we return to the
exact cell shape distribution given by \citet{mourtada2020minimax},
and evaluate the appropriate expectations by integration.
Note that the truncation by taking the minimum with one inside the expectation
is essential here, as otherwise second moment of the inverse Mondrian cell
volume is not even finite. As such, there is a ``penalty'' of $(\log n)^d$
when bounding truncated second moments,
and the upper bound for the $k$th moment is significantly
larger than the naive assumption of $(\lambda^d / n)^k$
whenever $k > 2$.
This ``small cell'' phenomenon in which the inverse volumes of Mondrian cells
have heavy tails is a recurring challenge in our analysis.

\begin{lemma}[Inverse moments of the volume of a Mondrian cell]%
  \label{lem:moment_cell}

  Suppose Assumption~\ref{ass:estimator} holds
  and let $T \sim \cM\big([0,1]^d, \lambda\big)$.
  Then with $k \geq 1$, for sufficiently large $n$,
  \begin{align*}
    \E\left[
      1 \wedge
      \frac{1}{(n |T(x)|)^k}
    \right]
    &\leq
    \left(
      \frac{2}{2-k}
      \frac{\lambda^{d k}}{n^k}
    \right)^{\I \left\{ k < 2 \right\}}
    \left(
      \frac{3 \lambda^{2d} (\log n)^d}{n^2}
    \right)^{\I \left\{ k \geq 2 \right\}}
    \prod_{j=1}^{d} \frac{1}{x_j (1-x_j)}.
  \end{align*}
\end{lemma}

\begin{proof}[Lemma~\ref{lem:moment_cell}]

  By \citet[Proposition~1]{mourtada2020minimax}, we have
  \begin{align*}
    |T(x)| = \prod_{j=1}^{d}
    \left\{
      \left(\frac{1}{\lambda} E_{j1} \right) \wedge x_j
      + \left( \frac{1}{\lambda} E_{j2} \right) \wedge (1-x_j)
    \right\},
  \end{align*}
  where $E_{j1}$ and $E_{j2}$
  are mutually independent $\Exp(1)$ random variables.
  Thus for any $0<t<1$,
  using the fact that $E_{j1} + E_{j2} \sim \Gam(2, 1)$,
  \begin{align*}
    \E \left[
      \frac{1}{1 \vee (n |T(x)|)^k}
    \right]
    &\leq
    \frac{1}{n^k}
    \E \left[
      \frac{\I\{\min_j (E_{j1} + E_{j2}) \geq t\}}{|T(x)|^k}
    \right]
    + \P \left(\min_{1 \leq j \leq d} (E_{j1} + E_{j2}) < t\right) \\
    &\leq
    \frac{1}{n^k}
    \prod_{j=1}^d
    \E \left[
      \frac{\I\{E_{j1} + E_{j2} \geq t\}}
      {\left(\frac{1}{\lambda} E_{j1} \wedge x_j
      + \frac{1}{\lambda} E_{j2} \wedge (1-x_j)\right)^k}
    \right]
    + d\, \P \left(E_{j1} < t\right) \\
    &\leq
    \frac{\lambda^{d k}}{n^k}
    \prod_{j=1}^d
    \frac{1}{x_j(1-x_j)}
    \E \left[
      \frac{\I\{E_{j1} + E_{j2} \geq t\}}
      {(E_{j1} +  E_{j2})^k \wedge 1}
    \right]
    + d (1 - e^{-t}) \\
    &\leq
    \frac{\lambda^{d k}}{n^k}
    \prod_{j=1}^d
    \frac{1}{x_j(1-x_j)}
    \left(
      \int_{t}^{1}
      \frac{e^{-s}}{s^{k-1}}
      \diffi s
      + \int_{1}^{\infty}
      s e^{-s}
      \diffi s
    \right)
    + d t \\
    &\leq
    \frac{\lambda^{d k}}{n^k}
    \prod_{j=1}^d
    \frac{1}{x_j(1-x_j)}
    \left(
      \int_{t}^{1}
      s^{1-k}
      \diffi s
      + 1
    \right)
    + d t \\
    &=
    d t
    + \frac{\lambda^{d k}}{n^k}
    \prod_{j=1}^d
    \frac{1}{x_j(1-x_j)}
    \times
    \begin{cases}
      1 + \frac{1}{2-k} - \frac{t^{2-k}}{2-k} & \text{if } 1 \leq k < 2, \\
      1 - \log t & \text{if } k = 2.
    \end{cases}
  \end{align*}
  If $k>2$ we simply use
  $\frac{1}{1 \vee (n |T(x)|)^k} \leq \frac{1}{1 \vee (n |T(x)|)^{2}}$.
  Now if $1 \leq k < 2$ we let $t \to 0$, giving
  \begin{align*}
    \E \left[
      \frac{1}{1 \vee (n |T(x)|)^k}
    \right]
    &\leq
    \frac{2}{2-k}
    \frac{\lambda^{d k}}{n^k}
    \prod_{j=1}^d
    \frac{1}{x_j(1-x_j)},
  \end{align*}
  and if $k = 2$ then we set $t = 1/n^2$ so that for sufficiently large $n$,
  \begin{align*}
    \E \left[
      \frac{1}{1 \vee (n |T(x)|)^2}
    \right]
    &\leq
    \frac{d}{n^2}
    + \frac{\lambda^{2d} \big(1 + 2 (\log n)^d\big)}{n^2}
    \prod_{j=1}^d
    \frac{1}{x_j(1-x_j)}
    \leq
    \frac{3 \lambda^{2d} (\log n)^d}{n^2}
    \prod_{j=1}^d
    \frac{1}{x_j(1-x_j)}.
  \end{align*}
  Lower bounds which match up to constants for $1 \leq k < 2$
  are easily obtained by noting
  $\E \left[ 1 \wedge \frac{1}{(n|T(x)|)^k} \right]
  \geq \E \left[ 1 \wedge \frac{1}{n|T(x)|} \right]^k$
  by Jensen's inequality and
  \begin{align*}
    \E \left[ 1 \wedge \frac{1}{n|T(x)|} \right]
    &\geq \frac{1}{1 + n \E \left[ |T(x)| \right]}
    \geq \frac{1}{1 + 2^d n / \lambda^d}
    \gtrsim \frac{\lambda^d}{n}.
  \end{align*}
  To obtain a lower bound when $k=2$, note that
  \begin{align*}
    \E \left[
      \frac{1}{1 \vee (n |T(x)|)^2}
    \right]
    &\geq
    \E \left[
      1 \wedge
      \frac{\lambda^d}{n^2}
      \prod_{j=1}^{d}
      \frac{1}{(E_{j1} + E_{j2})^2}
    \right]
    \geq
    \frac{\lambda^d}{n^2}
    \E \left[
      \frac{1}{(E_{1} + E_{2})^2 \vee n^{-1/d}}
    \right]^d \\
    &\geq
    \frac{\lambda^d}{n^2}
    \left(
      \int_{n^{-\frac{1}{2d}}}^{1} \frac{e^{-s}}{s} \diffi s
    \right)^d
    \geq
    \frac{\lambda^d}{n^2}
    \frac{1}{e}
    \left(
      \int_{n^{-\frac{1}{2d}}}^{1} \frac{1}{s} \diffi s
    \right)^d
    \geq
    \frac{\lambda^d}{n^2}
    \frac{1}{e}
    \left(
      \frac{1}{2d}
      \log n
    \right)^d.
  \end{align*}
\end{proof}

The ongoing endeavor to bound
moments of (products of) $\frac{\I_b(x)}{N_b(x)}$
is concluded with the next result,
chaining together the previous two lemmas to provide an
explicit bound with no expectations on the right-hand side.

\begin{lemma}[Simplified generalized moment bound for
  Mondrian random forest denominators]%
  \label{lem:simple_moment_denominator}
  Grant Assumptions~\ref{ass:data}
  and~\ref{ass:estimator}.
  Let $T_b \sim \cM\big([0,1]^d, \lambda\big)$
  and $k_b \geq 1$ for $1 \leq b \leq B_0$.
  Then with $k = \sum_{b=1}^{B_0} k_b$,
  for sufficiently large $n$,
  \begin{align*}
    &\E\left[
      \prod_{b=1}^{B_0}
      \frac{\I_b(x)}{N_b(x)^{k_b}}
    \right] \\
    &\quad\leq
    \left( \frac{36k}{\inf_{x \in [0,1]^d} f(x)} \right)^{2^{2k}}
    \left(
      \prod_{j=1}^{d} \frac{1}{x_j (1-x_j)}
    \right)^{B_0}
    \prod_{b=1}^{B_0}
    \left(
      \frac{2}{2-k_b}
      \frac{\lambda^{d k_b}}{n^{k_b}}
    \right)^{\I \left\{ k_b < 2 \right\}}
    \left(
      \frac{3 \lambda^{2d} (\log n)^d}{n^2}
    \right)^{\I \left\{ k_b \geq 2 \right\}}.
  \end{align*}
\end{lemma}

\begin{proof}[Lemma~\ref{lem:simple_moment_denominator}]
  This follows directly from
  Lemmas~\ref{lem:moment_denominator} and~\ref{lem:moment_cell}.
\end{proof}

Our final preliminary lemma is concerned with further properties of
the inverse truncated binomial distribution, again with the aim
of analyzing $\frac{\I_b(x)}{N_b(x)}$.
This time, instead of merely upper bounding the moments,
we aim to give convergence results for those moments,
again in terms of moments of $\frac{1}{n |T_b(x)|}$.
This time we only need to handle the first
and second moment, so this result does not strictly generalize
Lemma~\ref{lem:binomial_bound} except in simple cases.
The proof is by Taylor's theorem and the Cauchy--Schwarz inequality,
using explicit expressions for moments of the binomial distribution
and bounds from Lemma~\ref{lem:binomial_bound}.

\begin{lemma}[Expectation inequalities for the binomial distribution]%
  \label{lem:binomial_expectation}
  Let $N \sim \Bin(n, p)$ and take $a, b \geq 1$. Then
  \begin{align*}
    0
    &\leq
    \E \left[
      \frac{1}{N+a}
    \right]
    - \frac{1}{n p+a}
    \leq
    \frac{2^{19}}{(n p+a)^2}, \\
    0
    &\leq
    \E \left[
      \frac{1}{(N+a)(N+b)}
    \right]
    - \frac{1}{(n p+a)(n p+b)}
    \leq
    \frac{2^{27}}{(n p +a)(n p +b)}
    \left(
      \frac{1}{n p + a}
      + \frac{1}{n p + b}
    \right).
  \end{align*}

\end{lemma}

\begin{proof}[Lemma~\ref{lem:binomial_expectation}]

  For the first result,
  Taylor's theorem with Lagrange remainder
  applied to $N \mapsto \frac{1}{N+a}$ around $n p$ gives
  \begin{align*}
    \E \left[
      \frac{1}{N+a}
    \right]
    &=
    \E \left[
      \frac{1}{n p+a}
      - \frac{N - n p}{(n p+a)^2}
      + \frac{(N - n p)^2}{(\xi+a)^3}
    \right]
  \end{align*}
  for some $\xi$ between $n p$ and $N$.
  The second term on the right-hand side is zero-mean,
  clearly showing the non-negativity part of the result,
  and applying the Cauchy--Schwarz inequality
  to the remaining term gives
  \begin{align*}
    \E \left[
      \frac{1}{N+a}
    \right]
    - \frac{1}{n p+a}
    &\leq
    \E \left[
      \frac{(N - n p)^2}{(n p+a)^3}
      + \frac{(N - n p)^2}{(N+a)^3}
    \right] \\
    &\leq
    \frac{\E\big[(N - n p)^2\big]}{(n p+a)^3}
    + \sqrt{
      \E\big[(N - n p)^4\big]
      \E \left[
        \frac{1}{(N+a)^6}
    \right]}.
  \end{align*}
  Now we use
  $\E\big[(N - n p)^4\big] \leq n p(1+3n p)$
  and apply Lemma~\ref{lem:binomial_bound} to see that
  \begin{align*}
    \E \left[
      \frac{1}{N+a}
    \right]
    - \frac{1}{n p+a}
    &\leq
    \frac{n p}{(n p+a)^3}
    + \sqrt{\frac{54^6 n p(1+3 n p)}{(n p + a)^6}}
    \leq
    \frac{2^{19}}{(n p+a)^2}.
  \end{align*}
  For the second result,
  Taylor's theorem applied to $N \mapsto \frac{1}{(N+a)(N+b)}$
  around $n p$ gives
  \begin{align*}
    \E \left[
      \frac{1}{(N+a)(N+b)}
    \right]
    &=
    \E \left[
      \frac{1}{(n p+a)(n p + b)}
      - \frac{(N - n p)(2 n p + a + b)}{(n p + a)^2 (n p + b)^2}
    \right] \\
    &\quad+
    \E \left[
      \frac{(N - n p)^2}{(\xi+a)(\xi+b)}
      \left(
        \frac{1}{(\xi + a)^2}
        + \frac{1}{(\xi + a)(\xi + b)}
        + \frac{1}{(\xi + b)^2}
      \right)
    \right]
  \end{align*}
  for some $\xi$ between $n p$ and $N$.
  The second term on the right-hand side is zero-mean,
  clearly showing the non-negativity part of the result,
  and applying the Cauchy--Schwarz inequality
  to the remaining term gives
  \begin{align*}
    \E \left[
      \frac{1}{(N+a)(N+b)}
    \right]
    - \frac{1}{n p+a}
    &\leq
    \E \left[
      \frac{2 (N - n p)^2}{(N+a)(N+b)}
      \left(
        \frac{1}{(N + a)^2}
        + \frac{1}{(N + b)^2}
      \right)
    \right] \\
    &\quad+
    \E \left[
      \frac{2 (N - n p)^2}{(n p +a)(n p +b)}
      \left(
        \frac{1}{(n p + a)^2}
        + \frac{1}{(n p + b)^2}
      \right)
    \right] \\
    &\leq
    \sqrt{
      4 \E \left[ (N - n p)^4 \right]
      \E \left[
        \frac{1}{(N + a)^6 (N+b)^2}
        + \frac{1}{(N + b)^6 (N+a)^2}
    \right]} \\
    &\quad+
    \frac{2 \E\big[(N - n p)^2\big]}{(n p +a)(n p +b)}
    \left(
      \frac{1}{(n p + a)^2}
      + \frac{1}{(n p + b)^2}
    \right).
  \end{align*}
  Now we use
  $\E\big[(N - n p)^4\big] \leq n p(1+3n p)$
  and apply Lemma~\ref{lem:binomial_bound} to see that
  \begin{align*}
    \E \left[
      \frac{1}{(N+a)(N+b)}
    \right]
    - \frac{1}{n p+a}
    &\leq
    \sqrt{
      \frac{4n p (1 + 3n p) \cdot 72^8}{(n p + a)^2 (n p + b)^2}
      \left(
        \frac{1}{(n p + a)^4}
        + \frac{1}{(n p + b)^4}
    \right)} \\
    &\quad+
    \frac{2 n p}{(n p +a)(n p +b)}
    \left(
      \frac{1}{(n p + a)^2}
      + \frac{1}{(n p + b)^2}
    \right) \\
    &\leq
    \frac{2^{27}}{(n p + a) (n p + b)}
    \left(
      \frac{1}{n p + a}
      + \frac{1}{n p + b}
    \right).
  \end{align*}
\end{proof}

\subsection{Proofs for Section~\ref{sec:inference}}
\label{sec:proofs_inference}

We give rigorous proofs of the bias and variance
characterizations, rate of convergence, central limit theorem,
variance estimation, and confidence interval validity
results for the Mondrian random forest estimator.

\begin{proof}[Lemma~\ref{lem:bias}]

  We begin by showing that
  $\E \left[ \hat \mu(x) \mid \bX, \bT \right]$
  converges to
  $\E \left[ \hat \mu(x) \mid \bX \right]$.

  \proofparagraph{Removing the dependence on the trees}

  By measurability and with $\mu(X_i) = \E[Y_i \mid X_i]$ almost surely,
  \begin{align*}
    \E \left[ \hat \mu(x) \mid \bX, \bT \right]
    - \mu(x)
    &=
    \frac{1}{B}
    \sum_{b=1}^B
    \sum_{i=1}^n \big( \mu(X_i) - \mu(x) \big)
    \frac{\I_{i b}(x)}{N_b(x)}.
  \end{align*}
  Now conditional on $\bX$,
  the terms in the outer sum depend only on $T_b$ so are i.i.d.
  Since $\mu \in \cH^{\betamu}$,
  \begin{align*}
    \Var \left[
      \E \left[ \hat \mu(x) \mid \bX, \bT \right]
      - \mu(x)
      \mid \bX
    \right]
    &\leq
    \frac{1}{B}
    \E \left[
      \left(
        \sum_{i=1}^n \big( \mu(X_i) - \mu(x) \big)
        \frac{\I_{i b}(x)}{N_b(x)}
      \right)^2
      \Bigm| \bX
    \right] \\
    &\lesssim
    \frac{1}{B}
    \E \left[
      \max_{1 \leq i \leq n}
      \left\{
        \I_{i b}(x)
        \big\| X_i - x \big\|_2^{2(1 \wedge \betamu)}
      \right\}
      \left(
        \sum_{i=1}^n
        \frac{\I_{i b}(x)}{N_b(x)}
      \right)^2
      \Bigm| \bX
    \right] \\
    &\lesssim
    \frac{1}{B}
    \sum_{j=1}^{d}
    \E \left[
      |T(x)_j|^{2(1 \wedge \betamu)}
    \right]
    \lesssim
    \frac{1}{\lambda^{2(1 \wedge \betamu)} B},
  \end{align*}
  where we used the law of $T(x)_j$ from
  \citet[Proposition~1]{mourtada2020minimax}.
  Hence
  \begin{align*}
    \E \left[
      \big(
        \E \left[ \hat \mu(x) \mid \bX, \bT \right]
        - \E \left[ \hat \mu(x) \mid \bX \right]
      \big)^2
    \right]
    &\lesssim
    \frac{1}{\lambda^{2(1 \wedge \betamu)} B}.
  \end{align*}

  \proofparagraph{Showing the conditional bias converges}

  Now $\E \left[ \hat\mu(x) \mid \bX \right]$
  is a non-linear function of the i.i.d.\ random variables $X_i$,
  so we use the Efron--Stein inequality
  \citep{efron1981jackknife} to bound its variance.
  Let $\tilde X_{i j} = X_i$ if $i \neq j$ and be an
  independent copy of $X_j$, denoted $\tilde X_j$, if $i = j$.
  Write $\tilde \bX_j = (\tilde X_{1j}, \ldots, \tilde X_{n j})$
  and similarly
  $\tilde \I_{i j b}(x) = \I \big\{ \tilde X_{i j} \in T_b(x) \big\}$
  and $N_{j b}(x) = \sum_{i=1}^{n} \tilde \I_{i j b}(x)$.
  \begin{align}
    \nonumber
    &\Var \left[
      \sum_{i=1}^{n}
      \big( \mu(X_i) - \mu(x) \big)
      \E \left[
        \frac{\I_{i b}(x)}{N_b(x)}
        \Bigm| \bX
      \right]
    \right] \\
    \nonumber
    &\quad\leq
    \frac{1}{2}
    \sum_{j=1}^{n}
    \E \left[
      \left(
        \sum_{i=1}^{n}
        \big( \mu(X_i) - \mu(x) \big)
        \E \left[
          \frac{\I_{i b}(x)}{N_b(x)}
          \Bigm| \bX
        \right]
        - \sum_{i=1}^{n}
        \left( \mu(\tilde X_{i j}) - \mu(x) \right)
        \E \left[
          \frac{\tilde \I_{i j b}(x)}{\tilde N_{j b}(x)}
          \Bigm| \tilde \bX_j
        \right]
      \right)^2
    \right] \\
    \nonumber
    &\quad\leq
    \frac{1}{2}
    \sum_{j=1}^{n}
    \E \left[
      \left(
        \sum_{i=1}^{n}
        \left(
          \big( \mu(X_i) - \mu(x) \big)
          \frac{\I_{i b}(x)}{N_b(x)}
          - \left( \mu(\tilde X_{i j}) - \mu(x) \right)
          \frac{\tilde \I_{i j b}(x)}{\tilde N_{j b}(x)}
        \right)
      \right)^2
    \right] \\
    \nonumber
    &\quad\leq
    \sum_{j=1}^{n}
    \E \left[
      \left(
        \sum_{i \neq j}
        \big( \mu(X_i) - \mu(x) \big)
        \left(
          \frac{\I_{i b}(x)}{N_b(x)} - \frac{\I_{i b}(x)}{\tilde N_{j b}(x)}
        \right)
      \right)^2
    \right] \\
    \label{eq:bias_efron_stein}
    &\qquad+ 2 \sum_{j=1}^{n}
    \E \left[
      \left( \mu(X_j) - \mu(x) \right)^2
      \frac{\I_{j b}(x)}{N_b(x)^2}
    \right].
  \end{align}
  For the first term in \eqref{eq:bias_efron_stein} to be non-zero,
  we must have $|N_b(x) - \tilde N_{j b}(x)| = 1$.
  Writing $N_{-j b}(x) = \sum_{i \neq j} \I_{i b}(x)$,
  we may assume by symmetry that
  $\tilde N_{j b}(x) = N_{-j b}(x)$ and $N_b(x) = N_{-j b}(x) + 1$,
  and also that $\I_{j b}(x) = 1$.
  Hence since $f$ is bounded and $\mu \in \cH^{\betamu}$,
  writing $\I_{-j b}(x) = \I \left\{ N_{-j b}(x) \geq 1 \right\}$,
  by the H{\"o}lder inequality, Lemma~\ref{lem:largest_cell}
  and Lemma~\ref{lem:simple_moment_denominator},
  \begin{align*}
    &\sum_{j=1}^{n}
    \E \left[
      \left(
        \sum_{i \neq j}
        \left( \mu(X_i) - \mu(x) \right)
        \left(
          \frac{\I_{i b}(x)}{N_b(x)} - \frac{\I_{i b}(x)}{\tilde N_{j b}(x)}
        \right)
      \right)^2
    \right] \\
    &\quad\lesssim
    \sum_{j=1}^{n}
    \E \left[
      \max_{1 \leq l \leq d}
      |T_b(x)_l|^{2(1 \wedge \betamu)}
      \left(
        \frac{\sum_{i \neq j}\I_{i b}(x) \I_{j b}(x)}
        {N_{-j b}(x)(N_{-j b}(x) + 1)}
      \right)^2
    \right]
    \lesssim
    \E \left[
      \max_{1 \leq l \leq d}
      |T_b(x)_l|^{2(1 \wedge \betamu)}
      \frac{\I_{b}(x)}{N_{b}(x)}
    \right] \\
    &\lesssim
    \E \left[
      \max_{1 \leq l \leq d}
      |T_b(x)_l|^{6(1 \wedge \betamu)}
    \right]^{1/3}
    \E \left[
      \frac{\I_{b}(x)}{N_{b}(x)^{3/2}}
    \right]^{2/3}
    \lesssim
    \frac{1}{\lambda^{2(1 \wedge \betamu)}}
    \frac{\lambda^d}{n}.
  \end{align*}
  For the second term in \eqref{eq:bias_efron_stein} we
  again use $\mu \in \cH^{\betamu}$ to see
  \begin{align*}
    \sum_{j=1}^{n}
    \E \left[
      \left( \mu(X_j) - \mu(x) \right)^2
      \frac{\I_{j b}(x)}{N_b(x)^2}
    \right]
    &\lesssim
    \E \left[
      \max_{1 \leq l \leq d}
      |T_b(x)_l|^{2(1 \wedge \betamu)}
      \frac{\I_{b}(x)}{N_b(x)}
    \right]
    \lesssim
    \frac{1}{\lambda^{2(1 \wedge \betamu)}}
    \frac{\lambda^{d}}{n}
  \end{align*}
  in the same manner. Hence
  \begin{align*}
    \Var \left[
      \sum_{i=1}^{n}
      \left( \mu(X_i) - \mu(x) \right)
      \E \left[
        \frac{\I_{i b}(x)}{N_b(x)}
        \Bigm| \bX
      \right]
    \right]
    &\lesssim
    \frac{1}{\lambda^{2(1 \wedge \betamu)}}
    \frac{\lambda^{d}}{n},
  \end{align*}
  and so by the above and the previous part,
  \begin{align*}
    \E \left[
      \big(
        \E \left[ \hat \mu(x) \mid \bX, \bT \right]
        - \E \left[ \hat \mu(x) \right]
      \big)^2
    \right]
    &\lesssim
    \frac{1}{\lambda^{2(1 \wedge \betamu)} B}
    + \frac{1}{\lambda^{2(1 \wedge \betamu)}}
    \frac{\lambda^{d}}{n}.
  \end{align*}

  \proofparagraph{Computing the limiting bias}

  It remains to compute the limiting value of
  $\E \left[ \hat \mu(x) \right] - \mu(x)$.
  Let $\bX_{-i} = (X_1, \ldots, X_{i-1}, X_{i+1}, \ldots, X_n)$
  and $N_{-i b}(x) = \sum_{j=1}^n \I\{j \neq i\} \I\{X_j \in T_b(x)\}$.
  Then
  \begin{align*}
    \E \left[ \hat \mu(x) \right]
    - \mu(x)
    &=
    \E \left[
      \sum_{i=1}^{n}
      \left( \mu(X_i) - \mu(x) \right)
      \frac{\I_{i b}(x)}{N_b(x)}
    \right]
    =
    \sum_{i=1}^{n}
    \E \left[
      \E \left[
        \frac{\left( \mu(X_i) - \mu(x) \right)\I_{i b}(x)}
        {N_{-i b}(x) + 1}
        \bigm| \bT, \bX_{-i}
      \right]
    \right] \\
    &=
    n \,
    \E \left[
      \frac{\int_{T_b(x)} \left( \mu(s) - \mu(x) \right) f(s) \diffi s}
      {N_{-i b}(x) + 1}
    \right].
  \end{align*}
  By Lemma~\ref{lem:binomial_expectation}, as
  $N_{-i b}(x) \sim \Bin\left(n-1, \int_{T_b(x)} f(s) \diffi s \right)$
  given $\bT$ and $f$ is bounded away from zero,
  \begin{align*}
    \left|
    \E \left[
      \frac{1}{N_{-i b}(x) + 1}
      \Bigm| \bT
    \right]
    - \frac{1}{(n-1) \int_{T_b(x)} f(s) \diffi s + 1}
    \right|
    &\lesssim
    \frac{1}{n^2 \left( \int_{T_b(x)} f(s) \diffi s \right)^2}
    \wedge 1
    \lesssim
    \frac{1}{n^2 |T_b(x)|^2}
    \wedge 1
  \end{align*}
  and also
  \begin{align*}
    \left|
    \frac{1}{(n-1) \int_{T_b(x)} f(s) \diffi s + 1}
    - \frac{1}{n \int_{T_b(x)} f(s) \diffi s}
    \right|
    &\lesssim
    \frac{1}{n^2 \left( \int_{T_b(x)} f(s) \diffi s\right)^2}
    \wedge 1
    \lesssim
    \frac{1}{n^2 |T_b(x)|^2}
    \wedge 1.
  \end{align*}
  So by Lemma~\ref{lem:largest_cell}
  and Lemma~\ref{lem:moment_cell},
  since $\mu \in \cH^{\betamu}$ and $f$ is bounded below,
  using the H{\"o}lder inequality,
  \begin{align*}
    &\left|
    \E \left[ \hat \mu(x) \right]
    - \mu(x)
    - \E \left[
      \frac{\int_{T_b(x)} \left( \mu(s) - \mu(x) \right) f(s) \diffi s}
      {\int_{T_b(x)} f(s) \diffi s}
    \right]
    \right| \\
    &\quad\lesssim
    \E \left[
      \frac{n \int_{T_b(x)}
        \left| \mu(s) - \mu(x) \right|
      f(s) \diffi s}
      {n^2 |T_b(x)|^2 \vee 1}
    \right]
    \lesssim
    \E \left[
      \frac{\max_{1 \leq l \leq d} |T_b(x)_l|^{1 \wedge \betamu} }
      {n |T_b(x)| \vee 1}
    \right] \\
    &\quad\lesssim
    \E \left[
      \max_{1 \leq l \leq d} |T_b(x)_l|^{3(1 \wedge \betamu)}
    \right]^{1/3}
    \E \left[
      \frac{1} {n^{3/2} |T_b(x)|^{3/2} \vee 1}
    \right]^{2/3}
    \lesssim
    \frac{1}{\lambda^{1 \wedge \betamu}}
    \frac{\lambda^d}{n}.
  \end{align*}
  Next set
  $A = \frac{1}{f(x) |T_b(x)|} \int_{T_b(x)} (f(s) - f(x)) \diffi s
  \geq \inf_{s \in [0,1]^d} \frac{f(s)}{f(x)} - 1 > -1$.
  Use the Maclaurin series of $\frac{1}{1+x}$
  up to order $\flbeta - 1$ to see
  $\frac{1}{1 + A} = \sum_{k=0}^{\flbeta - 1} (-1)^k A^k
  + O \left( A^{\flbeta} \right)$.
  Hence
  \begin{align*}
    \E \left[
      \frac{\int_{T_b(x)} \left( \mu(s) - \mu(x) \right) f(s) \diffi s}
      {\int_{T_b(x)} f(s) \diffi s}
    \right]
    &=
    \E \left[
      \frac{\int_{T_b(x)} \left( \mu(s) - \mu(x) \right) f(s) \diffi s}
      {f(x) |{T_b(x)}|}
      \frac{1}{1 + A}
    \right] \\
    &=
    \E \left[
      \frac{\int_{T_b(x)} \left( \mu(s) - \mu(x) \right) f(s) \diffi s}
      {f(x) |{T_b(x)}|}
      \left(
        \sum_{k=0}^{\flbeta - 1}
        (-1)^k
        A^k
        + O \left( |A|^{\flbeta} \right)
      \right)
    \right].
  \end{align*}
  Note that since $\mu \in \cH^{\betamu}$
  and $f \in \cH^{\betaf}$,
  by integrating the tail probability given in Lemma~\ref{lem:largest_cell},
  the Maclaurin remainder term is bounded by
  \begin{align*}
    &\E \left[
      \frac{\int_{T_b(x)} \left| \mu(s) - \mu(x) \right| f(s) \diffi s}
      {f(x) |{T_b(x)}|}
      |A|^{\flbeta}
    \right] \\
    &\qquad=
    \E \left[
      \frac{\int_{T_b(x)} \left| \mu(s) - \mu(x) \right| f(s) \diffi s}
      {f(x) |{T_b(x)}|}
      \left(
        \frac{1}{f(x) |{T_b(x)}|} \int_{T_b(x)} (f(s) - f(x)) \diffi s
      \right)^{\flbeta}
    \right] \\
    &\qquad\lesssim
    \E \left[
      \max_{1 \leq l \leq d}
      |T_b(x)_l|^{1 \wedge \betamu}
      \left(
        \max_{1 \leq l \leq d}
        |T_b(x)_l|^{1 \wedge \betaf}
      \right)^{\flbeta}
    \right] \\
    &\qquad\lesssim
    \E \left[
      \max_{1 \leq l \leq d}
      |T_b(x)_l|^{1 \wedge \betamu + \flbeta (1 \wedge \betaf)}
    \right]
    \lesssim
    \E \left[
      \max_{1 \leq l \leq d}
      |T_b(x)_l|^{\beta}
    \right]
    \lesssim
    \frac{d}{\lambda^\beta}
    \int_0^\infty
    e^{-t} t^\beta \diffi t
    \lesssim
    \frac{1}{\lambda^\beta},
  \end{align*}
  where we used that
  $1 \wedge \betamu + \flbeta (1 \wedge \betaf) \geq \beta$.
  Hence to summarize the progress so far, we have
  \begin{align*}
    &\left|
    \E \left[
      \hat \mu(x)
    \right]
    - \mu(x)
    - \sum_{k=0}^{\flbeta-1}
    (-1)^k \,
    \E \left[
      \frac{\int_{T_b(x)} \left( \mu(s) - \mu(x) \right) f(s) \diffi s}
      {f(x)^{k+1} |T_b(x)|^{k+1}}
      \left(
        \int_{T_b(x)} (f(s) - f(x)) \diffi s
      \right)^k
    \right]
    \right| \\
    &\qquad\lesssim
    \frac{1}{\lambda^{1 \wedge \betamu}}
    \frac{\lambda^d}{n}
    + \frac{1}{\lambda^\beta}.
  \end{align*}
  We continue to evaluate this expectation.
  First, by Taylor's theorem and with
  $\nu$ a multi-index,
  since $f \in \cH^{\betaf}$,
  \begin{align*}
    \left(
      \int_{T_b(x)} (f(s) - f(x)) \diffi s
    \right)^k
    &=
    \left(
      \sum_{|\nu| = 1}^{\flbeta_f}
      \frac{\partial^\nu f(x)}{\nu !}
      \int_{T_b(x)}
      (s - x)^\nu
      \diffi s
    \right)^k
    + O \left(
      |T_b(x)|^k \max_{1 \leq l \leq d} |T_b(x)_l|^{\betaf}
    \right).
  \end{align*}
  Next, by the multinomial theorem
  with a multi-index $u$ indexed by $\nu$ with $|\nu| \geq 1$,
  \begin{align*}
    \left(
      \sum_{|\nu| = 1}^{\flbeta_f}
      \frac{\partial^\nu f(x)}{\nu !}
      \int_{T_b(x)}
      (s - x)^\nu
      \diffi s
    \right)^k
    &=
    \sum_{|u| = k}
    \binom{k}{u}
    \left(
      \frac{\partial^\nu f(x)}{\nu !}
      \int_{T_b(x)} (s-x)^\nu \diffi s
    \right)^u
  \end{align*}
  where $\binom{k}{u}$ is a multinomial coefficient.
  By Taylor's theorem with
  $\mu \in \cH^{\betamu}$ and
  $f \in \cH^{\betaf}$,
  and using
  $\betamu \wedge (1 \wedge \betamu + \betaf) \geq \beta$,
  \begin{align*}
    &\int_{T_b(x)} \left( \mu(s) - \mu(x) \right) f(s) \diffi s \\
    &\quad=
    \sum_{|\nu'|=1}^{\flbeta_\mu}
    \sum_{|\nu''|=0}^{\flbeta_f}
    \frac{\partial^{\nu'} \mu(x)}{\nu' !}
    \frac{\partial^{\nu''} f(x)}{\nu'' !}
    \int_{T_b(x)} (s-x)^{\nu' + \nu''} \diffi s
    + O \left( |T_b(x)| \max_{1 \leq l \leq d} |T_b(x)_l|^\beta \right).
  \end{align*}
  Now by integrating the tail probabilities in Lemma~\ref{lem:largest_cell},
  $ \E \left[ \max_{1 \leq l \leq d} |T_b(x)_l|^\beta \right]
  \lesssim \frac{1}{\lambda^\beta}$.
  Therefore by Lemma~\ref{lem:moment_cell},
  writing $T_b(x)^\nu$ for $\int_{T_b(x)} (s-x)^\nu \diffi s$,
  \begin{align*}
    &\sum_{k=0}^{\flbeta-1}
    (-1)^k \,
    \E \left[
      \frac{\int_{T_b(x)} \left( \mu(s) - \mu(x) \right) f(s) \diffi s}
      {f(x)^{k+1} |T_b(x)|^{k+1}}
      \left(
        \int_{T_b(x)} (f(s) - f(x)) \diffi s
      \right)^k
    \right] \\
    &\,=
    \sum_{k=0}^{\flbeta-1}
    (-1)^k \,
    \E \left[
      \frac{
        \sum_{|\nu'|=1}^{\flbeta_\mu}
        \sum_{|\nu''|=0}^{\flbeta_f}
        \frac{\partial^{\nu'} \mu(x)}{\nu' !}
        \frac{\partial^{\nu''} f(x)}{\nu'' !}
        T_b(x)^{\nu' + \nu''}
      }{f(x)^{k+1} |T_b(x)|^{k+1}}
      \sum_{|u| = k}
      \binom{k}{u}
      \left(
        \frac{\partial^\nu f(x)}{\nu !}
        T_b(x)^\nu
      \right)^u
    \right]
    + O \left(
      \frac{1}{\lambda^\beta}
    \right) \\
    &\,=
    \sum_{|\nu'|=1}^{\flbeta_\mu}
    \sum_{|\nu''|=0}^{\flbeta_f}
    \sum_{|u|=0}^{\flbeta-1}
    \frac{\partial^{\nu'} \mu(x)}{\nu' !}
    \frac{\partial^{\nu''} f(x)}{\nu'' !}
    \left( \frac{\partial^\nu f(x)}{\nu !} \right)^u
    \hspace*{-1mm}
    \binom{|u|}{u}
    \frac{(-1)^{|u|}}{f(x)^{|u|+1}}
    \E \hspace*{-1mm} \left[
      \frac{ T_b(x)^{\nu' + \nu''} (T_b(x)^\nu)^u}{|T_b(x)|^{|u|+1}}
    \right]
    \hspace*{-1mm}
    + O \left(
      \frac{1}{\lambda^\beta}
    \right) \hspace*{-1mm} .
  \end{align*}
  Now we show this is a polynomial in $\lambda$.
  For $1 \leq j \leq d$, define the independent variables
  $E_{1j*} \sim \Exp(1) \wedge (\lambda x_j)$
  and $E_{2j*} \sim \Exp(1) \wedge (\lambda (1-x_j))$
  so
  $T_b(x) = \prod_{j=1}^{d} [x_j - E_{1j*} / \lambda, x_j + E_{2j*} / \lambda]$.
  Then
  \begin{align*}
    T_b(x)^\nu
    &=
    \int_{T_b(x)} (s-x)^\nu \diffi s
    = \prod_{j=1}^d
    \int_{x_j - E_{1j*}/\lambda}^{x_j+E_{2j*}/\lambda}
    (s - x_j)^{\nu_j} \diffi s
    = \prod_{j=1}^d
    \int_{-E_{1j*}}^{E_{2j*}} (s / \lambda)^{\nu_j} 1/\lambda \diffi s \\
    &=
    \lambda^{-d - |\nu|}
    \prod_{j=1}^d
    \int_{-E_{1j*}}^{E_{2j*}} s^{\nu_j} \diffi s
    = \lambda^{-d - |\nu|}
    \prod_{j=1}^d
    \frac{E_{2j*}^{\nu_j + 1} + (-1)^{\nu_j} E_{1j*}^{\nu_j + 1}}
    {\nu_j + 1}.
  \end{align*}
  So by independence over $j$,
  \begin{align}
    \nonumber
    &\E \left[
      \frac{ T_b(x)^{\nu' + \nu''} (T_b(x)^\nu)^u}{|T_b(x)|^{|u|+1}}
    \right] \\
    \label{eq:bias_calc}
    &\quad=
    \lambda^{- |\nu'| - |\nu''| - |\nu| \cdot u}
    \prod_{j=1}^d
    \E \left[
      \frac{E_{2j*}^{\nu'_j + \nu''_j + 1}
      + (-1)^{\nu'_j + \nu''_j} E_{1j*}^{\nu'_j + \nu''_j + 1}}
      {(\nu'_j + \nu''_j + 1) (E_{2j*} + E_{1j*})}
      \frac{\left(E_{2j*}^{\nu_j + 1}
      + (-1)^{\nu_j} E_{1j*}^{\nu_j + 1}\right)^u}
      {(\nu_j + 1)^u (E_{2j*} + E_{1j*})^{|u|}}
    \right].
  \end{align}
  The final step is to replace $E_{1j*}$
  by $E_{1j} \sim \Exp(1)$ and similarly for $E_{2j*}$.
  Note that for a positive constant $C$,
  \begin{align*}
    \P \left(
      \bigcup_{j=1}^{d}
      \left(
        \left\{
          E_{1j*} \neq E_{1j}
        \right\}
        \cup
        \left\{
          E_{2j*} \neq E_{2j}
        \right\}
      \right)
    \right)
    &\leq
    2d\,
    \P \left(
      \Exp(1) \geq \lambda \min_{1 \leq j \leq d}
      (x_j \wedge (1-x_j))
    \right)
    \leq
    2d e^{-C \lambda}.
  \end{align*}
  Further, the quantity inside the expectation in \eqref{eq:bias_calc}
  is bounded almost surely by one and so
  the error incurred by replacing
  $E_{1j*}$ and $E_{2j*}$ by $E_{1j}$ and $E_{2j}$
  in \eqref{eq:bias_calc}
  is at most $2 d e^{-C \lambda} \lesssim \lambda^{-\beta}$.
  Thus the limiting bias is
  \begin{align}
    \label{eq:bias_proof}
    &\E \left[ \hat \mu(x) \right]
    - \mu(x)
    =
    \sum_{|\nu'|=1}^{\flbeta_\mu}
    \sum_{|\nu''|=0}^{\flbeta_f}
    \sum_{|u|=0}^{\flbeta - 1}
    \frac{\partial^{\nu'} \mu(x)}{\nu' !}
    \frac{\partial^{\nu''} f(x)}{\nu'' !}
    \left( \frac{\partial^\nu f(x)}{\nu !} \right)^u
    \binom{|u|}{u}
    \frac{(-1)^{|u|}}{f(x)^{|u|+1}}
    \, \lambda^{- |\nu'| - |\nu''| - |\nu| \cdot u} \\
    \nonumber
    &\ \times
    \prod_{j=1}^d
    \E \left[
      \frac{E_{2j}^{\nu'_j + \nu''_j + 1}
      + (-1)^{\nu'_j + \nu''_j} E_{1j}^{\nu'_j + \nu''_j + 1}}
      {(\nu'_j + \nu''_j + 1) (E_{2j} + E_{1j})}
      \frac{\left(E_{2j}^{\nu_j + 1}
      + (-1)^{\nu_j} E_{1j}^{\nu_j + 1}\right)^u}
      {(\nu_j + 1)^u (E_{2j} + E_{1j})^{|u|}}
    \right]
    + O \left( \frac{1}{\lambda^{1 \wedge \betamu}}
      \frac{\lambda^d}{n}
    + \frac{1}{\lambda^\beta} \right),
  \end{align}
  recalling that $u$ is a multi-index which is indexed by the multi-index $\nu$.
  This is a polynomial in $\lambda$ of degree at most $\flbeta$,
  since higher-order terms can be absorbed into $O(1 / \lambda^\beta)$,
  which has finite coefficients depending only on
  the derivatives up to order $\flbeta \leq \flbeta_\mu$
  and $\flbeta - 1 \leq \flbeta_f$
  of $\mu$ and $f$ respectively at $x$.
  Now we show that the odd-degree terms in this polynomial are all zero.
  Note that a term is of odd degree if and only if
  $|\nu'| + |\nu''| + |\nu| \cdot u$ is odd.
  This implies that there exists $1 \leq j \leq d$ such that
  exactly one of either
  $\nu'_j + \nu''_j$ is odd or
  $\sum_{|\nu|=1}^{\flbeta-1} \nu_j u_\nu$ is odd.

  If $\nu'_j + \nu''_j$ is odd, then
  $\sum_{|\nu|=1}^{\flbeta-1} \nu_j u_\nu$ is even, so
  $|\{\nu : \nu_j u_\nu \text{ is odd}\}|$ is even.
  Consider the effect of swapping $E_{1j}$ and $E_{2j}$,
  an operation which by independence preserves their joint law,
  in each of
  \begin{align}
    \label{eq:bias_odd_1}
    \frac{E_{2j}^{\nu'_j + \nu''_j + 1}
    + (-1)^{\nu'_j + \nu''_j} E_{1j}^{\nu'_j + \nu''_j + 1}}
    {E_{2j} + E_{1j}}
  \end{align}
  and
  \begin{align}
    \label{eq:bias_odd_2}
    \frac{\left(E_{2j}^{\nu_j + 1}
    + (-1)^{\nu_j} E_{1j}^{\nu_j + 1}\right)^u}
    {(E_{2j} + E_{1j})^{|u|}}
    =
    \prod_{\substack{|\nu| = 1 \\
    \nu_j u_\nu \text{ even}}}^{\flbeta - 1}
    \frac{\left(E_{2j}^{\nu_j + 1}
    + (-1)^{\nu_j} E_{1j}^{\nu_j + 1}\right)^{u_\nu}}
    {(E_{2j} + E_{1j})^{u_\nu}}
    \prod_{\substack{|\nu| = 1 \\
    \nu_j u_\nu \text{ odd}}}^{\flbeta - 1}
    \frac{\left(E_{2j}^{\nu_j + 1}
    + (-1)^{\nu_j} E_{1j}^{\nu_j + 1}\right)^{u_\nu}}
    {(E_{2j} + E_{1j})^{u_\nu}}.
  \end{align}
  Clearly $\nu'_j + \nu''_j$ being odd inverts the
  sign of \eqref{eq:bias_odd_1}.
  For \eqref{eq:bias_odd_2},
  each term in the first product has either
  $\nu_j$ even or $u_\nu$ even, so its sign is preserved.
  Every term in the second product of \eqref{eq:bias_odd_2}
  has its sign inverted due to both $\nu_j$ and $u_\nu$ being odd,
  but there are an even number of terms,
  preserving the overall sign.
  Therefore the expected product
  of \eqref{eq:bias_odd_1} and \eqref{eq:bias_odd_2} is zero by symmetry.

  If however $\nu'_j + \nu''_j$ is even, then
  $\sum_{|\nu|=1}^{\flbeta - 1} \nu_j u_\nu$ is odd so
  $|\{\nu : \nu_j u_\nu \text{ is odd}\}|$ is odd.
  Clearly the sign of \eqref{eq:bias_odd_1} is preserved.
  Again the sign of the first product in \eqref{eq:bias_odd_2}
  is preserved, and the sign of every term in \eqref{eq:bias_odd_2}
  is inverted. However there are now an odd number of terms in the
  second product, so its overall sign is inverted.
  Therefore the expected product
  of \eqref{eq:bias_odd_1} and \eqref{eq:bias_odd_2} is again zero.

  \proofparagraph{Calculating the second-order bias}

  Next we calculate some special cases, beginning with
  the form of the leading second-order bias,
  where the exponent in $\lambda$ is
  $|\nu'| + |\nu''| + u \cdot |\nu| = 2$,
  proceeding by cases on the values of $|\nu'|$, $|\nu''|$, and $|u|$.
  Firstly, if $|\nu'| = 2$ then $|\nu''| = |u| = 0$.
  Note that if any $\nu'_j = 1$ then the expectation in \eqref{eq:bias_proof}
  is zero.
  Hence we can assume $\nu'_j \in \{0, 2\}$, yielding
  \begin{align*}
    \frac{1}{2 \lambda^2}
    \sum_{j=1}^d
    \frac{\partial^2 \mu(x)}{\partial x_j^2}
    \frac{1}{3}
    \E \left[
      \frac{E_{2j}^{3} + E_{1j}^{3}} {E_{2j} + E_{1j}}
    \right]
    &=
    \frac{1}{2 \lambda^2}
    \sum_{j=1}^d
    \frac{\partial^2 \mu(x)}{\partial x_j^2}
    \frac{1}{3}
    \E \left[
      E_{1j}^{2}
      + E_{2j}^{2}
      - E_{1j} E_{2j}
    \right]
    =
    \frac{1}{2 \lambda^2}
    \sum_{j=1}^d
    \frac{\partial^2 \mu(x)}{\partial x_j^2},
  \end{align*}
  where we used that $E_{1j}$ and $E_{2j}$ are independent $\Exp(1)$.
  Next we consider $|\nu'| = 1$ and $|\nu''| = 1$, so $|u| = 0$.
  Note that if $\nu'_j = \nu''_{j'} = 1$ with $j \neq j'$ then the
  expectation in \eqref{eq:bias_proof} is zero.
  So we need only consider $\nu'_j = \nu''_j = 1$, giving
  \begin{align*}
    \frac{1}{\lambda^2}
    \frac{1}{f(x)}
    \sum_{j=1}^{d}
    \frac{\partial \mu(x)}{\partial x_j}
    \frac{\partial f(x)}{\partial x_j}
    \frac{1}{3}
    \E \left[
      \frac{E_{2j}^{3} + E_{1j}^{3}}
      {E_{2j} + E_{1j}}
    \right]
    &=
    \frac{1}{\lambda^2}
    \frac{1}{f(x)}
    \sum_{j=1}^{d}
    \frac{\partial \mu(x)}{\partial x_j}
    \frac{\partial f(x)}{\partial x_j}.
  \end{align*}
  Finally we have the case where $|\nu'| = 1$, $|\nu''| = 0$
  and $|u|=1$. Then $u_\nu = 1$ for some $|\nu| = 1$ and zero otherwise.
  Note that if $\nu'_j = \nu_{j'} = 1$ with $j \neq j'$ then the
  expectation is zero. So we need only consider $\nu'_j = \nu_j = 1$, giving
  \begin{align*}
    &- \frac{1}{\lambda^2}
    \frac{1}{f(x)}
    \sum_{j=1}^{d}
    \frac{\partial \mu(x)}{\partial x_j}
    \frac{\partial f(x)}{\partial x_j}
    \frac{1}{4}
    \E \left[
      \frac{(E_{2j}^2 - E_{1j}^2)^2}
      {(E_{2j} + E_{1j})^2}
    \right] \\
    &\quad=
    - \frac{1}{4 \lambda^2}
    \frac{1}{f(x)}
    \sum_{j=1}^{d}
    \frac{\partial \mu(x)}{\partial x_j}
    \frac{\partial f(x)}{\partial x_j}
    \E \left[
      E_{1j}^2
      + E_{2j}^2
      - 2 E_{1j} E_{2j}
    \right]
    =
    - \frac{1}{2 \lambda^2}
    \frac{1}{f(x)}
    \sum_{j=1}^{d}
    \frac{\partial \mu(x)}{\partial x_j}
    \frac{\partial f(x)}{\partial x_j}.
  \end{align*}
  Hence the second-order bias term is
  \begin{align*}
    \frac{1}{2 \lambda^2}
    \sum_{j=1}^d
    \frac{\partial^2 \mu(x)}{\partial x_j^2}
    + \frac{1}{2 \lambda^2}
    \frac{1}{f(x)}
    \sum_{j=1}^{d}
    \frac{\partial \mu(x)}{\partial x_j}
    \frac{\partial f(x)}{\partial x_j}.
  \end{align*}

  \proofparagraph{Calculating the bias if the data is uniformly distributed}

  If $X_i \sim \Unif\big([0,1]^d\big)$ then $f(x) = 1$ and
  the bias expansion from \eqref{eq:bias_proof} becomes
  \begin{align*}
    \sum_{|\nu'|=1}^{\flbeta}
    \lambda^{- |\nu'|}
    \frac{\partial^{\nu'} \mu(x)}{\nu' !}
    \prod_{j=1}^d
    \E \left[
      \frac{E_{2j}^{\nu'_j + 1}
      + (-1)^{\nu'_j} E_{1j}^{\nu'_j + 1}}
      {(\nu'_j + 1) (E_{2j} + E_{1j})}
    \right].
  \end{align*}
  Note that this is zero if any $\nu_j'$ is odd.
  Therefore we can group these terms based on the exponent of $\lambda$ to see
  \begin{align*}
    \frac{B_r(x)}{\lambda^{2r}}
    &=
    \frac{1}{\lambda^{2r}}
    \sum_{|\nu|=r}
    \frac{\partial^{2 \nu} \mu(x)}{(2 \nu) !}
    \prod_{j=1}^d
    \frac{1}{2\nu_j + 1}
    \E \left[
      \frac{E_{2j}^{2\nu_j + 1} + E_{1j}^{2\nu_j + 1}}
      {E_{2j} + E_{1j}}
    \right].
  \end{align*}
  Since $\int_0^\infty \frac{e^{-t}}{a+t} \diffi t = e^a \Gamma(0,a)$
  and $\int_0^\infty s^a \Gamma(0, a) \diffi s = \frac{a!}{a+1}$,
  with $\Gamma(0, a) = \int_a^\infty \frac{e^{-t}}{t} \diffi t$
  the upper incomplete gamma function,
  the expectation is easily calculated as
  \begin{align*}
    \E \left[
      \frac{E_{2j}^{2\nu_j + 1} + E_{1j}^{2\nu_j + 1}}
      {E_{2j} + E_{1j}}
    \right]
    &=
    2
    \int_{0}^{\infty}
    s^{2\nu_j + 1}
    e^{-s}
    \int_{0}^{\infty}
    \frac{e^{-t}}
    {s + t}
    \diffi t
    \diffi s
    =
    2
    \int_{0}^{\infty}
    s^{2\nu_j + 1}
    \Gamma(0, s)
    \diffi s
    =
    \frac{(2 \nu_j + 1)!}{\nu_j + 1},
  \end{align*}
  so
  \begin{align*}
    \frac{B_r(x)}{\lambda^{2r}}
    &=
    \frac{1}{\lambda^{2r}}
    \sum_{|\nu|=r}
    \frac{\partial^{2 \nu} \mu(x)}{(2 \nu) !}
    \prod_{j=1}^d
    \frac{1}{2\nu_j + 1}
    \frac{(2 \nu_j + 1)!}{\nu_j + 1}
    =
    \frac{1}{\lambda^{2r}}
    \sum_{|\nu|=r}
    \partial^{2 \nu} \mu(x)
    \prod_{j=1}^d
    \frac{1}{\nu_j + 1}.
  \end{align*}
\end{proof}

\begin{proof}[Lemma~\ref{lem:variance}]
  By Lemma~\ref{lem:variance_debiased} with $J=0$,
  $a_0 = 1$, and $\omega_0 = 1$.
\end{proof}

\begin{proof}[Theorem~\ref{thm:rate}]
  By Theorem~\ref{thm:minimax} with $J=0$, $a_0 = 1$, and $\omega_0 = 1$.
\end{proof}

\begin{proof}[Theorem~\ref{thm:clt}]
  By Theorem~\ref{thm:clt_debiased} with $J=0$,
  $a_0 = 1$, and $\omega_0 = 1$.
\end{proof}

\begin{proof}[Lemma~\ref{lem:variance_estimation}]
  By Lemma~\ref{lem:variance_estimation_debiased}
  with $J=0$, $a_0 = 1$, and $\omega_0 = 1$.
\end{proof}

\begin{proof}[Theorem~\ref{thm:confidence}]
  By Theorem~\ref{thm:confidence_debiased}
  with $J=0$, $a_0 = 1$, and $\omega_0 = 1$,
  replacing $\beta$ by $2 \wedge \beta$.
\end{proof}

\subsection{Proofs for Section~\ref{sec:debiased}}
\label{sec:proofs_debiased}

We give rigorous proofs of the bias and variance
characterizations, minimax optimality, central limit theorem,
variance estimation, and confidence interval validity
results for the debiased Mondrian random forest estimator.

The bias characterization of Lemma~\ref{lem:bias_debiased} with debiasing is
a purely algebraic consequence of the original bias characterization and the
construction of the debiased Mondrian random forest estimator.

\begin{proof}[Lemma~\ref{lem:bias_debiased}]

  By the definition of the debiased estimator and Lemma~\ref{lem:bias},
  as $J$ and $a_r$ are fixed,
  \begin{align*}
    &\E \left[
      \Bigg(
        \E \left[ \hat \mu_\rd(x) \mid \bX, \bT \right]
        -
        \sum_{l=0}^{J}
        \omega_l
        \Bigg(
          \mu(x)
          + \sum_{r=1}^{\lfloor \flbeta / 2 \rfloor}
          \frac{B_r(x)}{a_l^{2r} \lambda^{2r}}
        \Bigg)
      \Bigg)^2
    \right] \\
    &\quad=
    \E \left[
      \Bigg(
        \sum_{l=0}^{J}
        \omega_l
        \Bigg(
          \E \left[ \hat \mu_l(x) \mid \bX, \bT \right]
          -
          \mu(x)
          - \sum_{r=1}^{\lfloor \flbeta / 2 \rfloor}
          \frac{B_r(x)}{a_l^{2r} \lambda^{2r}}
        \Bigg)
      \Bigg)^2
    \right] \\
    &\quad\lesssim
    \sum_{l=0}^{J}
    \E \left[
      \Bigg(
        \E \left[ \hat \mu_l(x) \mid \bX, \bT \right]
        -
        \mu(x)
        - \sum_{r=1}^{\lfloor \flbeta / 2 \rfloor}
        \frac{B_r(x)}{a_l^{2r} \lambda^{2r}}
      \Bigg)^2
    \right]
    \lesssim
    \frac{1}{\lambda^{2 \beta}}
    + \frac{1}{\lambda^{2(1 \wedge \beta)} B}
    + \frac{1}{\lambda^{2(1 \wedge \beta)}}
    \frac{\lambda^d}{n}.
  \end{align*}
  It remains to evaluate the resulting bias.
  Recalling that $A_{r s} = a_{r-1}^{2 - 2s}$
  and $A \omega = e_0$, we have
  \begin{align*}
    \sum_{l=0}^J
    \omega_l
    \left(
      \mu(x)
      + \sum_{r=1}^{\lfloor \flbeta / 2 \rfloor}
      \frac{B_r(x)}{a_l^{2r} \lambda^{2r}}
    \right)
    &=
    \mu(x)
    \sum_{l=0}^J
    \omega_l
    +
    \sum_{r=1}^{\lfloor \flbeta / 2 \rfloor}
    \frac{B_r(x)}{\lambda^{2r}}
    \sum_{l=0}^J
    \frac{\omega_l}{a_l^{2r}} \\
    &=
    \mu(x)
    (A \omega)_1
    + \sum_{r=1}^{\lfloor \flbeta / 2 \rfloor \wedge J}
    \frac{B_r(x)}{\lambda^{2r}}
    (A \omega)_{r+1}
    + \sum_{r = (\lfloor \flbeta / 2 \rfloor \wedge J) + 1}
    ^{\lfloor \flbeta / 2 \rfloor}
    \frac{B_r(x)}{\lambda^{2r}}
    \sum_{l=0}^J
    \frac{\omega_l}{a_l^{2r}} \\
    &=
    \mu(x)
    + \I\{\lfloor \flbeta / 2 \rfloor \geq J + 1\}
    \frac{B_{J+1}(x)}{\lambda^{2J + 2}}
    \sum_{l=0}^J
    \frac{\omega_l}{a_l^{2J + 2}}
    + O \left( \frac{1}{\lambda^{2J + 4}} \right) \\
    &=
    \mu(x)
    + \I\{2J + 2 < \beta\}
    \frac{\bar\omega B_{J+1}(x)}{\lambda^{2J + 2}}
    + O \left( \frac{1}{\lambda^{2J + 4}} \right).
  \end{align*}
\end{proof}

\begin{proof}[Lemma~\ref{lem:variance_debiased}]
  Firstly, note that with $\sigma^2_i = \sigma^2(X_i)$ for brevity,
  \begin{align*}
    \tilde \Sigma_\rd(x)
    &=
    \frac{n}{\lambda^d}
    \Var \left[
      \sum_{r=0}^{J}
      \omega_r
      \frac{1}{B}
      \sum_{b=1}^B
      \sum_{i=1}^n
      \frac{ Y_i \I\big\{ X_i \in T_{b r}(x) \big\}}
      {N_{b r}(x)}
      \Bigm| \bX, \bT
    \right] \\
    &=
    \frac{n}{\lambda^d}
    \sum_{i=1}^n
    \sum_{r=0}^{J}
    \sum_{r'=0}^{J}
    \omega_r
    \omega_{r'}
    \frac{1}{B^2}
    \sum_{b=1}^B
    \sum_{b'=1}^B
    \frac{\I_{i b r}(x) \I_{i b' r'}(x) \sigma_i^2}
    {N_{b r}(x) N_{b' r'}(x)}.
  \end{align*}

  \proofparagraph{bounding the variance of $\tilde\Sigma_\rd(x)$}

  \begin{align}
    \nonumber
    \Var \left[
      \tilde \Sigma_\rd(x)
    \right]
    &=
    \Var \left[
      \frac{n}{\lambda^d}
      \sum_{i=1}^{n}
      \sum_{r=0}^{J}
      \sum_{r'=0}^{J}
      \omega_{r}
      \omega_{r'}
      \frac{1}{B^2}
      \sum_{b=1}^{B}
      \sum_{b'=1}^{B}
      \frac{\I_{i b r}(x) \I_{i b' r'}(x) \sigma_i^2}
      {N_{b r}(x) N_{b' r'}(x)}
    \right] \\
    &\nonumber
    \lesssim
    \frac{n^2}{\lambda^{2d}}
    \frac{1}{B^4}
    \Var \left[
      \sum_{i=1}^{n}
      \sum_{b=1}^{B}
      \sum_{b'=1}^{B}
      \frac{\I_{i b r}(x) \I_{i b' r'}(x) \sigma_i^2}
      {N_{b r}(x) N_{b' r'}(x)}
    \right] \\
    &\nonumber
    \lesssim
    \frac{n^2}{\lambda^{2d}}
    \frac{1}{B^4}
    \E \left[
      \Var \left[
        \sum_{i=1}^{n}
        \sum_{b=1}^{B}
        \sum_{b'=1}^{B}
        \frac{\I_{i b r}(x) \I_{i b' r'}(x) \sigma_i^2}
        {N_{b r}(x) N_{b' r'}(x)}
        \Bigm| \bX
      \right]
    \right] \\
    &\label{eq:clt_total_variance}
    \quad+
    \frac{n^2}{\lambda^{2d}}
    \frac{1}{B^4}
    \Var \left[
      \E \left[
        \sum_{i=1}^{n}
        \sum_{b=1}^{B}
        \sum_{b'=1}^{B}
        \frac{\I_{i b r}(x) \I_{i b' r'}(x) \sigma_i^2}
        {N_{b r}(x) N_{b' r'}(x)}
        \Bigm| \bX
      \right]
    \right].
  \end{align}
  For the first term in \eqref{eq:clt_total_variance},
  \begin{align*}
    &\E \left[
      \Var \left[
        \sum_{i=1}^{n}
        \sum_{b=1}^{B}
        \sum_{b'=1}^{B}
        \frac{\I_{i b r}(x) \I_{i b' r'}(x) \sigma_i^2}
        {N_{b r}(x) N_{b' r'}(x)}
        \Bigm| \bX
      \right]
    \right]
    = \sum_{i=1}^{n}
    \sum_{j=1}^{n}
    \sum_{b=1}^{B}
    \sum_{b'=1}^{B}
    \sum_{\tilde b=1}^{B}
    \sum_{\tilde b'=1}^{B} \\
    &\ \E \left[
      \sigma^2_i
      \sigma^2_j
      \left(
        \frac{\I_{i b r}(x)} {N_{b r}(x)}
        \frac{\I_{i b' r'}(x)} {N_{b' r'}(x)}
        \hspace*{-1mm}
        - \E \left[
          \frac{\I_{i b r}(x)} {N_{b r}(x)}
          \frac{\I_{i b' r'}(x)} {N_{b' r'}(x)}
          \hspace*{-1mm}
          \Bigm|
          \hspace*{-1mm}
          \bX
        \right]
      \right)
      \hspace*{-1mm}
      \bigg(
        \frac{\I_{j \tilde b r}(x)} {N_{\tilde b r}(x)}
        \frac{\I_{j \tilde b' r'}(x)} {N_{\tilde b' r'}(x)}
        \hspace*{-1mm}
        - \E \bigg[
          \frac{\I_{j \tilde b r}(x)} {N_{\tilde b r}(x)}
          \frac{\I_{j \tilde b' r'}(x)} {N_{\tilde b' r'}(x)}
          \hspace*{-1mm}
          \Bigm|
          \hspace*{-1mm}
          \bX
        \bigg]
      \bigg)
    \right].
  \end{align*}
  Since $T_{b r}$ is independent of $T_{b' r'}$ given
  $\bX$, the summands are zero
  whenever $|\{b, b', \tilde b, \tilde b'\}| = 4$.
  Further, by the Cauchy--Schwarz inequality
  and Lemma~\ref{lem:simple_moment_denominator},
  \begin{align*}
    &\frac{n^2}{\lambda^{2d}}
    \frac{1}{B^4}
    \E \left[
      \Var \left[
        \sum_{i=1}^{n}
        \sum_{b=1}^{B}
        \sum_{b'=1}^{B}
        \frac{\I_{i b r}(x)} {N_{b r}(x)}
        \frac{\I_{i b' r'}(x)} {N_{b' r'}(x)}
        \sigma^2(X_i)
        \Bigm| \bX
      \right]
    \right] \\
    &\quad\lesssim
    \frac{n^2}{\lambda^{2d}}
    \frac{1}{B^3}
    \sum_{b=1}^{B}
    \sum_{b'=1}^{B}
    \E \left[
      \left(
        \sum_{i=1}^{n}
        \frac{\I_{i b r}(x)} {N_{b r}(x)}
        \frac{\I_{i b' r'}(x)} {N_{b' r'}(x)}
      \right)^2
    \right]
    \lesssim
    \frac{n^2}{\lambda^{2d}}
    \frac{1}{B^3}
    \sum_{b=1}^{B}
    \sum_{b'=1}^{B}
    \E \left[
      \frac{\I_{b r}(x)} {N_{b r}(x)}
      \frac{\I_{b' r'}(x)} {N_{b' r'}(x)}
    \right] \\
    &\quad\lesssim
    \frac{n^2}{\lambda^{2d}}
    \frac{1}{B^3}
    \left(
      B^2
      \frac{\lambda^{2d}}{n^2}
      + B
      \frac{\lambda^{2d} (\log n)^d}{n^2}
    \right)
    \lesssim
    \frac{1}{B}
    + \frac{(\log n)^d}{B^2}
    \lesssim
    \frac{1}{B}.
  \end{align*}
  For the second term in \eqref{eq:clt_total_variance},
  the random variable inside the variance is a nonlinear
  function of the i.i.d.\ variables $X_i$,
  so we apply the Efron--Stein inequality
  \citep{efron1981jackknife}.
  Let $\hat X_{i j} = X_i$ if $i \neq j$ and be an
  independent copy of $X_j$, denoted $\hat X_j$, if $i = j$,
  and define $\sigma^2_{i j} = \sigma^2(\hat X_{i j})$.
  Write
  $\hat \I_{i j b r}(x) = \I \big\{ \hat X_{i j} \in T_{b r}(x) \big\}$
  and
  $\hat \I_{j b r}(x) = \I \big\{ \hat X_{j} \in T_{b r}(x) \big\}$,
  and also
  $\hat N_{j b r}(x) = \sum_{i=1}^{n} \hat \I_{i j b r}(x)$,
  We use the leave-one-out notation
  $N_{-j b r}(x) = \sum_{i \neq j} \I_{i b r}(x)$
  and also write
  $N_{-j b r \cap b' r'} = \sum_{i \neq j} \I_{i b r}(x) \I_{i b' r'}(x)$.
  \begin{align}
    \nonumber
    &\frac{n^2}{\lambda^{2d}}
    \frac{1}{B^4}
    \Var \left[
      \E \left[
        \sum_{i=1}^{n}
        \sum_{b=1}^{B}
        \sum_{b'=1}^{B}
        \frac{\I_{i b r}(x)} {N_{b r}(x)}
        \frac{\I_{i b' r'}(x)} {N_{b' r'}(x)}
        \sigma^2_i
        \Bigm| \bX
      \right]
    \right]
    \lesssim
    \frac{n^2}{\lambda^{2d}}
    \Var \left[
      \E \left[
        \sum_{i=1}^{n}
        \frac{\I_{i b r}(x)} {N_{b r}(x)}
        \frac{\I_{i b' r'}(x)} {N_{b' r'}(x)}
        \sigma^2_i
        \Bigm| \bX
      \right]
    \right] \\
    \nonumber
    &\quad\lesssim
    \frac{n^2}{\lambda^{2d}}
    \sum_{j=1}^{n}
    \E \left[
      \left(
        \sum_{i=1}^n
        \left(
          \frac{\I_{i b r}(x) \I_{i b' r'}(x) \sigma_i^2}
          {N_{b r}(x) N_{b' r'}(x)}
          - \frac{\hat \I_{i j b r}(x) \hat \I_{i j b' r'}(x)
          \hat \sigma_{i j}^2}
          {\hat N_{j b r}(x) \hat N_{j b' r'}(x)}
        \right)
      \right)^2
    \right] \\
    \nonumber
    &\quad\lesssim
    \frac{n^2}{\lambda^{2d}}
    \sum_{j=1}^{n}
    \E \left[
      \left(
        \left|
        \frac{1}
        {N_{b }(x) N_{b' r'}(x)}
        - \frac{1}
        {\hat N_{j b r}(x) \hat N_{j b' r'}(x)}
        \right|
        \sum_{i \neq j}
        \I_{i b r}(x) \I_{i b' r'}(x) \sigma_i^2
      \right)^2
    \right] \\
    \nonumber
    &\qquad+
    \frac{n^2}{\lambda^{2d}}
    \sum_{j=1}^{n}
    \E \left[
      \left(
        \frac{\I_{j b r}(x) \I_{j b' r'}(x) \sigma_j^2}
        {N_{b r}(x) N_{b' r'}(x)}
        - \frac{\hat \I_{j b r}(x) \hat \I_{j b' r'}(x)
        \hat \sigma_{j j}^2}
        {\hat N_{j b r}(x) \hat N_{j b' r'}(x)}
      \right)^2
    \right] \\
    \label{eq:berry_esseen_efron_stein}
    &\quad\lesssim
    \frac{n^2}{\lambda^{2d}}
    \sum_{j=1}^{n}
    \E \left[
      N_{-j b r \cap b' r'}(x)^2
      \left|
      \frac{1}
      {N_{b r}(x) N_{b' r'}(x)}
      - \frac{1}
      {\hat N_{j b r}(x) \hat N_{j b' r'}(x)}
      \right|^2
      \hspace*{-1mm}
      + \frac{\I_{j b r}(x) \I_{j b' r'}(x)}
      {N_{b r}(x)^2 N_{b' r'}(x)^2}
    \right].
  \end{align}
  For the first term in \eqref{eq:berry_esseen_efron_stein},
  note that since $|N_{b r}(x) - \hat N_{j b r}(x)| \leq
  \I_{j b r}(x) + \hat \I_{j b r}(x)$
  and similarly $|N_{b' r'}(x) - \hat N_{j b' r'}(x)| \leq
  \I_{j b' r'}(x) + \hat \I_{j b' r'}(x)$,
  \begin{align*}
    &\left|
    \frac{1}{N_{b r}(x) N_{b' r'}(x)}
    - \frac{1} {\hat N_{j b r}(x) \hat N_{j b' r'}(x)}
    \right| \\
    &\quad\leq
    \frac{1}{N_{b r}(x)}
    \left|
    \frac{1} {N_{b' r'}(x)} - \frac{1} {\hat N_{j b' r'}(x)}
    \right|
    + \frac{1}{\hat N_{j b' r'}(x)}
    \left|
    \frac{1} {N_{b r}(x)} - \frac{1} {\hat N_{j b r}(x)}
    \right| \\
    &\quad\leq
    \frac{\I_{j b r}(x) + \hat \I_{j b r}(x)}
    {N_{-j b r}(x) N_{-j b' r'}(x)^2}
    + \frac{\I_{j b' r'}(x) + \hat \I_{j b' r'}(x)}
    {N_{-j b' r'}(x) N_{-j b r}(x)^2}.
  \end{align*}
  Therefore by Lemma~\ref{lem:simple_moment_denominator},
  \begin{align*}
    &\frac{n^2}{\lambda^{2d}}
    \frac{1}{B^4}
    \Var \left[
      \E \left[
        \sum_{i=1}^{n}
        \sum_{b=1}^{B}
        \sum_{b'=1}^{B}
        \frac{\I_{i b r}(x) \I_{i b' r'}(x) \sigma^2(X_i)}
        {N_{b r}(x) N_{b' r'}(x)}
        \Bigm| \bX
      \right]
    \right] \\
    &\quad\lesssim
    \frac{n^2}{\lambda^{2d}}
    \sum_{j=1}^{n}
    \E \left[
      \frac{\I_{j b r}(x) \I_{b r \cap b' r'}(x)}
      {N_{b r}(x)^2 N_{b' r'}(x)^2}
    \right]
    \lesssim
    \frac{n^2}{\lambda^{2d}}
    \E \left[
      \frac{\I_{b r}(x) \I_{b' r'}(x)}
      {N_{b r}(x)^{3/2} N_{b' r'}(x)^{3/2}}
    \right]
    \lesssim
    \frac{n^2}{\lambda^{2d}}
    \frac{\lambda^{3d}}{n^3}
    \lesssim
    \frac{\lambda^{d}}{n}.
  \end{align*}
  We deduce that
  \begin{align*}
    \Var \left[
      \tilde \Sigma_\rd(x)
    \right]
    &\lesssim
    \frac{1}{B}
    + \frac{\lambda^d}{n}.
  \end{align*}

  \proofparagraph{controlling the expectation of $\tilde\Sigma_\rd(x)$}

  \begin{align*}
    \E \left[ \tilde\Sigma_\rd(x) \right]
    &=
    \frac{n}{\lambda^d}
    \sum_{i=1}^{n}
    \sum_{r=0}^{J}
    \sum_{r'=0}^{J}
    \omega_r
    \omega_{r'}
    \frac{1}{B^2}
    \sum_{b=1}^B
    \sum_{b'=1}^B
    \E \left[
      \frac{\I_{i b r}(x) \I_{i b' r'}(x) \sigma^2(X_i)}
      {N_{b r}(x) N_{b' r'}(x)}
    \right].
  \end{align*}
  Firstly, by Lemma~\ref{lem:simple_moment_denominator},
  the diagonal terms in the forest are
  \begin{align*}
    \left|
    \frac{n}{\lambda^d}
    \sum_{i=1}^{n}
    \sum_{r=0}^{J}
    \sum_{r'=0}^{J}
    \omega_r
    \omega_{r'}
    \frac{1}{B^2}
    \sum_{b=1}^B
    \E \left[
      \frac{\I_{i b r}(x) \I_{i b r'}(x) \sigma^2(X_i)}
      {N_{b r}(x) N_{b r'}(x)}
    \right]
    \right|
    \lesssim
    \frac{n}{\lambda^d}
    \frac{1}{B}
    \E \left[
      \frac{\I_{b r}(x)} {N_{b r}(x)}
    \right]
    \lesssim
    \frac{1}{B},
  \end{align*}
  so it suffices to take $b \neq b'$ since
  \begin{align*}
    \E \left[ \tilde\Sigma_\rd(x) \right]
    &=
    \frac{n^2}{\lambda^d}
    \sum_{r=0}^{J}
    \sum_{r'=0}^{J}
    \omega_r
    \omega_{r'} \,
    \E \left[
      \frac{\I_{i b r}(x) \I_{i b' r'}(x) \sigma^2(X_i)}
      {N_{b r}(x) N_{b' r'}(x)}
    \right]
    + O \left( \frac{1}{B} \right).
  \end{align*}
  Next, note that
  \begin{align*}
    \E \left[
      \frac{\I_{i b r}(x) \I_{i b' r'}(x) \sigma^2(X_i)}
      {N_{b r}(x) N_{b' r'}(x)}
    \right]
    &=
    \sigma^2(x)
    \E \left[
      \frac{\I_{i b r}(x) \I_{i b' r'}(x)}
      {N_{b r}(x) N_{b' r'}(x)}
    \right]
    + \E \left[
      \frac{\I_{i b r}(x) \I_{i b' r'}(x)
      \big(\sigma^2(X_i) - \sigma^2(x) \big)}
      {N_{b r}(x) N_{b' r'}(x)}
    \right].
  \end{align*}
  Since $\sigma^2 \in \cH^{\betasigma}$,
  we have by Lemma~\ref{lem:largest_cell} and
  Lemma~\ref{lem:simple_moment_denominator} that
  \begin{align*}
    \frac{n^2}{\lambda^d}
    \E \left[
      \frac{\I_{i b r}(x) \I_{i b' r'}(x)
      \big|\sigma^2(X_i) - \sigma^2(x) \big|}
      {N_{b r}(x) N_{b' r'}(x)}
    \right]
    &\lesssim
    \frac{n^2}{\lambda^d}
    \frac{1}{n}
    \E \left[
      \frac{\I_{b' r'}(x) \max_j |T_{b r}(x)_j|} {N_{b' r'}(x)}
    \right]
    \lesssim
    \frac{1}{\lambda^{1 \wedge \betasigma}}.
  \end{align*}
  Therefore
  \begin{align*}
    \frac{n^2}{\lambda^d}
    \E \left[
      \frac{\I_{i b r}(x) \I_{i b' r'}(x) \sigma^2(X_i)}
      {N_{b r}(x) N_{b' r'}(x)}
    \right]
    &=
    \sigma^2(x)
    \frac{n^2}{\lambda^d}
    \E \left[
      \frac{\I_{i b r}(x) \I_{i b' r'}(x)}
      {N_{b r}(x) N_{b' r'}(x)}
    \right]
    + O \left( \frac{1}{\lambda^{1 \wedge \betasigma}} \right).
  \end{align*}
  Next, by conditioning on
  $T_{b r}$, $T_{b' r'}$, $N_{-i b r}(x)$, and $N_{-i b' r'}(x)$,
  \begin{align*}
    \E \left[
      \frac{\I_{i b r}(x) \I_{i b' r'}(x)}
      {N_{b r}(x) N_{b' r'}(x)}
    \right]
    &= \E \left[
      \frac{\int_{T_{b r}(x) \cap T_{b' r'}(x)} f(\xi) \diffi \xi}
      {(N_{-i b r}(x)+1) (N_{-i b' r'}(x)+1)}
    \right] \\
    &=
    \E \hspace*{-0.5mm} \left[
      \frac{f(x) |T_{b r}(x) \cap T_{b' r'}(x)|}
      {(N_{-i b r}(x)+1) (N_{-i b' r'}(x)+1)}
    \right]
    \hspace*{-0.8mm}
    +
    \hspace*{-0.5mm}
    \E \hspace*{-0.5mm} \left[
      \frac{\int_{T_{b r}(x) \cap T_{b' r'}(x)}
      (f(\xi) - f(x)) \diffi \xi}
      {(N_{-i b r}(x)+1) (N_{-i b' r'}(x)+1)}
    \right] \\
    &=
    f(x) \,
    \E \left[
      \frac{|T_{b r}(x) \cap T_{b' r'}(x)|}
      {(N_{-i b r}(x)+1) (N_{-i b' r'}(x)+1)}
    \right]
    + O \left(
      \frac{\lambda^d}{n^2}
      \frac{1}{\lambda^{1 \wedge \betaf}}
    \right)
  \end{align*}
  by an argument based on Lemma~\ref{lem:largest_cell},
  the H\"older property of $f(x)$,
  and the proof of Lemma~\ref{lem:simple_moment_denominator}.
  Hence
  \begin{align*}
    \frac{n^2}{\lambda^d}
    \E \left[
      \frac{\I_{i b r}(x) \I_{i b' r'}(x) \sigma^2(X_i)}
      {N_{b r}(x) N_{b' r'}(x)}
    \right]
    &=
    \sigma^2(x)
    f(x)
    \frac{n^2}{\lambda^d}
    \E \left[
      \frac{|T_{b r}(x) \cap T_{b' r'}(x)|}
      {(N_{-i b r}(x)+1) (N_{-i b' r'}(x)+1)}
    \right] \\
    &\quad+
    O \left(
      \frac{1}{\lambda^{1 \wedge \betaf \wedge \betasigma}}
    \right).
  \end{align*}
  Apply Lemma~\ref{lem:binomial_expectation}
  to approximate the expectation with
  $N_{-i b' r' \setminus b r}(x) =
  \sum_{j \neq i} \I\{X_j \in T_{b' r'}(x) \setminus T_{b r}(x)\}$:
  \begin{align*}
    &\E \left[
      \frac{|T_{b r}(x) \cap T_{b' r'}(x)|}
      {(N_{-i b r}(x)+1) (N_{-i b' r'}(x)+1)}
    \right] \\
    &\quad=
    \E \left[
      \frac{|T_{b r}(x) \cap T_{b' r'}(x)|}
      {N_{-i b r}(x)+1}
      \,\E \left[
        \frac{1}
        {N_{-i b' r' \cap b r}(x)+N_{-i b' r' \setminus b r}(x)+1}
        \Bigm| \bT, N_{-i b' r' \cap b r}(x), N_{-i b r \setminus b' r'}(x)
      \right]
    \right].
  \end{align*}
  Now conditional on
  $\bT$, $N_{-i b' r' \cap b r}(x)$, and $N_{-i b r \setminus b' r'}(x)$,
  \begin{align*}
    N_{-i b' r' \setminus b r}(x)
    &\sim \Bin\left(
      n - 1 - N_{-i b r}(x), \
      \frac{\int_{T_{b' r'}(x) \setminus T_{b r}(x)} f(\xi) \diffi \xi}
      {1 - \int_{T_{b r}(x)}
      f(\xi) \diffi \xi}
    \right).
  \end{align*}
  We bound these parameters above and below.
  Firstly, by applying Lemma~\ref{lem:active_data} with $B=1$,
  we have
  \begin{align*}
    \P \left( N_{-i b r}(x) >
      t^{d+1}
      \frac{n}{\lambda^d}
    \right)
    &\leq
    4 d e^{- t / (4 \|f\|_\infty(1 + 1/a_r))}
    \leq
    e^{- t / C}
  \end{align*}
  for some $C > 0$ and sufficiently large $t$.
  Next, note if $f$ is $\betaf$-H\"older with constant $L$,
  by Lemma~\ref{lem:largest_cell},
  \begin{align*}
    &\P \left(
      \left|
      \frac{\int_{T_{b' r'}(x) \setminus T_{b r}(x)} f(\xi) \diffi \xi}
      {1 - \int_{T_{b r}(x)} f(\xi)
      \diffi \xi}
      - f(x) |T_{b' r'}(x) \setminus T_{b r}(x)|
      \right|
      > 2L \,
      |T_{b' r'}(x) \setminus T_{b r}(x)|
      \sum_{j=1}^d |T_{b' r'}(x)_j|^{1 \wedge \betaf}
    \right) \\
    &\quad\leq
    \P \left(
      \frac{\int_{T_{b' r'}(x) \setminus T_{b r}(x)}
      |f(\xi) - f(x)| \diffi \xi}
      {1 - \int_{T_{b r}(x)} f(\xi)
      \diffi \xi}
      > 2L \,
      |T_{b' r'}(x) \setminus T_{b r}(x)|
      \sum_{j=1}^d |T_{b' r'}(x)_j|^{1 \wedge \betaf}
    \right) \\
    &\quad\leq
    \P \left(
      \frac{1}{1 - \int_{T_{b r}(x)} f(\xi)
      \diffi \xi}
      > 2
    \right)
    \leq
    \P \left(
      \int_{T_{b r}(x)} f(\xi) \diffi \xi
      > \frac{1}{2}
    \right) \\
    &\quad\leq
    \P \left(
      |T_{b r}(x)|
      > \frac{1}{2\|f\|_\infty}
    \right)
    \lesssim
    \P \left(
      \max_{1 \leq j \leq d}
      |T_{b r}(x)_j|^{1 \wedge \betaf}
      > \frac{1}{2 \|f\|_\infty}
    \right)
    \lesssim
    2 d e^{- \lambda / (4 \|f\|_\infty)}
    \lesssim e^{-\lambda/C},
  \end{align*}
  increasing $C$ as necessary.
  Thus with probability at least $1 - e^{-t/C} - e^{-\lambda/C}$,
  \begin{align*}
    N_{-i b' r' \setminus b r}(x)
    &\leq \Bin\Bigg(
      n, \,
      |T_{b' r'}(x) \setminus T_{b r}(x)|
      \Bigg( f(x) + 2 L \sum_{j=1}^{d} |T_{b' r'}(x)_j|^{1 \wedge \betaf} \Bigg)
    \Bigg) \\
    N_{-i b' r' \setminus b r}(x)
    &\geq
    \Bin\Bigg(
      n
      \left( 1 - \frac{t^{d+1}}{\lambda^d}
      - \frac{1}{n} \right), \,
      |T_{b' r'}(x) \setminus T_{b r}(x)|
      \Bigg( f(x) - 2 L \sum_{j=1}^{d} |T_{b' r'}(x)_j|^{1 \wedge \betaf} \Bigg)
    \Bigg).
  \end{align*}
  So by Lemma~\ref{lem:binomial_expectation} conditionally on
  $\bT$, $N_{-i b' r' \cap b r}(x)$, and $N_{-i b r \setminus b' r'}(x)$,
  taking $t = 4 C \log n$ and recalling $\lambda \gtrsim (\log n)^3$,
  with probability at least $1 - n^{-3}$,
  \begin{align*}
    &\left|
    \E \left[
      \frac{1}
      {N_{-i b' r' \cap b r}(x)+N_{-i b' r' \setminus b r}(x)+1}
      \Bigm| \bT, N_{-i b' r' \cap b r}(x), N_{-i b r \setminus b' r'}(x)
    \right]
    \right.
    \\
    &\left.
    \qquad-
    \frac{1}
    {N_{-i b' r' \cap b r}(x) + n f(x) |T_{b' r'}(x) \setminus T_{b r}(x)|+1}
    \right|
    \lesssim
    \frac{1 + n |T_{b' r'}(x)_j|^{1 \wedge \betaf}
    |T_{b' r'}(x) \setminus T_{b r}(x)|}
    {\left(N_{-i b' r' \cap b r}(x)
    + n |T_{b' r'}(x) \setminus T_{b r}(x)|+1\right)^2}.
  \end{align*}
  Therefore by the same approach as the proof of
  Lemma~\ref{lem:moment_denominator},
  \begin{align*}
    &
    \left|
    \E \left[
      \frac{|T_{b r}(x) \cap T_{b' r'}(x)|}
      {(N_{-i b r}(x)+1) (N_{-i b' r'}(x)+1)}
      - \frac{|T_{b r}(x) \cap T_{b' r'}(x)|}
      {(N_{-i b r}(x) \! + \! 1)
        (N_{-i b' r' \cap b r}(x)+n f(x)
      |T_{b' r'}(x) \setminus T_{b r}(x)|+1)}
    \right]
    \right| \\
    &\quad\lesssim
    \E \left[
      \frac{|T_{b r}(x) \cap T_{b' r'}(x)|}{N_{-i b r}(x)+1}
      \frac{1 + n |T_{b' r'}(x)_j|^{1 \wedge \betaf}
      |T_{b' r'}(x) \setminus T_{b r}(x)|}
      {\left(N_{-i b' r' \cap b r}(x)
      + n |T_{b' r'}(x) \setminus T_{b r}(x)|+1\right)^2}
    \right]
    + n^{-3} \\
    &\quad\lesssim
    \E \left[
      \frac{|T_{b r}(x) \cap T_{b' r'}(x)|}
      {n |T_{b r}(x)|+1}
      \frac{1 + n |T_{b' r'}(x)_j|^{1 \wedge \betaf}
      |T_{b' r'}(x) \setminus T_{b r}(x)|}
      {(n |T_{b' r'}(x)| + 1)^2}
    \right]
    + n^{-3} \\
    &\quad\lesssim
    \E \left[
      \frac{1}{n}
      \frac{1}
      {(n |T_{b r}(x)| + 1) (n |T_{b' r'}(x)| + 1)}
      + \frac{1}{n}
      \frac{|T_{b' r'}(x)_j|^{1 \wedge \betaf}}
      {n |T_{b' r'}(x)| + 1}
    \right]
    + n^{-3} \\
    &\quad\lesssim
    \frac{\lambda^{2d}}{n^3}
    + \frac{1}{n \lambda^{1 \wedge \betaf}}
    \frac{\lambda^d}{n}
    \lesssim
    \frac{\lambda^d}{n^2}
    \left(
      \frac{\lambda^{d}}{n}
      + \frac{1}{\lambda^{1 \wedge \betaf}}
    \right).
  \end{align*}
  Now apply the same argument to the other
  term in the expectation, to see that
  \begin{align*}
    &\left|
    \E \left[
      \frac{1}
      {N_{-i b r \cap b' r'}(x)+N_{-i b r \setminus b' r'}(x)+1}
      \Bigm| \bT, N_{-i b r \cap b' r'}(x), N_{-i b' r' \setminus b r}(x)
    \right]
    \right.
    \\
    &\left.
    \qquad-
    \frac{1}
    {N_{-i b r \cap b' r'}(x) + n f(x) |T_{b r}(x) \setminus T_{b' r'}(x)|+1}
    \right|
    \lesssim
    \frac{1 + n |T_{b r}(x)_j|^{1 \wedge \betaf}
    |T_{b r}(x) \setminus T_{b' r'}(x)|}
    {\left(N_{-i b r \cap b' r'}(x)
    + n |T_{b r}(x) \setminus T_{b' r'}(x)|+1\right)^2}.
  \end{align*}
  with probability at least $1 - n^{-3} - e^{-\lambda/C}$,
  and so likewise again with $t = 4 C \log n$,
  \begin{align*}
    &\frac{n^2}{\lambda^d}
    \left|
    \E \left[
      \frac{|T_{b r}(x) \cap T_{b' r'}(x)|}{N_{-i b r}(x)+1}
      \frac{1}
      {N_{-i b' r' \cap b r}(x)+n f(x) |T_{b' r'}(x) \setminus T_{b r}(x)|+1}
    \right]
    \right.
    \\
    &\left.
    \quad-
    \E \left[
      \frac{|T_{b r}(x) \cap T_{b' r'}(x)|}
      {N_{-i b r \cap b' r'}(x) + n f(x) |T_{b r}(x) \setminus T_{b' r'}(x)|+1}
      \frac{1}
      {N_{-i b' r' \cap b r}(x)+n f(x) |T_{b' r'}(x) \setminus T_{b r}(x)|+1}
    \right]
    \right| \\
    &\lesssim
    \frac{n^2}{\lambda^d} \,
    \E \left[
      \frac{1 + n |T_{b r}(x)_j|^{1 \wedge \betaf}
      |T_{b r}(x) \setminus T_{b' r'}(x)|}
      {\left(N_{-i b r \cap b' r'}(x)
      + n |T_{b r}(x) \setminus T_{b' r'}(x)|+1\right)^2}
      \frac{|T_{b r}(x) \cap T_{b' r'}(x)|}
      {N_{-i b' r' \cap b r}(x)+n f(x) |T_{b' r'}(x) \setminus T_{b r}(x)|+1}
    \right] \\
    &\quad+
    \frac{n^2}{\lambda^d}
    n^{-3}
    \lesssim
    \frac{\lambda^d}{n}
    + \frac{1}{\lambda^{1 \wedge \betaf}}.
  \end{align*}
  Thus far we have proven that
  \begin{align*}
    &\frac{n^2}{\lambda^d}
    \E \left[
      \frac{\I_{i b r}(x) \I_{i b' r'}(x) \sigma^2(X_i)}
      {N_{b r}(x) N_{b' r'}(x)}
    \right]
    = \sigma^2(x)
    f(x)
    \frac{n^2}{\lambda^d} \\
    &\quad\times
    \E \left[
      \frac{|T_{b r}(x) \cap T_{b' r'}(x)|}
      {N_{-i b r \cap b' r'}(x) + n f(x) |T_{b r}(x) \setminus T_{b' r'}(x)|+1}
      \frac{1}
      {N_{-i b' r' \cap b r}(x)+n f(x) |T_{b' r'}(x) \setminus T_{b r}(x)|+1}
    \right] \\
    &\qquad+
    O \left(
      \frac{1}{\lambda^{1 \wedge \betaf \wedge \betasigma}}
      + \frac{\lambda^d}{n}
    \right).
  \end{align*}
  Next we remove the $N_{-i b r \cap b' r'}(x)$ terms.
  As before, with probability at least $1 - e^{-t/C} - e^{-\lambda/C}$,
  conditional on $\bT$,
  \begin{align*}
    N_{-i b r \cap b' r'}(x)
    &\leq \Bin\Bigg(
      n, \,
      |T_{b r}(x) \cap T_{b' r'}(x)|
      \Bigg( f(x) + 2L \sum_{j=1}^d |T_{b r}(x)_j|^{1 \wedge \betaf} \Bigg)
    \Bigg), \\
    N_{-i b r \cap b' r'}(x)
    &\geq
    \Bin\Bigg(
      n
      \left( 1 - \frac{t^{d+1}}{\lambda^d}
      - \frac{1}{n} \right), \,
      |T_{b r}(x) \cap T_{b' r'}(x)|
      \Bigg( f(x) - 2L \sum_{j=1}^d |T_{b r}(x)_j|^{1 \wedge \betaf} \Bigg)
    \Bigg).
  \end{align*}
  So by Lemma~\ref{lem:binomial_expectation}
  conditionally on $\bT$, with $t = 4 C \log n$ and
  with probability at least $1 - n^{-3}$,
  \begin{align*}
    &
    \left|
    \E \left[
      \frac{1}
      {N_{-i b r \cap b' r'}(x) + n f(x) |T_{b r}(x) \setminus T_{b' r'}(x)|+1}
      \frac{1}
      {N_{-i b' r' \cap b r}(x)+n f(x) |T_{b' r'}(x) \setminus T_{b r}(x)|+1}
      \Bigm| \bT
    \right]
    \right.
    \\
    &\left.
    \qquad-
    \frac{1}
    {n f(x) |T_{b r}(x)|+1}
    \frac{1}
    {n f(x) |T_{b' r'}(x)|+1}
    \right| \\
    &\quad\lesssim
    \frac{1 + n |T_{b r}(x)_j|^{1 \wedge \betaf}
    |T_{b r}(x) \cap T_{b' r'}(x)|}
    {(n |T_{b r}(x)| + 1)(n |T_{b' r'}(x)| + 1)}
    \left(
      \frac{1}{n |T_{b r}(x)| + 1}
      + \frac{1}{n |T_{b' r'}(x)| + 1}
    \right).
  \end{align*}
  Now by Lemma~\ref{lem:moment_cell},
  \begin{align*}
    &\frac{n^2}{\lambda^d}
    \left|
    \E \left[
      \frac{|T_{b r}(x) \cap T_{b' r'}(x)|}
      {N_{-i b r \cap b' r'}(x) + n f(x) |T_{b r}(x) \setminus T_{b' r'}(x)|+1}
      \frac{1}
      {N_{-i b' r' \cap b r}(x)+n f(x) |T_{b' r'}(x) \setminus T_{b r}(x)|+1}
    \right]
    \right.
    \\
    &\left.
    \qquad-
    \E \left[
      \frac{|T_{b r}(x) \cap T_{b' r'}(x)|}
      {n f(x) |T_{b r}(x)|+1}
      \frac{1}
      {n f(x) |T_{b' r'}(x)|+1}
    \right]
    \right| \\
    &\quad\lesssim
    \frac{n^2}{\lambda^d}
    \E \left[
      \frac{1 + n |T_{b r}(x)_j|^{1 \wedge \betaf}
      |T_{b r}(x) \cap T_{b' r'}(x)|}
      {(n |T_{b r}(x)| + 1)(n |T_{b' r'}(x)| + 1)}
      \left(
        \frac{|T_{b r}(x) \cap T_{b' r'}(x)|}{n |T_{b r}(x)| + 1}
        + \frac{|T_{b r}(x) \cap T_{b' r'}(x)|}{n |T_{b' r'}(x)| + 1}
      \right)
    \right]
    + \frac{1}{n \lambda^d} \\
    &\quad\lesssim
    \frac{n^2}{\lambda^d}
    \frac{1}{n^3}
    \E \left[
      \frac{1 + n |T_{b r}(x)_j|^{1 \wedge \betaf}
      |T_{b r}(x) \cap T_{b' r'}(x)|}
      {|T_{b r}(x)| |T_{b' r'}(x)|}
    \right]
    + \frac{1}{n \lambda^d} \\
    &\quad\lesssim
    \frac{1}{n \lambda^d}
    \E \left[
      \frac{1}{|T_{b r}(x)| |T_{b' r'}(x)|}
    \right]
    + \frac{1}{\lambda^d}
    \E \left[
      \frac{|T_{b r}(x)_j|^{1 \wedge \betaf}}{|T_{b' r'}(x)|}
    \right]
    + \frac{1}{n \lambda^d}
    \lesssim
    \frac{\lambda^d}{n}
    + \frac{1}{\lambda^{1 \wedge \betaf}}.
  \end{align*}
  This allows us to deduce that
  \begin{align*}
    \frac{n^2}{\lambda^d}
    \E \left[
      \frac{\I_{i b r}(x) \I_{i b' r'}(x) \sigma^2(X_i)}
      {N_{b r}(x) N_{b' r'}(x)}
    \right]
    &=
    \sigma^2(x)
    f(x)
    \frac{n^2}{\lambda^d}
    \E \left[
      \frac{|T_{b r}(x) \cap T_{b' r'}(x)|}
      {(n f(x) |T_{b r}(x)|+1)(n f(x) |T_{b' r'}(x)|+1)}
    \right] \\
    &\quad+
    O \left(
      \frac{1}{\lambda^{1 \wedge \betaf \wedge \betasigma}}
      + \frac{\lambda^d}{n}
    \right),
  \end{align*}
  and so
  \begin{align*}
    \E \left[ \tilde\Sigma_\rd(x) \right]
    &=
    \sigma^2(x)
    f(x)
    \frac{n^2}{\lambda^d}
    \sum_{r=0}^{J}
    \sum_{r'=0}^{J}
    \omega_r
    \omega_{r'}
    \E \left[
      \frac{|T_{b r}(x) \cap T_{b' r'}(x)|}
      {(n f(x) |T_{b r}(x)|+1)(n f(x) |T_{b' r'}(x)|+1)}
    \right] \\
    &\quad+
    O \left(
      \frac{1}{\lambda^{1 \wedge \betaf \wedge \betasigma}}
      + \frac{\lambda^d}{n}
      + \frac{1}{B}
    \right).
  \end{align*}

  \proofparagraph{calculating the limiting variance $\Sigma_\rd(x)$}

  Now that we have reduced the variance to an expression
  only involving the sizes of Mondrian cells,
  we can exploit their exact distribution to compute this expectation.
  Recall from \citet[Proposition~1]{mourtada2020minimax}
  that we can write
  \begin{align*}
    &|T_{b r}(x)|
    = \prod_{j=1}^{d}
    \left(
      \frac{E_{1j}}{a_r \lambda} \wedge x_j
      + \frac{E_{2j}}{a_r \lambda} \wedge (1 - x_j)
    \right),
    \qquad
    |T_{b' r'}(x)|
    =
    \prod_{j=1}^{d}
    \left(
      \frac{E_{3j}}{a_{r'} \lambda} \wedge x_j
      +  \frac{E_{4j}}{a_{r'} \lambda} \wedge (1 - x_j)
    \right), \\
    &|T_{b r }(x)\cap T_{b' r'}(x)|
    = \prod_{j=1}^{d}
    \left(
      \frac{E_{1j}}{a_r \lambda} \wedge
      \frac{E_{3j}}{a_{r'} \lambda}
      \wedge x_j
      +  \frac{E_{2j}}{a_r \lambda} \wedge
      \frac{E_{4j}}{a_{r'} \lambda}
      \wedge (1 - x_j)
    \right)
  \end{align*}
  where $E_{1j}$, $E_{2j}$, $E_{3j}$, and $E_{4j}$
  are independent and $\Exp(1)$.
  Define their non-truncated versions as
  \begin{align*}
    |\tilde T_{b r}(x)|
    &=
    a_r^{-d}
    \lambda^{-d}
    \prod_{j=1}^{d}
    \left( E_{1j} + E_{2j} \right),
    &|\tilde T_{b' r'}(x)|
    &=
    a_{r'}^{-d}
    \lambda^{-d}
    \prod_{j=1}^{d}
    \left( E_{3j} + E_{4j} \right), \\
    |\tilde T_{b r}(x) \cap \tilde T_{b' r'}(x)|
    &=
    \lambda^{-d}
    \prod_{j=1}^{d}
    \left(
      \frac{E_{1j}}{a_r}
      \wedge
      \frac{E_{3j}}{a_{r'}}
      + \frac{E_{2j}}{a_r}
      \wedge
      \frac{E_{4j}}{a_{r'}}
    \right),
  \end{align*}
  and note that
  \begin{align*}
    &\P \left(
      \big( \tilde T_{b r}(x), \tilde T_{b' r'}(x),
      \tilde T_{b r}(x) \cap T_{b' r'}(x) \big)
      \neq
      \big( T_{b r}(x), T_{b' r'}(x), T_{b r}(x) \cap T_{b' r'}(x) \big)
    \right) \\
    &\quad\leq
    \sum_{j=1}^{d}
    \big(
      \P(E_{1j} \geq a_r \lambda x_j)
      + \P(E_{3j} \geq a_{r'} \lambda x_j)
      + \P(E_{2j} \geq a_r \lambda (1 - x_j))
      + \P(E_{4j} \geq a_{r'} \lambda (1 - x_j))
    \big) \\
    &\quad\leq
    e^{-C \lambda}
  \end{align*}
  for some $C > 0$ and sufficiently large $\lambda$.
  Hence by the Cauchy--Schwarz inequality
  and Lemma~\ref{lem:moment_cell},
  \begin{align*}
    &
    \frac{n^2}{\lambda^d}
    \left|
    \E \left[
      \frac{|T_{b r}(x) \cap T_{b' r'}(x)|}
      {n f(x) |T_{b r}(x)|+1}
      \frac{1}
      {n f(x) |T_{b' r'}(x)|+1}
    \right]
    - \E \left[
      \frac{|\tilde T_{b r}(x) \cap T_{b' r'}(x)|}
      {n f(x) |\tilde T_{b r}(x)|+1}
      \frac{1}
      {n f(x) |\tilde T_{b' r'}(x)|+1}
    \right]
    \right| \\
    &\quad\lesssim
    \frac{n^2}{\lambda^d}
    e^{-C \lambda}
    \lesssim
    \frac{1}{n \lambda^d}
  \end{align*}
  as $\lambda \gtrsim (\log n)^3$. Therefore
  \begin{align*}
    \frac{n^2}{\lambda^d}
    \E \left[
      \frac{\I_{i b r}(x) \I_{i b' r'}(x) \sigma^2(X_i)}
      {N_{b r}(x) N_{b' r'}(x)}
    \right]
    &=
    \sigma^2(x)
    f(x)
    \frac{n^2}{\lambda^d}
    \E \left[
      \frac{|\tilde T_{b r}(x) \cap \tilde T_{b' r'}(x)|}
      {(n f(x) |\tilde T_{b r}(x)|+1)(n f(x) |\tilde T_{b' r'}(x)|+1)}
    \right] \\
    &\quad+
    O \left(
      \frac{1}{\lambda^{1 \wedge \betaf \wedge \betasigma}}
      + \frac{\lambda^d}{n}
    \right).
  \end{align*}
  Now we remove the superfluous units in the denominators.
  Firstly, by independence of the trees,
  \begin{align*}
    &
    \frac{n^2}{\lambda^d}
    \left|
    \E \left[
      \frac{|\tilde T_{b r}(x) \cap \tilde T_{b' r'}(x)|}
      {(n f(x) |\tilde T_{b r}(x)|+1)(n f(x) |\tilde T_{b' r'}(x)|+1)}
    \right]
    - \E \left[
      \frac{|\tilde T_{b r}(x) \cap \tilde T_{b' r'}(x)|}
      {(n f(x) |\tilde T_{b r}(x)|+1)(n f(x) |\tilde T_{b' r'}(x)|)}
    \right]
    \right| \\
    &\quad\lesssim
    \frac{n^2}{\lambda^d}
    \E \left[
      \frac{|\tilde T_{b r}(x) \cap \tilde T_{b' r'}(x)|}
      {n |\tilde T_{b r}(x)|}
      \frac{1}
      {n^2 |\tilde T_{b' r'}(x)|^2}
    \right]
    \lesssim
    \frac{1}{n \lambda^d}
    \E \left[
      \frac{1}{|T_{b r}(x)|}
    \right]
    \E \left[
      \frac{1}{|T_{b' r'}(x)|}
    \right]
    \lesssim
    \frac{\lambda^d}{n}.
  \end{align*}
  Secondly, we have in exactly the same manner that
  \begin{align*}
    \frac{n^2}{\lambda^d}
    \left|
    \E \left[
      \frac{|\tilde T_{b r}(x) \cap T_{b' r'}(x)|}
      {(n f(x) |\tilde T_{b r}(x)|+1)(n f(x) |\tilde T_{b' r'}(x)|)}
    \right]
    - \E \left[
      \frac{|\tilde T_{b r}(x) \cap T_{b' r'}(x)|}
      {n^2 f(x)^2 |\tilde T_{b r}(x)| |\tilde T_{b' r'}(x)|}
    \right]
    \right|
    &\lesssim
    \frac{\lambda^d}{n}.
  \end{align*}
  Therefore
  \begin{align*}
    \frac{n^2}{\lambda^d}
    \E \left[
      \frac{\I_{i b r}(x) \I_{i b' r'}(x) \sigma^2(X_i)}
      {N_{b r}(x) N_{b' r'}(x)}
    \right]
    &=
    \frac{\sigma^2(x)}{f(x)}
    \frac{1}{\lambda^d}
    \E \left[
      \frac{|\tilde T_{b r}(x) \cap \tilde T_{b' r'}(x)|}
      {|\tilde T_{b r}(x)| |\tilde T_{b' r'}(x)|}
    \right]
    + O \left(
      \frac{1}{\lambda^{1 \wedge \betaf \wedge \betasigma}}
      + \frac{\lambda^d}{n}
    \right).
  \end{align*}
  It remains to compute this integral.
  By independence over $1 \leq j \leq d$,
  \begin{align*}
    &\E \left[
      \frac{|\tilde T_{b r}(x) \cap \tilde T_{b' r'}(x)|}
      {|\tilde T_{b r}(x)| |\tilde T_{b' r'}(x)|}
    \right] \\
    &\quad=
    a_r^d a_{r'}^d \lambda^d
    \prod_{j=1}^d
    \E \left[
      \frac{ (E_{1j} / a_r) \wedge (E_{3j} / a_{r'})
      + (E_{2j} a_r) \wedge (E_{4j} / a_{r'}) }
      { \left( E_{1j} + E_{2j} \right) \left( E_{3j} + E_{4j} \right)}
    \right] \\
    &\quad=
    2^d a_r^d a_{r'}^d \lambda^d
    \prod_{j=1}^d
    \E \left[
      \frac{ (E_{1j} / a_r) \wedge (E_{3j} / a_{r'})}
      { \left( E_{1j} + E_{2j} \right) \left( E_{3j} + E_{4j} \right) }
    \right] \\
    &\quad=
    2^d a_r^d a_{r'}^d \lambda^d
    \prod_{j=1}^d
    \int_{0}^{\infty}
    \int_{0}^{\infty}
    \int_{0}^{\infty}
    \int_{0}^{\infty}
    \frac{ (t_1 / a_r) \wedge (t_3 / a_{r'}) }
    { \left( t_1 + t_2 \right) \left( t_3 + t_4 \right) }
    e^{-t_1 - t_2 - t_3 - t_4}
    \diffi t_1
    \diffi t_2
    \diffi t_3
    \diffi t_4 \\
    &\quad=
    2^d a_r^d a_{r'}^d \lambda^d
    \prod_{j=1}^d
    \int_{0}^{\infty}
    \int_{0}^{\infty}
    ((t_1 / a_r) \wedge (t_3 / a_{r'}))
    e^{-t_1 - t_3}
    \left(
      \int_{0}^{\infty}
      \frac{e^{-t_2}}{t_1 + t_2}
      \diffi t_2
    \right)
    \left(
      \int_{0}^{\infty}
      \frac{e^{-t_4}}{t_3 + t_4}
      \diffi t_4
    \right)
    \diffi t_1
    \diffi t_3 \\
    &\quad=
    2^d a_r^d a_{r'}^d \lambda^d
    \prod_{j=1}^d
    \int_{0}^{\infty}
    \int_{0}^{\infty}
    ((t / a_r) \wedge (s / a_{r'}))
    \Gamma(0, t)
    \Gamma(0, s)
    \diffi t
    \diffi s,
  \end{align*}
  where we used
  $\int_0^\infty \frac{e^{-t}}{a + t} \diffi t = e^a \Gamma(0, a)$
  with $\Gamma(0, a) = \int_a^\infty \frac{e^{-t}}{t} \diffi t$
  the upper incomplete gamma function. Now
  \begin{align*}
    &2
    \int_{0}^{\infty}
    \int_{0}^{\infty}
    ((t / a_r) \wedge (s / a_{r'}))
    \Gamma(0, t)
    \Gamma(0, s)
    \diffi t
    \diffi s \\
    &\quad=
    \int_0^\infty
    \Gamma(0, t)
    \left(
      \frac{1}{a_{r'}}
      \int_0^{a_{r'} t / a_r}
      2 s \Gamma(0, s)
      \diffi{s}
      +
      \frac{t}{a_r}
      \int_{a_{r'} t / a_r}^\infty
      2 \Gamma(0, s)
      \diffi{s}
    \right)
    \diffi{t} \\
    &\quad=
    \int_0^\infty
    \Gamma(0, t)
    \left(
      \frac{t}{a_r}
      e^{- \frac{a_{r'}}{a_r}t}
      - \frac{1}{a_{r'}} e^{- \frac{a_{r'}}{a_r}t}
      + \frac{1}{a_{r'}}
      - \frac{a_{r'}}{a_r^2} t^2
      \Gamma\left(0, \frac{a_{r'}}{a_r} t\right)
    \right)
    \diffi{t} \\
    &\quad=
    \frac{1}{a_r}
    \int_0^\infty
    \hspace*{-1mm}
    t e^{- \frac{a_{r'}}{a_r} t}
    \Gamma(0, t)
    \diffi{t}
    - \frac{1}{a_{r'}}
    \int_0^\infty
    \hspace*{-1mm}
    e^{- \frac{a_{r'}}{a_r} t}
    \Gamma(0, t)
    \diffi{t} \\
    &\qquad+
    \frac{1}{a_{r'}}
    \int_0^\infty
    \hspace*{-1mm}
    \Gamma(0, t)
    \diffi{t}
    -
    \frac{a_{r'}}{a_r^2}
    \int_0^\infty
    \hspace*{-1mm}
    t^2 \Gamma\left(0, \frac{a_{r'}}{a_r} t\right)
    \Gamma(0, t)
    \diffi{t},
  \end{align*}
  since
  $\int_0^a 2 t \Gamma(0, t) \diffi t = a^2 \Gamma(0, a) - a e^{-a} -e^{-a} + 1$
  and
  $\int_a^\infty \Gamma(0, t) \diffi t = e^{-a} - a \Gamma(0, a)$.
  Next, we use
  $ \int_{0}^{\infty} \Gamma(0, t) \diffi t = 1$,
  $\int_{0}^{\infty} e^{-at} \Gamma(0, t) \diffi t
  = \frac{\log(1+a)}{a}$,
  $\int_{0}^{\infty} t e^{-at} \Gamma(0, t) \diffi t
  = \frac{\log(1+a)}{a^2} - \frac{1}{a(a+1)}$
  and
  $\int_{0}^{\infty} t^2 \Gamma(0, t) \Gamma(0, at) \diffi t
  = - \frac{2a^2 + a + 2}{3a^2 (a+1)} + \frac{2(a^3 + 1) \log(a+1)}{3a^3}
  - \frac{2 \log a}{3}$
  to see
  \begin{align*}
    &2
    \int_{0}^{\infty}
    \int_{0}^{\infty}
    ((t / a_r) \wedge (s / a_{r'}))
    \Gamma(0, t)
    \Gamma(0, s)
    \diffi t
    \diffi s \\
    &\quad=
    \frac{a_r \log(1+a_{r'} / a_r)}{a_{r'}^2}
    - \frac{a_r / a_{r'}}{a_r + a_{r'}}
    - \frac{a_r \log(1 + a_{r'} / a_r)}{a_{r'}^2}
    + \frac{1}{a_{r'}} \\
    &\qquad+
    \frac{2 a_{r'}^2 + a_r a_{r'} + 2 a_r^2}
    {3 a_r a_{r'} (a_r + a_{r'})}
    - \frac{2(a_{r'}^3 + a_r^3) \log(a_{r'} / a_r+1)}{3 a_r^2 a_{r'}^2}
    + \frac{2 a_{r'} \log (a_{r'} / a_r)}{3 a_r^2} \\
    &\quad=
    \frac{2}{3 a_r}
    \left(
      1 - \frac{a_{r'}}{a_r}
      \log\left(\frac{a_{r}}{a_{r'}} + 1\right)
    \right)
    + \frac{2}{3 a_{r'}}
    \left(
      1 - \frac{a_r }{a_{r'}}
      \log\left(\frac{a_{r'}}{a_{r}} + 1\right)
    \right).
  \end{align*}
  Finally we conclude this part by giving the limiting variance.
  \begin{align*}
    &\E \left[ \tilde\Sigma_\rd(x) \right] \\
    &\quad=
    \frac{\sigma^2(x)}{f(x)}
    \sum_{r=0}^{J}
    \sum_{r'=0}^{J}
    \omega_r
    \omega_{r'}
    \left(
      \frac{2 a_{r'}}{3}
      \left(
        1 - \frac{a_{r'}}{a_r}
        \log\left(\frac{a_r}{a_{r'}} + 1\right)
      \right)
      + \frac{2 a_r}{3}
      \left(
        1 - \frac{a_r}{a_{r'}}
        \log\left(\frac{a_{r'}}{a_r} + 1\right)
      \right)
    \right)^d \\
    &\qquad+
    O \left(
      \frac{1}{\lambda^{1 \wedge \betaf \wedge \betasigma}}
      + \frac{\lambda^d}{n}
      + \frac{1}{B}
    \right) \\
    &\quad= \Sigma_\rd(x)
    + O \left(
      \frac{1}{\lambda^{1 \wedge \betaf \wedge \betasigma}}
      + \frac{\lambda^d}{n}
      + \frac{1}{B}
    \right).
  \end{align*}
  It follows from this and the previous part that
  \begin{align*}
    \E \left[ \big(
        \tilde\Sigma_\rd(x) - \Sigma_\rd(x)
    \big)^2 \right]
    &\lesssim
    \frac{\lambda^d}{n}
    + \frac{1}{B}
    + \frac{1}{\lambda^{2(1 \wedge \betaf \wedge \betasigma)}}.
  \end{align*}

  \proofparagraph{a lower bound for the second moment with a single tree}

  We finally show here that if $B=1$ (a forest with a single tree),
  then $\tilde\Sigma_\rd(x)$ has a divergent second moment.
  We take $J = 0$ for brevity (no debiasing),
  and recall that $\sigma^2(x)$ and $f(x)$ are bounded below.
  Further, since $\lambda T(x)_j \leq \Gamma(2, 1)$,
  we have by Jensen's inequality that
  \begin{align*}
    \E \left[
      \tilde \Sigma_\rd(x)^2
    \right]
    &=
    \E \left[
      \left(
        \frac{n}{\lambda^d}
        \sum_{i=1}^n
        \frac{\I \{X_i \in T(x)\} \sigma^2(X_i)}{N(x)^2}
      \right)^2
    \right]
    \gtrsim
    \frac{n^2}{\lambda^{2 d}}
    \E \left[
      \frac{1}{N(x)^4}
      \left(
        \sum_{i=1}^n
        \I\{X_i \in T(x)\}
      \right)^2
    \right] \\
    &=
    \frac{n^2}{\lambda^{2 d}}
    \sum_{i=1}^n
    \E \left[
      \frac{\I \{X_i \in T(x)\}}{N(x)^3}
    \right]
    =
    \frac{n^3}{\lambda^{2 d}}
    \E \left[
      \frac{\I \{X_i \in T(x)\}}{(N_{-i}(x) + 1)^3}
    \right]
    \gtrsim
    \frac{n^3}{\lambda^{2 d}}
    \E \left[
      \frac{|T(x)|}{\E \left[ (N_{-i}(x) + 1)^3 \mid T \right]}
    \right] \\
    &\gtrsim
    \frac{n^3}{\lambda^{2 d}}
    \E \left[
      \frac{|T(x)|}{n^3 |T(x)|^3 + 1}
    \right]
    \gtrsim
    \frac{1}{\lambda^{2 d}}
    \E \left[
      \frac{\I \{|T(x)| \geq 1/n\}}{|T(x)|^2}
    \right]
    \gtrsim
    \frac{1}{\lambda^{2 d}}
    \E \left[
      \frac{\I \{|T(x)_j| \geq 1/n\}}{|T(x)_j|^2}
    \right]^d \\
    &\gtrsim
    \left(
      \int_{1/n}^{1}
      \frac{s e^{-s}}{s^2}
      \diffi s
    \right)^d
    \gtrsim
    \left(
      \int_{1/n}^{1}
      \frac{1}{s}
      \diffi s
    \right)^d
    \gtrsim
    \left( \log n \right)^d.
  \end{align*}
\end{proof}

\begin{proof}[Theorem~\ref{thm:minimax}]

  The bias--variance decomposition with Lemma~\ref{lem:bias_debiased}
  and \eqref{eq:variance_bound_debiased},
  along with the proof of Theorem~\ref{thm:clt_debiased},
  setting $J = \lfloor \flbeta / 2 \rfloor$, gives
  \begin{align*}
    \E \left[
      \big(
        \hat \mu_\rd(x)
        - \mu(x)
      \big)^2
    \right]
    &=
    \E \left[
      \big(
        \hat \mu_\rd(x)
        - \E \left[ \hat \mu_\rd(x) \mid \bX, \bT \right]
      \big)^2
    \right]
    + \E \left[
      \big(
        \E \left[ \hat \mu_\rd(x) \mid \bX, \bT \right]
        - \mu(x)
      \big)^2
    \right] \\
    &\lesssim
    \frac{\lambda^d}{n}
    + \frac{1}{\lambda^{2\beta}}
    + \frac{1}{\lambda^{2(1 \wedge \beta)} B}.
  \end{align*}
  As $\lambda \asymp n^{\frac{1}{d + 2 \beta}}$
  and $B \gtrsim n^{\frac{2 \beta - 2 (1 \wedge \beta)}{d + 2 \beta}}$,
  we have
  \begin{align*}
    \E \left[
      \big(
        \hat \mu_\rd(x)
        - \mu(x)
      \big)^2
    \right]
    &\lesssim
    n^{-\frac{2\beta}{d + 2 \beta}}.
  \end{align*}
\end{proof}

\begin{proof}[Theorem~\ref{thm:clt_debiased}]

  Define
  $S_i(x) =
  \sqrt{n / \lambda^d} \sum_{r=0}^{J} \omega_r
  \frac{1}{B} \sum_{b=1}^B
  \frac{\I_{i b r}(x) \varepsilon_i} {N_{b r}(x)}$,
  which are independent and zero mean conditional on
  $(\bX, \bT)$, and satisfy
  \begin{align*}
    \sqrt{\frac{n}{\lambda^d}}
    \big(
      \hat\mu_\rd(x)
      - \E\left[
        \hat\mu_\rd(x) \mid \bX, \bT
      \right]
    \big)
    = \sum_{i=1}^n S_i(x).
  \end{align*}
  Therefore by \citet[Theorem~5.7]{Petrov_1995_Book} conditional on
  $(\bX, \bT)$, with $\zeta = \delta \wedge 1$,
  \begin{align*}
    \sup_{t \in \R}
    \left|
    \P \left(
      \tilde \Sigma_\rd(x)^{-1/2}
      \sum_{i=1}^{n} S_i
      \leq t
      \Bigm| \bX, \bT
    \right)
    - \Phi(t)
    \right|
    \lesssim
    1 \wedge \left(
      \tilde \Sigma_\rd(x)^{-1 - \zeta / 2}
      \sum_{i=1}^{n}
      \E \left[ |S_i|^{2 + \zeta} \mid \bX, \bT \right]
    \right).
  \end{align*}
  It immediately follows that
  \begin{align*}
    \sup_{t \in \R}
    \left|
    \P \left(
      \tilde \Sigma_\rd(x)^{-1/2}
      \sum_{i=1}^{n} S_i
      \leq t
    \right)
    - \Phi(t)
    \right|
    &=
    \sup_{t \in \R}
    \left|
    \E \left[
      \P \left(
        \tilde \Sigma_\rd(x)^{-1/2}
        \sum_{i=1}^{n} S_i
        \leq t
        \Bigm| \bX, \bT
      \right)
    \right]
    - \Phi(t)
    \right| \\
    &\leq
    \E \left[
      \sup_{t \in \R}
      \left|
      \P \left(
        \tilde \Sigma_\rd(x)^{-1/2}
        \sum_{i=1}^{n} S_i
        \leq t
        \Bigm| \bX, \bT
      \right)
      - \Phi(t)
      \right|
    \right] \\
    &\lesssim
    \E \left[
      1 \wedge
      \left(
        \tilde \Sigma_\rd(x)^{-1 - \zeta / 2}
        \sum_{i=1}^{n}
        \E \left[ |S_i|^{2 + \zeta} \mid \bX, \bT \right]
      \right)
    \right].
  \end{align*}
  To bound this quantity, we first partition by the
  event that $\tilde \Sigma_\rd(x)$ is bounded away from zero:
  \begin{align}
    \nonumber
    \E \left[
      1 \wedge
      \left(
        \tilde \Sigma_\rd(x)^{-1 - \zeta / 2}
        \sum_{i=1}^{n}
        \E \left[ |S_i|^{2 + \zeta} \mid \bX, \bT \right]
      \right)
    \right]
    &\lesssim
    \E \big[ \tilde \Sigma_\rd(x) \big]^{-1 - \zeta / 2}
    \sum_{i=1}^{n}
    \E \left[
      |S_i|^{2 + \zeta}
    \right] \\
    \label{eq:berry_esseen_proof}
    &\quad+
    \P \left(
      \Big| \tilde \Sigma_\rd(x) - \E\big[\tilde\Sigma_\rd(x)\big] \Big|
      > \frac{\E\big[\tilde\Sigma_\rd(x)\big]}{2}
    \right).
  \end{align}
  The first term in \eqref{eq:berry_esseen_proof} is bounded as follows.
  We already have from the proof of
  Lemma~\ref{lem:variance_debiased} that eventually
  $\E \big[ \tilde \Sigma_\rd(x) \big] \geq \Sigma_\rd(x) / 2 \gtrsim 1$.
  Since $\E[\varepsilon_i^{2 + \delta} \mid \bX]$
  is bounded, by Jensen's inequality,
  \begin{align*}
    \sum_{i=1}^{n}
    \E \left[
      |S_i|^{2 + \zeta}
    \right]
    &=
    \sum_{i=1}^{n}
    \E \left[
      \left|
      \sqrt{\frac{n}{\lambda^d}}
      \sum_{r=0}^{J}
      \omega_r
      \frac{1}{B} \sum_{b=1}^B
      \frac{\I_{i b r}(x) \varepsilon_i} {N_{b r}(x)}
      \right|^{2 + \zeta}
    \right] \\
    &\leq
    \left( \frac{n}{\lambda^d} \right)^{1 + \zeta/2}
    (J + 1)^{1 + \zeta}
    \sum_{r=0}^{J}
    |\omega_r|^3 \,
    \E \left[
      \sum_{i=1}^{n}
      \left(
        \frac{1}{B} \sum_{b=1}^B
        \frac{\I_{i b r}(x)} {N_{b r}(x)}
      \right)^{2 + \zeta}
    \right] \\
    &\lesssim
    \left( \frac{n}{\lambda^d} \right)^{1 + \zeta/2}
    \E \left[
      \sum_{i=1}^{n}
      \left(
        \frac{1}{B} \sum_{b=1}^B
        \frac{\I_{i b r}(x)} {N_{b r}(x)}
      \right)^{2 + \zeta}
    \right].
  \end{align*}
  We now proceed by cases. If $\zeta = 1$, note that
  by Lemma~\ref{lem:simple_moment_denominator} and
  with $B \gtrsim (\log n)^d$,
  \begin{align*}
    \E \left[
      \sum_{i=1}^{n}
      \left(
        \frac{1}{B} \sum_{b=1}^B
        \frac{\I_{i b r}(x)} {N_{b r}(x)}
      \right)^{3}
    \right]
    &=
    \E \left[
      \sum_{i=1}^{n}
      \frac{1}{B^3}
      \sum_{b=1}^B
      \sum_{b'=1}^B
      \sum_{b''=1}^B
      \frac{\I_{i b r}(x)} {N_{b r}(x)}
      \frac{\I_{i b' r'}(x)} {N_{b' r'}(x)}
      \frac{\I_{i b'' r''}(x)} {N_{b'' r''}(x)}
    \right] \\
    &\leq
    \frac{1}{B^2}
    \sum_{b=1}^B
    \sum_{b'=1}^B
    \E \left[
      \frac{\I_{b r}(x)} {N_{b r}(x)}
      \frac{\I_{b' r'}(x)} {N_{b' r'}(x)}
    \right]
    \lesssim
    \frac{\lambda^{2d}}{n^2}
    + \frac{1}{B}
    \frac{\lambda^{2d} (\log n)^d}{n^2}
    \lesssim
    \frac{\lambda^{2d}}{n^2}.
  \end{align*}
  Alternatively, if $\zeta \in (0, 1)$, by Jensen's inequality
  and Lemma~\ref{lem:simple_moment_denominator},
  \begin{align*}
    \E \left[
      \sum_{i=1}^{n}
      \left(
        \frac{1}{B} \sum_{b=1}^B
        \frac{\I_{i b r}(x)} {N_{b r}(x)}
      \right)^{2 + \zeta}
    \right]
    &\leq
    \E \left[
      \sum_{i=1}^{n}
      \left(
        \frac{\I_{i b r}(x)} {N_{b r}(x)}
      \right)^{2 + \zeta}
    \right]
    \leq
    \E \left[
      \sum_{i=1}^{n}
      \frac{\I_{i b r}(x)} {N_{b r}(x)}
      \left(
        \frac{\I_{b r}(x)} {N_{b r}(x)}
      \right)^{1 + \zeta}
    \right] \\
    &\leq
    \E \left[
      \left(
        \frac{\I_{b r}(x)} {N_{b r}(x)}
      \right)^{1 + \zeta}
    \right]
    \lesssim
    \left( \frac{\lambda^d}{n} \right)^{1 + \zeta}.
  \end{align*}
  Both cases lead to the conclusion that
  \begin{align*}
    \E \big[ \tilde \Sigma_\rd(x) \big]^{-1 - \zeta / 2}
    \sum_{i=1}^{n}
    \E \left[
      |S_i|^{2 + \zeta}
    \right]
    &\lesssim
    \left( \frac{n}{\lambda^d} \right)^{1 + \zeta/2}
    \left( \frac{\lambda^d}{n} \right)^{1 + \zeta}
    = \left( \frac{\lambda^d}{n} \right)^{\zeta / 2}.
  \end{align*}
  For the second term in \eqref{eq:berry_esseen_proof},
  Chebyshev's inequality along with the proof of
  Lemma~\ref{lem:variance_debiased} give
  \begin{align*}
    \P \left(
      \Big| \tilde \Sigma_\rd(x) - \E\big[\tilde\Sigma_\rd(x)\big] \Big|
      > \frac{\E\big[\tilde\Sigma_\rd(x)\big]}{2}
    \right)
    \lesssim
    \Var \left[ \tilde \Sigma_\rd(x) \right]
    \lesssim
    \frac{1}{B}
    + \frac{\lambda^d}{n}.
  \end{align*}
  Therefore
  \begin{align*}
    \sup_{t \in \R}
    \left|
    \P \left(
      \tilde \Sigma_\rd(x)^{-1/2}
      \sum_{i=1}^{n} S_i
      \leq t
    \right)
    - \Phi(t)
    \right|
    &\lesssim
    \left( \frac{\lambda^d}{n} \right)^{\frac{1 \wedge \delta}{2}}
    + \frac{1}{B}.
  \end{align*}
\end{proof}

\begin{proof}[Lemma~\ref{lem:variance_estimation_debiased}]

  We begin by showing that
  $\hat\sigma^2(x)$ is consistent for $\sigma^2(x)$.

  \proofparagraph{consistency of $\hat\sigma^2(x)$}

  Recall that
  \begin{align}
    \label{eq:sigma2_hat_proof}
    \hat\sigma^2(x)
    &=
    \frac{1}{B}
    \sum_{b=1}^{B}
    \sum_{i=1}^n
    \frac{Y_i^2 \, \I_{i b}(x)} {N_b(x)}
    - \hat \mu(x)^2.
  \end{align}
  The first term in \eqref{eq:sigma2_hat_proof}
  is simply a Mondrian forest estimator of
  $\E[Y_i^2 \mid X_i = x] = \sigma^2(x) + \mu(x)^2$,
  which is bounded and in $\cH^{\betamu \wedge \betasigma}$.
  Therefore, by Lemma~\ref{lem:bias}, its conditional bias is
  \begin{align*}
    & \E \left[
      \left(
        \frac{1}{B}
        \sum_{b=1}^{B}
        \sum_{i=1}^n
        \frac{\big(\sigma^2(X_i) + \mu(X_i)^2- \sigma^2(x) - \mu(x)^2\big)
        \, \I_{i b}(x)}
        {N_b(x)}
      \right)^2
    \right] \\
    &\quad\lesssim
    \frac{1}{\lambda^{2(2 \wedge \beta \wedge \betasigma)}}
    + \frac{1}
    {\lambda^{2(1 \wedge \beta \wedge \betasigma)} B}
    + \frac{1}{\lambda^{2(1 \wedge \beta \wedge \betasigma)}}
    \frac{\lambda^d}{n}.
  \end{align*}
  We handle the stochastic part with a truncation argument. Let
  $S_i = Y_i^2 - \sigma^2(X_i) - \mu(X_i)^2$ and
  $\tilde S_i = S_i \, \I\{|S_i| \leq M\}
  - \E \left[ S_i \, \I\{|S_i| \leq M\} \mid X_i \right]$
  where $M > 0$ is to be determined. We bound
  \begin{align}
    \label{eq:truncation}
    \frac{1}{B}
    \sum_{b=1}^{B}
    \sum_{i=1}^n
    \frac{\tilde S_i \, \I_{i b}(x)} {N_b(x)}
    + \frac{1}{B}
    \sum_{b=1}^{B}
    \sum_{i=1}^n
    \frac{(S_i - \tilde S_i) \, \I_{i b}(x)} {N_b(x)}.
  \end{align}
  The first term in \eqref{eq:truncation} is controlled with a variance bound,
  noting $\tilde S_i \leq 2M$ almost surely.
  \begin{align*}
    \Var \left[
      \frac{1}{B}
      \sum_{b=1}^{B}
      \sum_{i=1}^n
      \frac{\tilde S_i \, \I_{i b}(x)} {N_b(x)}
    \right]
    \leq
    \sum_{i=1}^n
    \E \left[
      \frac{ \tilde S_i^2 \, \I_{i b}(x)}
      {N_b(x)^2}
    \right]
    \leq
    4 M^2
    \sum_{i=1}^n
    \E \left[
      \frac{\I_{b}(x)} {N_b(x)}
    \right]
    \lesssim
    M^2 \frac{\lambda^d}{n}.
  \end{align*}
  For the second term in \eqref{eq:truncation}, note that
  $S_i - \tilde S_i = S_i \, \I\{|S_i| > M\}
  - \E \left[ S_i \, \I\{|S_i| > M\} \mid X_i \right]$
  because $\E[S_i \mid X_i] = 0$.
  Since $\E \left[ |Y_i|^{2 + \delta} \mid X_i \right]$ is bounded,
  so is $\E \left[ |S_i|^{1 + \delta/2} \mid X_i \right]$. Thus
  \begin{align*}
    \E \left[
      \left|
      \frac{1}{B}
      \sum_{b=1}^{B}
      \sum_{i=1}^n
      \frac{(S_i - \tilde S_i) \, \I_{i b}(x)} {N_b(x)}
      \right|
    \right]
    &\leq
    \sum_{i=1}^n
    \E \left[
      \frac{\I_{i b}(x)} {N_b(x)}
      \E \left[
        |S_i - \tilde S_i|
        \mid X_i
      \right]
    \right] \\
    &\leq
    2 \sum_{i=1}^n
    \E \left[
      \frac{\I_{i b}(x)} {N_b(x)}
      \E \big[ |S_i| \, \I\{|S_i| > M\} \mid X_i \big]
    \right] \\
    &\leq
    \frac{2}{M^{\delta / 2}}
    \sum_{i=1}^n
    \E \left[
      \frac{\I_{i b}(x)} {N_b(x)}
      \E \left[ |S_i|^{1+\delta/2} \mid X_i \right]
    \right]
    \lesssim
    \frac{1}{M^{\delta / 2}}.
  \end{align*}
  Consistency of the second term in \eqref{eq:sigma2_hat_proof}
  follows directly from
  Lemma~\ref{lem:bias} and Theorem~\ref{thm:clt_debiased}
  with the same bias and variance bounds.
  Therefore
  \begin{align*}
    \E \left[
      \big|\hat\sigma^2(x) - \sigma^2(x)\big|
    \right]
    &\lesssim
    \frac{1}{\lambda^{2 \wedge \beta \wedge \betasigma}}
    + \frac{1}{\lambda^{1 \wedge \beta \wedge \betasigma} \sqrt B}
    + \frac{1}{\lambda^{1 \wedge \beta \wedge \betasigma}}
    \sqrt{\frac{\lambda^d}{n}}
    + M \sqrt{\frac{\lambda^d}{n}}
    + \frac{1}{M^{\delta/2}} \\
    &\lesssim
    \frac{1}{\lambda^{2 \wedge \beta \wedge \betasigma}}
    + \frac{1}{\lambda^{1 \wedge \beta \wedge \betasigma} \sqrt B}
    + \left(
      \frac{\lambda^d}{n}
    \right)^{\frac{\delta}{4 + 2 \delta}},
  \end{align*}
  where we set
  $M = \left( \frac{\lambda^d}{n} \right)^{-\frac{1}{2+\delta}}$.
  Note that if $\delta \geq 2$ then the variance argument applies
  directly, without the need for truncation, yielding
  \begin{align*}
    \E \left[
      \big|\hat\sigma^2(x) - \sigma^2(x)\big|^2
    \right]
    &\lesssim
    \frac{1}{\lambda^{2(2 \wedge \beta \wedge \betasigma)}}
    + \frac{1}{\lambda^{2(1 \wedge \beta \wedge \betasigma)} B}
    + \frac{\lambda^d}{n}.
  \end{align*}
  \proofparagraph{consistency of the sum}
  Note that
  \begin{align*}
    \frac{n}{\lambda^d}
    \sum_{i=1}^n
    \left(
      \sum_{r=0}^J
      \omega_r
      \frac{1}{B}
      \sum_{b=1}^B
      \frac{\I\{X_i \in T_{b r}(x)\}}
      {\sum_{i=1}^n \I\{X_i \in T_{b r}(x)\}}
    \right)^2
    &=
    \frac{n}{\lambda^d}
    \frac{1}{B^2}
    \sum_{i=1}^n
    \sum_{r=0}^J
    \sum_{r'=0}^J
    \omega_r
    \omega_{r'}
    \sum_{b=1}^B
    \sum_{b'=1}^B
    \frac{\I_{i b r}(x) \I_{i b' r'}(x)}
    {N_{b r}(x) N_{b' r'}(x)}.
  \end{align*}
  This is exactly the same as $\tilde\Sigma_\rd(x)$,
  if we were to take $\sigma^2(x) = 1$.
  Thus by Lemma~\ref{lem:variance_debiased}, we obtain
  \begin{align*}
    &\E \left[
      \left(
        \frac{n}{\lambda^d}
        \frac{1}{B^2}
        \sum_{i=1}^n
        \sum_{r=0}^J
        \sum_{r'=0}^J
        \omega_r
        \omega_{r'}
        \sum_{b=1}^B
        \sum_{b'=1}^B
        \frac{\I_{i b r}(x) \I_{i b' r'}(x)}
        {N_{b r}(x) N_{b' r'}(x)}
        - \frac{\Sigma_\rd(x)}{\sigma^2(x)}
      \right)^2
    \right]
    \lesssim
    \frac{\lambda^d}{n}
    + \frac{1}{B}
    + \frac{1}{\lambda^{2(1 \wedge \betaf \wedge \betasigma)}}.
  \end{align*}

  \proofparagraph{conclusion}

  By the previous parts and the Cauchy--Schwarz inequality,
  \begin{align*}
    \E \left[
      \left|
      \hat\Sigma_\rd(x)
      - \Sigma_\rd(x)
      \right|^{1/2}
    \right]^2
    &\lesssim
    \left(
      \frac{\lambda^d}{n}
    \right)^{\frac{\delta}{4 + 2 \delta}}
    + \frac{1}{\sqrt B}
    + \frac{1}{\lambda^{1 \wedge \betamu \wedge \betaf \wedge \betasigma}}.
  \end{align*}
  If $\delta \geq 2$ then we obtain
  \begin{align*}
    \E \left[
      \left|
      \hat\Sigma_\rd(x)
      - \Sigma_\rd(x)
      \right|
    \right]
    &\lesssim
    \sqrt{\frac{\lambda^d}{n}}
    + \frac{1}{\sqrt B}
    + \frac{1}{\lambda^{1 \wedge \betamu \wedge \betaf \wedge \betasigma}}.
  \end{align*}
  Combining these yields
  \begin{align*}
    \E \left[
      \left|
      \hat\Sigma_\rd(x)
      - \Sigma_\rd(x)
      \right|^{\frac{2 - \I\{\delta < 2\}}{2}}
    \right]^{\frac{2}{2 - \I\{\delta < 2\}}}
    &\lesssim
    \left(\frac{\lambda^d}{n}\right)^{\frac{1}{2}
    - \frac{\I\{\delta < 2\}}{2 + \delta}}
    + \frac{1}{\sqrt B}
    + \frac{1}{\lambda^{1 \wedge \betamu \wedge \betaf \wedge \betasigma}}.
  \end{align*}
\end{proof}

\begin{proof}[Theorem~\ref{thm:confidence_debiased}]

  Let $\tau$ and $\hat \tau$ be real-valued random variables.
  Then for any $\varepsilon > 0$,
  \begin{align*}
    \sup_{t \in \R}
    \left| \P \big( \hat \tau \leq t \big) - \Phi(t) \right|
    &\leq
    \sup_{t \in \R}
    \left| \P \big( \tau \leq t \big) - \Phi(t) \right|
    + \varepsilon \sqrt{2/\pi}
    + \P \big( |\hat \tau - \tau| > \varepsilon \big).
  \end{align*}
  Defining $a / 0 = 0$ for all $a \in \R$
  to accommodate the event $\hat\Sigma_\rd(x) = 0$,
  we apply this result to
  \begin{align*}
    \hat\tau =
    \sqrt{\frac{n}{\lambda^d}}
    \left(
      \frac{\hat \mu_\rd(x) - \mu(x)}
      {\sqrt{\smash[b]{\hat \Sigma_\rd(x)}}}
      - \frac{\E\left[ \hat \mu_\rd(x) \right] - \mu(x)}
      {\sqrt{\Sigma_\rd(x)}}
    \right)
    &&\text{and}
    && \tau = \sqrt{\frac{n}{\lambda^d}}
    \frac{\hat \mu_\rd(x) - \E\left[ \hat \mu_\rd(x) \mid \bX, \bT \right]}
    {\sqrt{\smash[b]{\tilde \Sigma_\rd(x)}}}
  \end{align*}
  respectively, noting that
  $\sup_{t \in \R} \left| \P \big( \tau \leq t \big) - \Phi(t) \right|
  \lesssim \left( \frac{\lambda^d}{n} \right)^{\frac{1 \wedge \delta}{2}}
  + \frac{1}{B}$ by Theorem~\ref{thm:clt_debiased}.
  With
  \begin{align*}
    v &=
    \sqrt{\frac{n}{\lambda^d}}
    \frac{\E\left[ \hat \mu_\rd(x) \right] - \mu(x)}{\sqrt{\Sigma_\rd(x)}}
    \lesssim
    \sqrt{\frac{\lambda^d}{n}}
    \frac{1}{\lambda^{1 \wedge \betamu}}
    + \sqrt{\frac{n}{\lambda^d}}
    \frac{1}{\lambda^\beta}
  \end{align*}
  by the proof of Lemma~\ref{lem:bias},
  and by Taylor's theorem, for some $s, s' \in \R$, we have
  \begin{align*}
    &\left|
    \P \big( \mu(x) \in \CI_\rd(x) \big)
    - (1 - \alpha)
    \right| \\
    &\quad=
    \left|
    \P \left(
      q_{\alpha / 2}
      \leq
      \sqrt{\frac{n}{\lambda^d}}
      \frac{\hat \mu_\rd(x) - \mu(x)}
      {\sqrt{\smash[b]{\hat \Sigma_\rd(x)}}}
      \leq
      q_{1 - \alpha / 2}
    \right)
    - (1 - \alpha)
    \right| \\
    &\quad=
    \left|
    \P \left(
      q_{\alpha / 2} - v
      \leq
      \hat\tau
      \leq
      q_{1 - \alpha / 2} - v
    \right)
    - (1 - \alpha)
    \right| \\
    &\quad\leq
    \left|
    \P \left(
      \hat\tau
      \leq
      q_{1 - \alpha / 2} - v
    \right)
    - \Phi(q_{1 - \alpha / 2} - v)
    \right|
    + \left|
    \P \left(
      \hat\tau
      < q_{\alpha / 2} - v
    \right)
    - \Phi(q_{\alpha / 2} - v)
    \right| \\
    &\qquad+
    \left|
    \Phi(q_{1 - \alpha / 2} - v)
    - (1 - \alpha / 2)
    - \Phi(q_{\alpha / 2} - v)
    + \alpha / 2
    \right| \\
    &\quad\lesssim
    \left( \frac{\lambda^d}{n} \right)^{\frac{1 \wedge \delta}{2}}
    + \frac{1}{B}
    + \varepsilon + \P \left(|\hat\tau - \tau| > \varepsilon \right)
    + \left|
    - v \phi(1 - \alpha/2) + v^2 \phi'(s) / 2
    + v \phi(\alpha/2) - v^2 \phi'(s') / 2
    \right| \\
    &\quad\lesssim
    \left( \frac{\lambda^d}{n} \right)^{\frac{1 \wedge \delta}{2}}
    + \frac{1}{B} + \varepsilon
    + \frac{n}{\lambda^d} \frac{1}{\lambda^{2\beta}}
    + \P \left(|\hat\tau - \tau| > \varepsilon \right).
  \end{align*}
  It remains to bound $\P \left(|\hat\tau - \tau| > \varepsilon \right)$.
  Observe that
  \begin{align*}
    |\hat\tau - \tau|
    &\leq
    R_1 + R_2 + R_3, \\
    R_1 &=
    \sqrt{\frac{n}{\lambda^d}}
    \big| \hat \mu_\rd(x) - \mu(x) \big|
    \left|
    \frac{1}{\sqrt{\smash[b]{\hat \Sigma_\rd(x)}}}
    - \frac{1}{\sqrt{\smash[b]{\tilde \Sigma_\rd(x)}}}
    \right|, \\
    R_2 &=
    \sqrt{\frac{n}{\lambda^d}}
    \big| \E\left[ \hat \mu_\rd(x) \mid \bX, \bT \right] - \mu(x) \big|
    \left|
    \frac{1}{\sqrt{\smash[b]{\tilde \Sigma_\rd(x)}}}
    - \frac{1}{\sqrt{{\Sigma_\rd(x)}}}
    \right|, \\
    R_3 &=
    \sqrt{\frac{n}{\lambda^d}}
    \frac{ \big| \E\left[ \hat \mu_\rd(x) \mid \bX, \bT \right]
    - \E\left[ \hat \mu_\rd(x) \right]\big|}
    {\sqrt{\smash[b]{\Sigma_\rd(x)}}}.
  \end{align*}
  We begin with $R_1$.
  Take $a > 1$ and $b^{2/3} = \Sigma_\rd(x)/2$, so
  by the proof of Lemma~\ref{lem:variance_debiased},
  \begin{align*}
    \P \left( R_1 > \varepsilon \right)
    &\leq
    \P \left(
      \sqrt{\frac{n}{\lambda^d}}
      \big|
      \hat \mu_\rd(x) - \mu(x)
      \big|
      > a \varepsilon
    \right)
    + \P \left(
      \left|
      \frac{1}{\sqrt{\smash[b]{\hat \Sigma_\rd(x)}}}
      - \frac{1}{\sqrt{\smash[b]{\tilde \Sigma_\rd(x)}}}
      \right|
      > \frac{1}{a}
    \right) \\
    &\leq
    \frac{n}{a^2 \varepsilon^2 \lambda^d}
    \E \big[
      (\hat \mu_\rd(x) - \mu(x))^2
    \big]
    + \P \left( \tilde \Sigma_\rd(x) < b^{2/3} \right)
    + \P \left( \hat \Sigma_\rd(x) < b^{2/3} \right) \\
    &\quad+
    \P \left(
      \left|
      \hat \Sigma_\rd(x) - \tilde \Sigma_\rd(x)
      \right|
      > \frac{b}{a}
    \right) \\
    &\lesssim
    \frac{n}{a^2 \varepsilon^2 \lambda^d}
    \left( \frac{\lambda^d}{n} + \frac{1}{\lambda^{2 \beta}}
    + \frac{1}{\lambda^{2(1 \wedge \beta)} B}\right)
    + \Var \left[ \tilde \Sigma_\rd(x) \right]
    + \P \left( |\hat \Sigma_\rd(x) - \Sigma_\rd(x)|
    > \Sigma_\rd(x)/2 \right) \\
    &\quad+ \P \left(
      \left|
      \tilde \Sigma_\rd(x) - \Sigma_\rd(x)
      \right|
      > \frac{b}{2a}
    \right)
    + \P \left(
      \left|
      \hat \Sigma_\rd(x) - \Sigma_\rd(x)
      \right|
      > \frac{b}{2a}
    \right) \\
    &\lesssim
    \frac{1}{a^2 \varepsilon^2}
    + \frac{1}{B} + \frac{\lambda^d}{n}
    + a^2
    \left(
      \frac{\lambda^d}{n} + \frac{1}{B}
      + \frac{1}{\lambda^{2(1 \wedge \betaf \wedge \betasigma)}}
    \right)
    + \P \left(
      \left|
      \hat \Sigma_\rd(x) - \Sigma_\rd(x)
      \right|
      > \frac{b}{2a}
    \right).
  \end{align*}
  If $\delta < 2$ then the proof of
  Lemma~\ref{lem:variance_estimation_debiased}
  along with Markov's inequality gives
  \begin{align*}
    \P \left(
      \left|
      \hat \Sigma_\rd(x) - \Sigma_\rd(x)
      \right|
      > \frac{b}{2a}
    \right)
    &\lesssim
    \sqrt{a} \,
    \E \left[
      \left|
      \hat\Sigma_\rd(x)
      - \Sigma_\rd(x)
      \right|^{1/2}
    \right] \\
    &\lesssim
    \sqrt a
    \left(
      \left(
        \frac{\lambda^d}{n}
      \right)^{\frac{\delta}{8 + 4 \delta}}
      + \frac{1}{B^{1/4}}
      + \frac{1}{\lambda^{(1 \wedge \betamu \wedge \betaf \wedge \betasigma)/2}}
    \right),
  \end{align*}
  and whenever this converges to zero, minimizing over $a$ yields
  \begin{align*}
    \P \left( R_1 > \varepsilon \right)
    &\lesssim
    \frac{1}{a^2 \varepsilon^2}
    + \frac{1}{B} + \frac{\lambda^d}{n}
    + \sqrt a
    \left(
      \left(
        \frac{\lambda^d}{n}
      \right)^{\frac{\delta}{8 + 4 \delta}}
      + \frac{1}{B^{1/4}}
      + \frac{1}{\lambda^{(1 \wedge \betamu \wedge \betaf \wedge \betasigma)/2}}
    \right) \\
    &\lesssim
    \frac{1}{\varepsilon^{2/5}}
    \left(
      \left(
        \frac{\lambda^d}{n}
      \right)^{\frac{\delta}{10 + 5 \delta}}
      + \frac{1}{B^{1/5}}
      + \frac{1}{\lambda^{2(1 \wedge \betamu \wedge \betaf \wedge
      \betasigma)/5}}
    \right).
  \end{align*}
  If $\delta \geq 2$ however then we instead obtain
  \begin{align*}
    \P \left(
      \left|
      \hat \Sigma_\rd(x) - \Sigma_\rd(x)
      \right|
      > \frac{b}{2a}
    \right)
    &\lesssim
    a \,
    \E \left[
      \left|
      \hat\Sigma_\rd(x)
      - \Sigma_\rd(x)
      \right|
    \right]
    \lesssim
    a \left(
      \sqrt{\frac{\lambda^d}{n}}
      + \frac{1}{\sqrt B}
      + \frac{1}
      {\lambda^{1 \wedge \betamu \wedge \betaf \wedge \betasigma}}
    \right),
  \end{align*}
  and again if this converges to zero then minimizing over $a$ gives
  \begin{align*}
    \P \left( R_1 > \varepsilon \right)
    &\lesssim
    \frac{1}{a^2 \varepsilon^2}
    + \frac{1}{B} + \frac{\lambda^d}{n}
    + a \left(
      \sqrt{\frac{\lambda^d}{n}}
      + \frac{1}{\sqrt B}
      + \frac{1}
      {\lambda^{1 \wedge \betamu \wedge \betaf \wedge \betasigma}}
    \right) \\
    &\lesssim
    \frac{1}{\varepsilon^{2/3}}
    \left(
      \left(\frac{\lambda^d}{n}\right)^{1/3}
      + \frac{1}{B^{1/3}}
      + \frac{1}
      {\lambda^{2(1 \wedge \betamu \wedge \betaf \wedge \betasigma)/3}}
    \right).
  \end{align*}
  For $R_2$, note that the same arguments used for $R_1$ apply again,
  yielding a bound no worse than that for $R_1$.
  Finally, for $R_3$, we have by Lemma~\ref{lem:bias_debiased} that
  \begin{align*}
    \P \left( R_3 > \varepsilon \right)
    &\leq
    \frac{1}{\varepsilon^2}
    \E \left[
      R_3^2
    \right]
    \lesssim
    \frac{1}{\varepsilon^2}
    \frac{n}{\lambda^d}
    \E \left[
      \big(
        \E\left[ \hat \mu_\rd(x) \mid \bX, \bT \right]
        - \E\left[ \hat \mu_\rd(x) \right]
      \big)^2
    \right] \\
    &\lesssim
    \frac{1}{\varepsilon^2}
    \frac{n}{\lambda^d}
    \left(
      \frac{1}{\lambda^{2(1 \wedge \betamu)} B}
      + \frac{1}{\lambda^{2(1 \wedge \betamu)}}
      \frac{\lambda^d}{n}
    \right)
    \lesssim
    \frac{1}{\varepsilon^2}
    \left(
      \frac{n}{\lambda^d}
      \frac{1}{\lambda^{2(1 \wedge \betamu)} B}
      + \frac{1}{\lambda^{2(1 \wedge \betamu)}}
    \right).
  \end{align*}
  So far we have shown that if $\varepsilon \to 0$,
  then for $\delta < 2$,
  \begin{align*}
    &\left| \P \big( \mu(x) \in \CI_\rd(x) \big) - (1 - \alpha) \right|
    \lesssim
    \left( \frac{\lambda^d}{n} \right)^{\frac{1 \wedge \delta}{2}}
    + \frac{1}{B} + \varepsilon
    + \frac{n}{\lambda^d} \frac{1}{\lambda^{2\beta}}
    + \P \left(|\hat\tau - \tau| > \varepsilon \right) \\
    &\quad\lesssim
    \varepsilon
    + \frac{n}{\lambda^d} \frac{1}{\lambda^{2\beta}}
    + \frac{1}{\varepsilon^{2/5}}
    \Bigg(
      \left(
        \frac{\lambda^d}{n}
      \right)^{\frac{\delta}{10 + 5 \delta}}
      + \frac{1}{B^{1/5}}
      + \frac{1}
      {\lambda^{2(1 \wedge \betamu \wedge \betaf \wedge \betasigma)/5}}
      + \left(
        \frac{n}{\lambda^d}
      \right)^{1/5}
      \frac{1}{\lambda^{2(1 \wedge \betamu)/5} B^{1/5}}
    \Bigg) \\
    &\quad\lesssim
    \frac{n}{\lambda^d} \frac{1}{\lambda^{2\beta}}
    + \left(
      \left(
        \frac{\lambda^d}{n}
      \right)^{\frac{\delta}{2 + \delta}}
      + \frac{1}{B}
      + \frac{1}
      {\lambda^{2(1 \wedge \betamu \wedge \betaf \wedge \betasigma)}}
      + \frac{n}{\lambda^d}
      \frac{1}{\lambda^{2(1 \wedge \beta)} B}
    \right)^{1/7},
  \end{align*}
  while for $\delta \geq 2$,
  \begin{align*}
    &\left| \P \big( \mu(x) \in \CI_\rd(x) \big) - (1 - \alpha) \right|
    \lesssim
    \left( \frac{\lambda^d}{n} \right)^{\frac{1 \wedge \delta}{2}}
    + \frac{1}{B} + \varepsilon
    + \frac{n}{\lambda^d} \frac{1}{\lambda^{2\beta}}
    + \P \left(|\hat\tau - \tau| > \varepsilon \right) \\
    &\quad\lesssim
    \varepsilon
    + \frac{n}{\lambda^d} \frac{1}{\lambda^{2\beta}}
    + \frac{1}{\varepsilon^{2/3}}
    \left(
      \left(\frac{\lambda^d}{n}\right)^{1/3}
      + \frac{1}{B^{1/3}}
      + \frac{1}
      {\lambda^{2(1 \wedge \betamu \wedge \betaf \wedge \betasigma)/3}}
      + \left(
        \frac{n}{\lambda^d}
      \right)^{1/3}
      \frac{1}{\lambda^{2(1 \wedge \betamu)/3} B^{1/3}}
    \right) \\
    &\quad\lesssim
    \frac{n}{\lambda^d} \frac{1}{\lambda^{2\beta}}
    + \left(
      \frac{\lambda^d}{n}
      + \frac{1}{B}
      + \frac{1}
      {\lambda^{2(1 \wedge \betamu \wedge \betaf \wedge \betasigma)}}
      + \frac{n}{\lambda^d}
      \frac{1}{\lambda^{2(1 \wedge \betamu)} B}
    \right)^{1/5},
  \end{align*}
  where in both displays we minimized over $\varepsilon > 0$.
\end{proof}

\subsection{Proofs for Section~\ref{sec:compute}}
\label{sec:proofs_compute}

\begin{proof}[Lemma~\ref{lem:compute_batch}]

  All computational complexities in this proof are understood to be upper bounds
  up to constants.
  The first step is to select $\lambda$ using polynomial fitting
  as in Section~\ref{sec:parameter_selection}.
  Constructing the design matrix $\bP$
  requires raising a number to a power of at most $J+1$
  a total of $n d (J+1)$ times, giving a complexity of $n d (J+1)^2$.
  Multiplying the design matrix to obtain $\bP^\T \bP$ is
  $n d^2 (J + 1)^2$, and inverting this is
  $d^3 (J + 1)^3 \lesssim n d^2 (J + 1)^2$,
  giving an overall complexity of $n d^2 (J + 1)^2$ for selecting the lifetime.

  Calculating the debiasing coefficients $\omega_r$ as in
  Section~\ref{sec:debiased}
  involves inverting a $(J+1) \times (J+1)$ matrix, so is
  $(J+1)^3 \lesssim n (J+1)^2$.
  Next, constructing $U(x)$ as in \eqref{eq:compute_local}
  requires $B d (J+1)$ comparisons,
  and forming $I(x)$ then needs $n d$ comparisons.

  Once $I(x)$ is available,
  Calculating $N_{b r}(x)$, $S_{b r}(x)$ and $V_{b r}(x)$
  as in \eqref{eq:compute_sums}
  each take $B d (J+1) |I(x)|$ operations,
  and from these we compute
  $\hat\mu(x)$ and $\hat\sigma^2(x)$ in $(J+1) B$ using
  \eqref{eq:compute_mu_sigma2}.
  Constructing $\hat\Sigma_\rd(x)$ as in \eqref{eq:compute_Sigma_hat}
  is $B d (J+1) |I(x)|$, and calculating $\CI_\rd(x)$ with \eqref{eq:CI_d}
  has complexity $1$.

  Thus the overall complexity of Algorithm~\ref{alg:batch} is
  $n d^2 (J+1)^2 + B d (J+1) + B d (J+1) |I(x)|$.
  To obtain the average case behavior we present a bound for
  $\E \left[ |I(x)| \right]$. Firstly, since $f(x)$ is bounded
  and by the distribution of Mondrian cells,
  \begin{align*}
    \E \left[ |I(x)| \right]
    &=
    \E \left[
      \sum_{i=1}^{n} \I\{X_i \in U(x)\}
    \right]
    =
    \sum_{i=1}^{n}
    \E \left[
      \P\big(X_i \in U(x) \mid U(x)\big)
    \right]
    \lesssim
    n \, \E \left[
      |U(x)|
    \right]
    \lesssim
    n \, \E \left[
      |U(x)_j|
    \right]^d.
  \end{align*}
  Next, by Lemma~\ref{lem:largest_cell}, we have that
  \begin{align*}
    \P \left(
      |U(x)_j| \geq \frac{4 t + 4 \log (2 B (J+1))}{\lambda}
    \right)
    &\leq
    \P \left(
      \max_{0 \leq r \leq J}
      \max_{1 \leq b \leq B}
      |T_b(x)_j|
      \geq \frac{2 t + 2 \log (2 B (J+1))}{\lambda}
    \right)
    \leq
    e^{-t},
  \end{align*}
  and integrating the tail probability yields
  \begin{align*}
    \E \left[
      |U(x)_j|
    \right]
    &\lesssim \frac{\log (2 B (J+1))}{\lambda},
    &&\text{so that}
    &\E \left[ |I(x)| \right]
    &\lesssim
    \frac{n \log (2 B (J+1))^d}{\lambda^d}.
  \end{align*}
\end{proof}

\begin{proof}[Lemma~\ref{lem:compute_online}]

  As in the proof of Lemma~\ref{lem:compute_batch},
  the complexity of selecting the lifetime is $(n + k) d^2 (J + 1)^2$.
  Since this occurs with probability at most $k / K$, and $k \leq n$,
  the average case time complexity is $\frac{k n d^2 (J + 1)^2}{K}$.

  To update the trees, we sample and perform comparisons with at most
  $B^* d (J+1) \lesssim B d (J+1)$ exponential random variables.
  We verify here that the resulting trees have the correct distribution,
  since by \citet[Proposition~1]{mourtada2020minimax} and the memoryless
  property of the exponential distribution, with
  $E_{b r j 1}'$ and $E_{b r j 1}'$ i.i.d.\ copies of $E_{b r j 1}$,
  \begin{align*}
    T_{b r}^*(x)_j^-
    &= T_{b r}(x)_j^- \vee
    \left(x_j - \frac{E_{b r j 1}}{\lambda^* - \lambda}\right)
    = 0 \vee
    \bigg(x_j - \frac{E_{b r j 1}'}{\lambda}\bigg)
    \vee
    \bigg(x_j - \frac{E_{b r j 1}}{\lambda^* - \lambda}\bigg) \\
    &= x_j - \left(
      x_j \wedge
      \frac{E_{b r j 1}'}{\lambda}
      \wedge
      \frac{E_{b r j 1}}{\lambda^* - \lambda}
    \right)
    = x_j - \left(
      x_j \wedge
      \frac{E_{b r j 1}''}{\lambda^*}
    \right),
  \end{align*}
  as required. The same argument applies to $T_{b r}^*(x)_j^+$. We also bound
  the expected number of trees which have changed. By a union bound and with
  $E_1$ and $E_2$ i.i.d.\ $\Exp(1)$,
  \begin{align*}
    &\E \left[
      \sum_{r=0}^{J}
      \sum_{b=1}^{B}
      \I \big\{ T_{b r}^*(x) \neq T_{b r}(x) \big\}
    \right]
    = B (J+1) \, \P \big(T_{b r}^*(x) \neq T_{b r}(x)\big) \\
    &\quad\leq 2 d B (J+1) \,
    \P \left( \frac{E_1}{\lambda^* - \lambda} < \frac{E_2}{\lambda} \right)
    \leq 2 d B (J+1) \frac{\lambda^* - \lambda}{\lambda^*}
    \lesssim \frac{B (J+1) k}{n},
  \end{align*}
  since $\frac{\lambda^* - \lambda}{\lambda}
  = \left( \frac{n+k}{n} \right)^\zeta - 1
  \leq \frac{k \zeta}{n} \leq \frac{k}{n d}$.
  Constructing $U(x)$ requires $B d (J+1)$ comparisons, and since
  for $b \leq B$ we have $T_{b r}^*(x) \subseteq T_{b r}(x)$,
  \begin{align*}
    \P \left( U^*(x) \nsubseteq U(x) \right)
    &\leq
    d \, \P \left( U^*(x)_j^- < U(x)_j^- \right)
    + d \, \P \left( U^*(x)_j^+ > U(x)_j^+ \right) \\
    &\leq
    2 d \, \P \left(
      \max_{B+1 \leq b \leq B^*} \max_{0 \leq r \leq J} T_{b r}^*(x)_j^+
      > \max_{1 \leq b \leq B} \max_{0 \leq r \leq J} T_{b r}^*(x)_j^+
    \right) \\
    &\leq
    2 (B^* - B) d (J+1) \, \P \left(
      T_{b r}^*(x)_j^+
      > \max_{1 \leq b \leq B} \max_{0 \leq r \leq J} T_{b r}^*(x)_j^+
    \right) \\
    &\lesssim \frac{(B^* - B) d (J+1)}{B}
    \lesssim \frac{k d (J+1)}{n},
  \end{align*}
  since $\frac{B^* - B}{B} \leq \left( \frac{n+k}{n} \right)^\xi - 1
  \leq \frac{k}{n}$.
  The average case complexity of calculating $I^*(x)$ is therefore
  \begin{align*}
    &\E \left[
      d |I(x)| + d k + d n \I\{U^*(x) \nsubseteq U(x)\}
    \right] \\
    &\quad\leq
    d \, \E \left[ |I(x)| \right] + d k
    + d n \, \P \left( U^*(x) \nsubseteq U(x) \right)
    \lesssim
    \frac{n d \log(2 B (J+1))^d}{\lambda^d}
    + k d^2 (J+1).
  \end{align*}
  A similar calculation shows the cost of calculating all of
  $N_{b r}^*(x)$, $S_{b r}^*(x)$, and $V_{b r}^*(x)$ is at most
  \begin{align*}
    &\sum_{b=1}^{B}
    \sum_{r=0}^{J}
    \E \left[
      \frac{d k}{n} |I^*(x)| + d |I^*(x)| \, \I\{ T_{b r}(x) \neq T^*_{b r}(x)\}
    \right]
    + \sum_{b=B+1}^{B^*}
    \sum_{r=0}^{J}
    \E \left[
      d |I^*(x)|
    \right] \\
    &\quad\lesssim
    \E \left[
      \frac{B (J+1) d k}{n} |I^*(x)|
      + d |I^*(x)|
      \sum_{b=1}^{B}
      \sum_{r=0}^{J}
      \I\{ T_{b r}(x) \neq T^*_{b r}(x)\}
      + d (B^* - B) (J+1)
      |I^*(x)|
    \right] \\
    &\quad\lesssim
    \frac{B k d (J+1) \log (2 B (J+1))^d}{\lambda^d}.
  \end{align*}
  where we used that the bounds for $|I(x)|$ and
  $\sum_{r=0}^{J} \sum_{b=1}^{B} \I \big\{ T_{b r}^*(x) \neq T_{b r}(x) \big\}$
  hold also in $L^2$ and applied the Cauchy--Schwarz inequality.

  Finally, updating $\hat\Sigma_\rd(x)$ takes
  $B d (J+1) |I(x)|$ computations,
  which is done with probability at most $k/K$,
  yielding a time complexity of
  $\frac{n B d (J+1) \log (2 B (J+1))^d}{K \lambda^d}$ on average.
  The overall average case time complexity is therefore bounded by
  \begin{align*}
    d (J + 1)
    \left( \frac{k n d (J + 1)}{K} + k d + B \right)
    + \frac{d (J+1) \log (2 B (J+1))^d}{\lambda^d}
    \left( n + B k + \frac{n B}{K} \right).
  \end{align*}
\end{proof}

\section{Additional empirical results}
\iftoggle{journal}{}{\label{sec:additional_empirical}}

Tables~\ref{tab:d1_n1000_B1},~\ref{tab:d2_n1000_B1},
\ref{tab:d1_n1000_B10} and~\ref{tab:d2_n1000_B10} present some additional
empirical results not given in the main paper. The data generating process is
identical to that in Section~\ref{sec:implementation}, and we demonstrate here
the effect of a smaller forest size $B$,
taking $B=1$ in Tables~\ref{tab:d1_n1000_B1} and~\ref{tab:d2_n1000_B1},
and $B=10$ in Tables~\ref{tab:d1_n1000_B10} and~\ref{tab:d2_n1000_B10}.
Note that the bias in
Table~\ref{tab:d1_n1000_B1} is not significantly larger than that in
Table~\ref{tab:d1_n1000_B800}, even though Lemma~\ref{lem:bias} suggests that
the bias should be much greater. This is because Lemma~\ref{lem:bias} is stated
for the \emph{conditional} bias. In fact, the repeated experiments
(we use $3000$ independent trials) have the
same effect as using a large forest in reducing the apparent bias of the
estimator. As such, the error incurred appears in the standard deviation column
instead; indeed the standard deviations in Table~\ref{tab:d1_n1000_B1} are
substantially higher than those in Table~\ref{tab:d1_n1000_B800}.

\begin{table}[p]
  \begin{center}
    \begin{scriptsize}
\begin{tabular}{|c|cc|cc|cccc|cc|ccc|cc|}
  \hline
  & $J$ & LS & LM & $\lambda$ & RMSE & Bias & SD & Bias/SD &
  $\SDhattab$ & $\hat\sigma^2$ & ARMSE &
  ABias & ASD & CR & CIW \\
  \hline
  \multirow{6}{*}{\rotatebox{90}{No debiasing}}&0& $\hat\lambda_{0}$& 1.0&
  14.72& 0.0606& -0.0235& 0.0558& 0.4211& 0.0322& 0.0907& 0.0358& -0.0259&
  0.0232& 77.1\%& 0.126\\
  && $\lambda_{0}$& 1.2& 23.10& 0.0526& -0.0093& 0.0518& 0.1792& 0.0397&
  0.0883& 0.0306& -0.0092& 0.0292& 88.9\%& 0.156\\
  &&& 1.1& 21.18& 0.0565& -0.0102& 0.0555& 0.1828& 0.0381& 0.0886& 0.0300&
  -0.0110& 0.0279& 87.2\%& 0.149\\
  &&& 1.0& 19.25& 0.0551& -0.0150& 0.0530& 0.2834& 0.0361& 0.0890& 0.0298&
  -0.0133& 0.0266& 84.9\%& 0.142\\
  &&& 0.9& 17.33& 0.0504& -0.0157& 0.0479& 0.3268& 0.0345& 0.0895& 0.0301&
  -0.0164& 0.0253& 83.3\%& 0.135\\
  &&& 0.8& 15.40& 0.0568& -0.0202& 0.0530& 0.3814& 0.0323& 0.0902& 0.0316&
  -0.0208& 0.0238& 79.2\%& 0.127\\
  \hline
  \multirow{6}{*}{\rotatebox{90}{Debiasing}}&1& $\hat\lambda_{1}$& 1.0& 11.20&
  0.1150& -0.0058& 0.1149& 0.0508& 0.0633& 0.1184& 0.0302& -0.0031& 0.0287&
  83.0\%& 0.248\\
  && $\lambda_{1}$& 1.2& 7.86& 0.1274& -0.0046& 0.1274& 0.0362& 0.0568& 0.1389&
  0.0245& -0.0038& 0.0242& 71.7\%& 0.223\\
  &&& 1.1& 7.21& 0.1360& -0.0097& 0.1357& 0.0719& 0.0546& 0.1471& 0.0238&
  -0.0053& 0.0232& 66.1\%& 0.214\\
  &&& 1.0& 6.55& 0.1465& -0.0147& 0.1457& 0.1008& 0.0531& 0.1551& 0.0235&
  -0.0078& 0.0221& 63.5\%& 0.208\\
  &&& 0.9& 5.90& 0.1595& -0.0150& 0.1587& 0.0946& 0.0543& 0.1692& 0.0241&
  -0.0119& 0.0210& 60.1\%& 0.213\\
  &&& 0.8& 5.24& 0.1799& -0.0256& 0.1781& 0.1439& 0.0531& 0.1862& 0.0275&
  -0.0191& 0.0198& 54.0\%& 0.208\\
  \hline
  \multirow{6}{*}{\rotatebox{90}{Robust BC}}&1& $\hat\lambda_{0}$& 1.0& 14.72&
  0.1156& -0.0004& 0.1156& 0.0036& 0.0702& 0.1111& 0.0332& -0.0006& 0.0330&
  90.1\%& 0.275\\
  && $\lambda_{0}$& 1.2& 23.10& 0.1187& -0.0042& 0.1186& 0.0353& 0.0855&
  0.1044& 0.0415& -0.0001& 0.0415& 96.2\%& 0.335\\
  &&& 1.1& 21.18& 0.1093& -0.0011& 0.1093& 0.0101& 0.0842& 0.1025& 0.0398&
  -0.0001& 0.0398& 95.5\%& 0.330\\
  &&& 1.0& 19.25& 0.1006& -0.0026& 0.1005& 0.0259& 0.0808& 0.1020& 0.0379&
  -0.0001& 0.0379& 94.4\%& 0.317\\
  &&& 0.9& 17.33& 0.0873& 0.0007& 0.0873& 0.0077& 0.0751& 0.1012& 0.0360&
  -0.0002& 0.0360& 93.4\%& 0.294\\
  &&& 0.8& 15.40& 0.1032& -0.0018& 0.1032& 0.0174& 0.0710& 0.1050& 0.0339&
  -0.0003& 0.0339& 92.3\%& 0.278\\
  \hline
\end{tabular}
     \end{scriptsize}
  \end{center}
  \caption{Simulation results with
  $d=1$, $n=1000$, and $B=1$, over $3000$ repeats}
  \label{tab:d1_n1000_B1}
\end{table}

\begin{table}[p]
  \begin{center}
    \begin{scriptsize}
\begin{tabular}{|c|cc|cc|cccc|cc|ccc|cc|}
  \hline
  & $J$ & LS & LM & $\lambda$ & RMSE & Bias & SD & Bias/SD &
  $\SDhattab$ & $\hat\sigma^2$ & ARMSE &
  ABias & ASD & CR & CIW \\
  \hline
  \multirow{6}{*}{\rotatebox{90}{No debiasing}}&0& $\hat\lambda_{0}$& 1.0&
  12.32& 0.3192& -0.1082& 0.3003& 0.3603& 0.0709& 0.0846& 0.0831& -0.0670&
  0.0478& 64.2\%& 0.278\\
  && $\lambda_{0}$& 1.2& 18.39& 0.5522& -0.1706& 0.5252& 0.3248& 0.0788&
  0.0658& 0.0771& -0.0292& 0.0714& 67.1\%& 0.309\\
  &&& 1.1& 16.85& 0.4737& -0.1344& 0.4542& 0.2960& 0.0780& 0.0707& 0.0741&
  -0.0347& 0.0654& 70.3\%& 0.306\\
  &&& 1.0& 15.32& 0.4220& -0.1178& 0.4052& 0.2908& 0.0764& 0.0743& 0.0728&
  -0.0420& 0.0595& 69.6\%& 0.299\\
  &&& 0.9& 13.79& 0.3473& -0.1027& 0.3318& 0.3096& 0.0746& 0.0796& 0.0746&
  -0.0519& 0.0535& 70.4\%& 0.293\\
  &&& 0.8& 12.26& 0.3137& -0.1021& 0.2966& 0.3442& 0.0702& 0.0835& 0.0811&
  -0.0657& 0.0476& 65.5\%& 0.275\\
  \hline
  \multirow{6}{*}{\rotatebox{90}{Debiasing}}&1& $\hat\lambda_{1}$& 1.0& 9.22&
  0.7451& -0.1351& 0.7328& 0.1843& 0.2742& 0.6519& 0.0731& -0.0093& 0.0692&
  85.5\%& 1.075\\
  && $\lambda_{1}$& 1.2& 7.18& 0.5078& -0.0781& 0.5018& 0.1556& 0.2272& 0.4124&
  0.0550& -0.0108& 0.0540& 81.7\%& 0.891\\
  &&& 1.1& 6.58& 0.4920& -0.0686& 0.4872& 0.1408& 0.2227& 0.4197& 0.0518&
  -0.0154& 0.0495& 79.7\%& 0.873\\
  &&& 1.0& 5.99& 0.4391& -0.0700& 0.4335& 0.1616& 0.2083& 0.3817& 0.0503&
  -0.0225& 0.0450& 74.9\%& 0.816\\
  &&& 0.9& 5.39& 0.4185& -0.0725& 0.4121& 0.1759& 0.1954& 0.3862& 0.0530&
  -0.0343& 0.0405& 72.5\%& 0.766\\
  &&& 0.8& 4.79& 0.3645& -0.0789& 0.3559& 0.2217& 0.1766& 0.3631& 0.0656&
  -0.0549& 0.0360& 68.5\%& 0.692\\
  \hline
  \multirow{6}{*}{\rotatebox{90}{Robust BC}}&1& $\hat\lambda_{0}$& 1.0& 12.32&
  0.9612& -0.1957& 0.9411& 0.2079& 0.3495& 0.9536& 0.0926& -0.0014& 0.0926&
  88.3\%& 1.370\\
  && $\lambda_{0}$& 1.2& 18.39& 1.3984& -0.4142& 1.3357& 0.3101& 0.5041&
  1.8378& 0.1381& -0.0003& 0.1381& 80.1\%& 1.976\\
  &&& 1.1& 16.85& 1.2900& -0.3571& 1.2396& 0.2881& 0.4627& 1.5834& 0.1266&
  -0.0004& 0.1266& 83.3\%& 1.814\\
  &&& 1.0& 15.32& 1.1782& -0.3052& 1.1379& 0.2682& 0.4299& 1.3422& 0.1151&
  -0.0005& 0.1151& 85.4\%& 1.685\\
  &&& 0.9& 13.79& 1.0995& -0.2733& 1.0650& 0.2567& 0.3901& 1.2125& 0.1036&
  -0.0008& 0.1036& 87.7\%& 1.529\\
  &&& 0.8& 12.26& 0.9288& -0.1874& 0.9097& 0.2060& 0.3361& 0.8958& 0.0921&
  -0.0013& 0.0921& 89.3\%& 1.318\\
  \hline
\end{tabular}
     \end{scriptsize}
  \end{center}
  \caption{Simulation results with
  $d=2$, $n=1000$, and $B=1$, over $3000$ repeats}
  \label{tab:d2_n1000_B1}
\end{table}

\begin{table}[p]
  \begin{center}
    \begin{scriptsize}
\begin{tabular}{|c|cc|cc|cccc|cc|ccc|cc|}
  \hline
  & $J$ & LS & LM & $\lambda$ & RMSE & Bias & SD & Bias/SD &
  $\SDhattab$ & $\hat\sigma^2$ & ARMSE &
  ABias & ASD & CR & CIW \\
  \hline
  \multirow{6}{*}{\rotatebox{90}{No debiasing}}&0& $\hat\lambda_{0}$& 1.0&
  14.72& 0.0376& -0.0241& 0.0289& 0.8324& 0.0250& 0.0931& 0.0363& -0.0263&
  0.0232& 79.9\%& 0.098\\
  && $\lambda_{0}$& 1.2& 23.10& 0.0325& -0.0089& 0.0313& 0.2837& 0.0310&
  0.0895& 0.0306& -0.0092& 0.0292& 93.8\%& 0.122\\
  &&& 1.1& 21.18& 0.0317& -0.0104& 0.0300& 0.3462& 0.0296& 0.0898& 0.0300&
  -0.0110& 0.0279& 92.6\%& 0.116\\
  &&& 1.0& 19.25& 0.0320& -0.0126& 0.0293& 0.4306& 0.0285& 0.0903& 0.0298&
  -0.0133& 0.0266& 91.5\%& 0.112\\
  &&& 0.9& 17.33& 0.0324& -0.0153& 0.0286& 0.5368& 0.0270& 0.0907& 0.0301&
  -0.0164& 0.0253& 89.0\%& 0.106\\
  &&& 0.8& 15.40& 0.0334& -0.0200& 0.0268& 0.7464& 0.0255& 0.0917& 0.0316&
  -0.0208& 0.0238& 85.6\%& 0.100\\
  \hline
  \multirow{6}{*}{\rotatebox{90}{Debiasing}}&1& $\hat\lambda_{1}$& 1.0& 11.05&
  0.0434& -0.0029& 0.0433& 0.0678& 0.0350& 0.1023& 0.0297& -0.0027& 0.0285&
  89.0\%& 0.137\\
  && $\lambda_{1}$& 1.2& 7.86& 0.0458& -0.0062& 0.0454& 0.1376& 0.0316& 0.1138&
  0.0245& -0.0038& 0.0242& 82.0\%& 0.124\\
  &&& 1.1& 7.21& 0.0484& -0.0088& 0.0476& 0.1843& 0.0307& 0.1181& 0.0238&
  -0.0053& 0.0232& 78.2\%& 0.120\\
  &&& 1.0& 6.55& 0.0511& -0.0119& 0.0497& 0.2404& 0.0301& 0.1239& 0.0235&
  -0.0078& 0.0221& 75.3\%& 0.118\\
  &&& 0.9& 5.90& 0.0572& -0.0187& 0.0541& 0.3466& 0.0291& 0.1299& 0.0241&
  -0.0119& 0.0210& 68.5\%& 0.114\\
  &&& 0.8& 5.24& 0.0641& -0.0249& 0.0591& 0.4209& 0.0284& 0.1389& 0.0275&
  -0.0191& 0.0198& 63.0\%& 0.111\\
  \hline
  \multirow{6}{*}{\rotatebox{90}{Robust BC}}&1& $\hat\lambda_{0}$& 1.0& 14.72&
  0.0419& -0.0010& 0.0419& 0.0229& 0.0393& 0.0951& 0.0334& -0.0009& 0.0330&
  93.7\%& 0.154\\
  && $\lambda_{0}$& 1.2& 23.10& 0.0500& -0.0001& 0.0500& 0.0014& 0.0483&
  0.0910& 0.0415& -0.0001& 0.0415& 95.4\%& 0.189\\
  &&& 1.1& 21.18& 0.0483& 0.0009& 0.0483& 0.0193& 0.0468& 0.0912& 0.0398&
  -0.0001& 0.0398& 95.4\%& 0.183\\
  &&& 1.0& 19.25& 0.0460& 0.0008& 0.0460& 0.0164& 0.0445& 0.0917& 0.0379&
  -0.0001& 0.0379& 95.2\%& 0.174\\
  &&& 0.9& 17.33& 0.0439& 0.0006& 0.0439& 0.0134& 0.0424& 0.0923& 0.0360&
  -0.0002& 0.0360& 94.7\%& 0.166\\
  &&& 0.8& 15.40& 0.0424& 0.0004& 0.0424& 0.0084& 0.0400& 0.0932& 0.0339&
  -0.0003& 0.0339& 93.9\%& 0.157\\
  \hline
\end{tabular}
     \end{scriptsize}
  \end{center}
  \caption{Simulation results with
  $d=1$, $n=1000$, and $B=10$, over $3000$ repeats}
  \label{tab:d1_n1000_B10}
\end{table}

\begin{table}[p]
  \begin{center}
    \begin{scriptsize}
\begin{tabular}{|c|cc|cc|cccc|cc|ccc|cc|}
  \hline
  & $J$ & LS & LM & $\lambda$ & RMSE & Bias & SD & Bias/SD &
  $\SDhattab$ & $\hat\sigma^2$ & ARMSE &
  ABias & ASD & CR & CIW \\
  \hline
  \multirow{6}{*}{\rotatebox{90}{No debiasing}}&0& $\hat\lambda_{0}$& 1.0&
  12.27& 0.0893& -0.0648& 0.0614& 1.0566& 0.0551& 0.0979& 0.0836& -0.0678&
  0.0476& 72.7\%& 0.216\\
  && $\lambda_{0}$& 1.2& 18.39& 0.0857& -0.0310& 0.0799& 0.3884& 0.0707&
  0.0866& 0.0771& -0.0292& 0.0714& 87.7\%& 0.277\\
  &&& 1.1& 16.85& 0.0826& -0.0346& 0.0750& 0.4619& 0.0675& 0.0881& 0.0741&
  -0.0347& 0.0654& 87.9\%& 0.265\\
  &&& 1.0& 15.32& 0.0819& -0.0435& 0.0694& 0.6267& 0.0635& 0.0906& 0.0728&
  -0.0420& 0.0595& 85.2\%& 0.249\\
  &&& 0.9& 13.79& 0.0815& -0.0505& 0.0640& 0.7891& 0.0595& 0.0934& 0.0746&
  -0.0519& 0.0535& 81.4\%& 0.233\\
  &&& 0.8& 12.26& 0.0863& -0.0626& 0.0593& 1.0557& 0.0553& 0.0972& 0.0811&
  -0.0657& 0.0476& 75.0\%& 0.217\\
  \hline
  \multirow{6}{*}{\rotatebox{90}{Debiasing}}&1& $\hat\lambda_{1}$& 1.0& 9.24&
  0.1034& -0.0138& 0.1024& 0.1348& 0.1017& 0.1356& 0.0723& -0.0082& 0.0694&
  92.4\%& 0.399\\
  && $\lambda_{1}$& 1.2& 7.18& 0.0959& -0.0221& 0.0933& 0.2368& 0.0925& 0.1577&
  0.0550& -0.0108& 0.0540& 88.8\%& 0.363\\
  &&& 1.1& 6.58& 0.0964& -0.0284& 0.0921& 0.3082& 0.0901& 0.1698& 0.0518&
  -0.0154& 0.0495& 87.1\%& 0.353\\
  &&& 1.0& 5.99& 0.1007& -0.0378& 0.0934& 0.4046& 0.0855& 0.1831& 0.0503&
  -0.0225& 0.0450& 83.3\%& 0.335\\
  &&& 0.9& 5.39& 0.1079& -0.0503& 0.0955& 0.5269& 0.0814& 0.2002& 0.0530&
  -0.0343& 0.0405& 77.4\%& 0.319\\
  &&& 0.8& 4.79& 0.1208& -0.0748& 0.0949& 0.7879& 0.0740& 0.2138& 0.0656&
  -0.0549& 0.0360& 68.3\%& 0.290\\
  \hline
  \multirow{6}{*}{\rotatebox{90}{Robust BC}}&1& $\hat\lambda_{0}$& 1.0& 12.27&
  0.1231& -0.0049& 0.1230& 0.0399& 0.1182& 0.1146& 0.0923& -0.0015& 0.0922&
  95.1\%& 0.463\\
  && $\lambda_{0}$& 1.2& 18.39& 0.1577& -0.0036& 0.1576& 0.0226& 0.1361&
  0.1058& 0.1381& -0.0003& 0.1381& 92.7\%& 0.534\\
  &&& 1.1& 16.85& 0.1491& -0.0021& 0.1491& 0.0139& 0.1338& 0.1065& 0.1266&
  -0.0004& 0.1266& 93.4\%& 0.525\\
  &&& 1.0& 15.32& 0.1404& -0.0012& 0.1404& 0.0088& 0.1287& 0.1071& 0.1151&
  -0.0005& 0.1151& 94.3\%& 0.504\\
  &&& 0.9& 13.79& 0.1294& -0.0018& 0.1294& 0.0138& 0.1231& 0.1093& 0.1036&
  -0.0008& 0.1036& 95.3\%& 0.482\\
  &&& 0.8& 12.26& 0.1179& -0.0028& 0.1178& 0.0240& 0.1172& 0.1130& 0.0921&
  -0.0013& 0.0921& 95.6\%& 0.459\\
  \hline
\end{tabular}
     \end{scriptsize}
  \end{center}
  \caption{Simulation results with
  $d=2$, $n=1000$, and $B=10$, over $3000$ repeats}
  \label{tab:d2_n1000_B10}
\end{table}
 
\pagebreak

\bibliographystyle{hapalike}
\bibliography{refs}

\end{document}